\documentclass[a4paper,11pt,twocolumn]{article}
\usepackage[english]{babel}
\usepackage[utf8]{inputenc}

\usepackage{graphicx,wrapfig}
\usepackage[margin=0.7in]{geometry}

\usepackage{caption}
\usepackage{float}
\usepackage{subcaption}
\usepackage{dsfont}
\usepackage{amsmath}
\usepackage{amssymb}
\usepackage{amsthm}
\usepackage{amsfonts}
\usepackage{mathtools}
\usepackage{algorithm}
\usepackage{algorithmic}
\usepackage{hyperref}
\usepackage{cleveref}
\usepackage{bm}
\usepackage{bbold}
\usepackage{stmaryrd}
\usepackage{xcolor}
\usepackage[export]{adjustbox}
\usepackage{stackengine}
\usepackage{enumerate}

\renewenvironment{abstract}{\bf\small {\em\ Abstract---}}{}

\newtheorem{proposition}{Proposition}

\newtheorem{theorem}{Theorem}
\newtheorem{definition}{Definition}
\newtheorem{corollary}{Corollary}
\newtheorem{lemma}{Lemma}

\DeclareMathOperator*{\argmin}{arg\,min}

\def\ran{\mathrm{ran}}
\def\dist{\mathrm{dist}}

\def\Jac{\mathrm{Jac}}
\def\Re{\mathrm{Re}}

\def\Id{\mathrm{Id}}
\def\md{\mathrm{md}}

\def\R{\mathbb{R}}
\def\Nbb{\mathbb{N}}
\def\Z{\mathbb{Z}}
\def\C{\mathbb{C}}

\def\Nc{\mathcal{N}}
\def\Rc{\mathcal{R}}

\def\f12{\frac{1}{2}}
\def\ind{\mathds{1}}

\def\eqdef{\stackrel{\mathrm{def}}{=}}
\def\E{\mathbb{E}}

\title{Spurious minimizers in non uniform Fourier sampling optimization}
\author{Alban Gossard\textsuperscript{1,2} \\ \footnotesize \href{mailto:alban.paul.gossard@gmail.com}{alban.paul.gossard@gmail.com}
\and Frédéric de Gournay\textsuperscript{1,3} \\ \footnotesize \href{mailto:degourna@insa-toulouse.fr}{degourna@insa-toulouse.fr}
\and Pierre Weiss\textsuperscript{1,2} \\ \footnotesize \href{mailto:pierre.armand.weiss@gmail.com}{pierre.armand.weiss@gmail.com}
\and \footnotesize \textsuperscript{1} Institut de Math\'ematiques de Toulouse; UMR5219; Universit\'e de Toulouse; CNRS
\and \footnotesize \textsuperscript{2} Université de Toulouse, F-31062 Toulouse Cedex 9, France
\and \footnotesize \textsuperscript{3} INSA, F-31077 Toulouse, France
}
\date{\today}

\begin{document}
\maketitle

\begin{abstract}
A recent trend in the signal/image processing literature is the optimization of Fourier sampling schemes for specific datasets of signals.
In this paper, we explain why choosing optimal non Cartesian Fourier sampling patterns is a difficult nonconvex problem by bringing to light two optimization issues.
The first one is the existence of a combinatorial number of spurious minimizers for a generic class of signals.
The second one is a vanishing gradient effect for the high frequencies.
We conclude the paper by showing how using large datasets can mitigate the first effect and illustrate experimentally the benefits of using stochastic gradient algorithms with a variable metric.
\end{abstract}

\begin{figure*}[h]
\centering
\begin{subfigure}[b]{0.325\textwidth}
	\includegraphics[width=\linewidth]{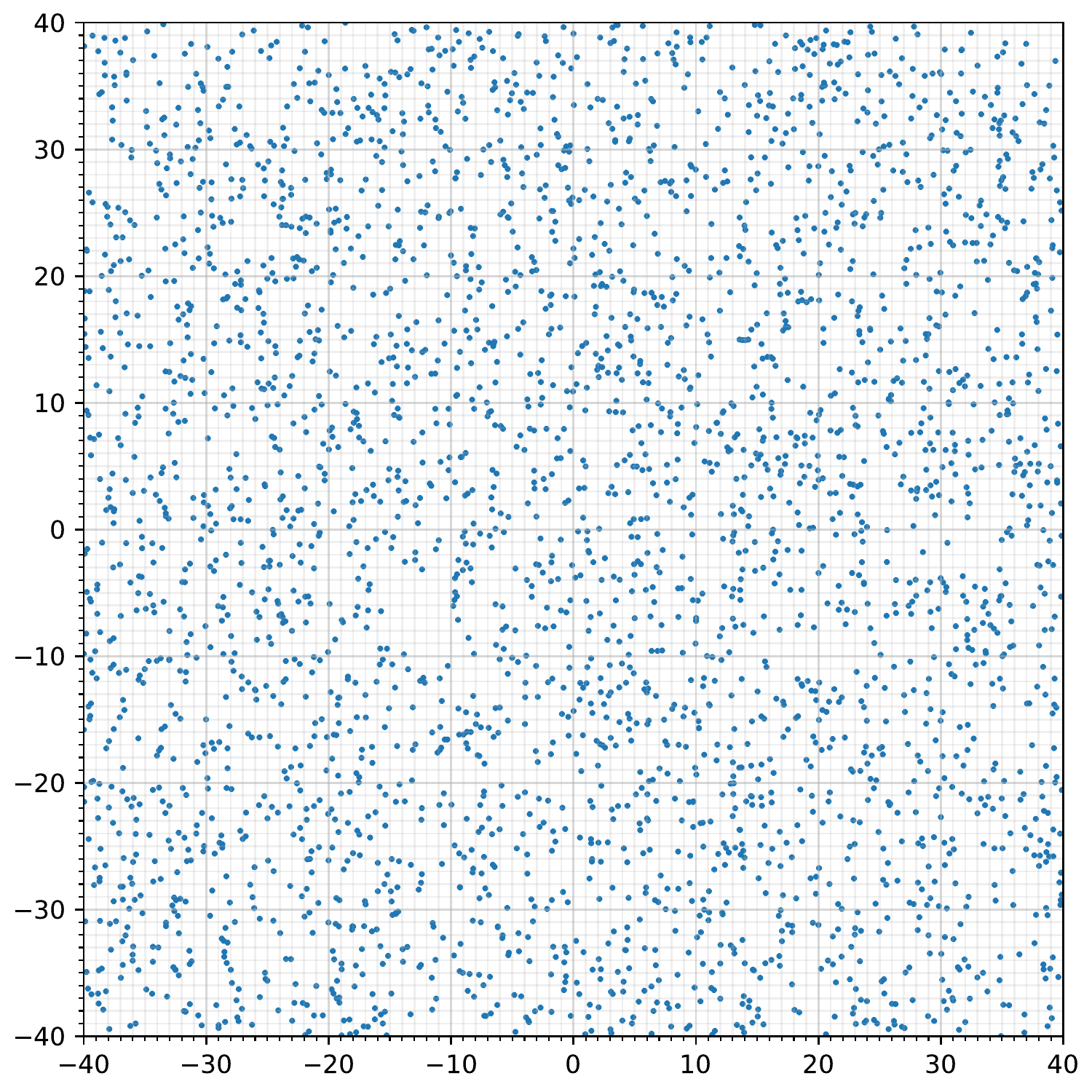}
	\caption{Initial}
\end{subfigure}
\begin{subfigure}[b]{0.325\textwidth}
	\includegraphics[width=\linewidth]{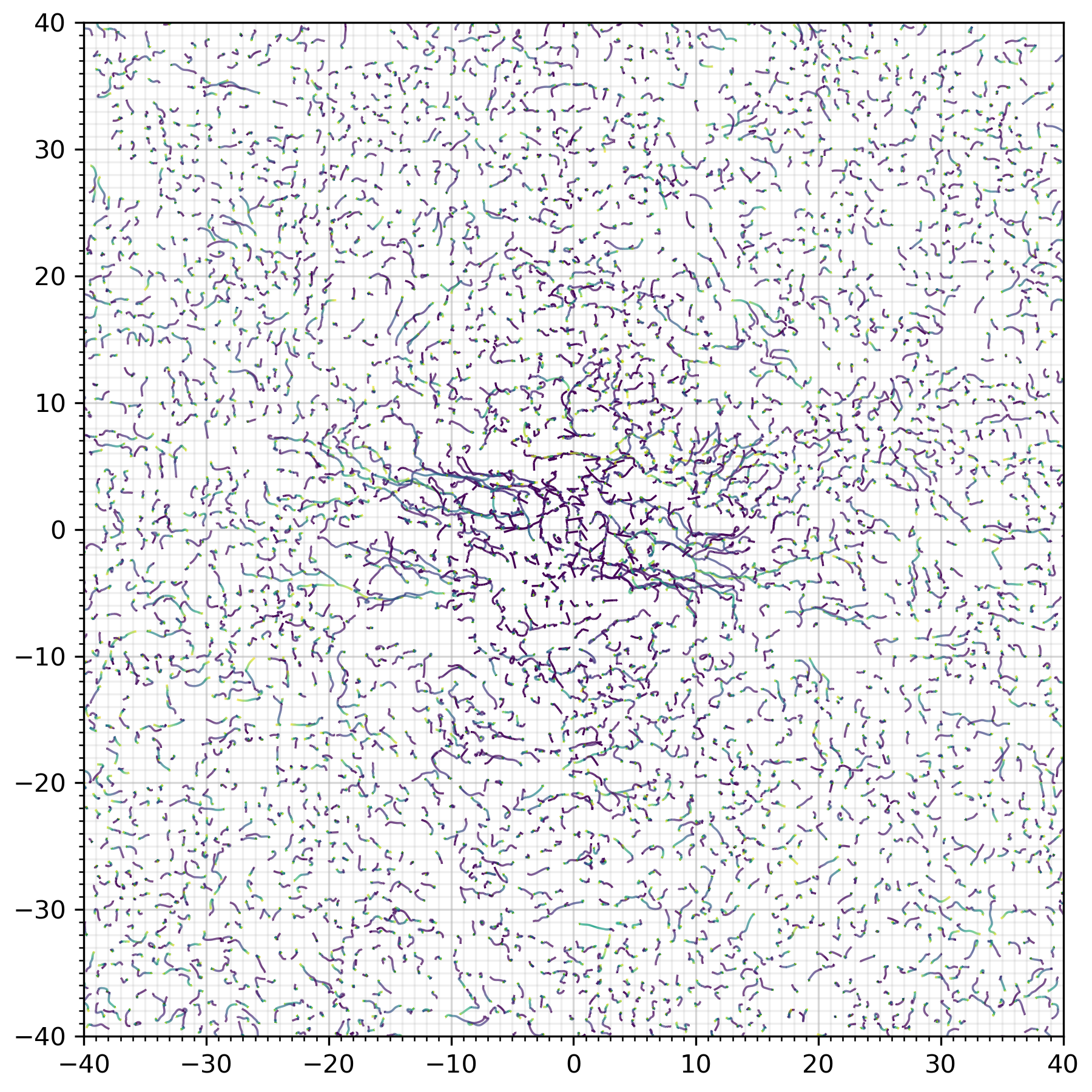}
	\caption{Trajectory over $10^7$ iterations \label{fig:traj_ini_2D}}
\end{subfigure}
\begin{subfigure}[b]{0.325\textwidth}
	\includegraphics[width=\linewidth]{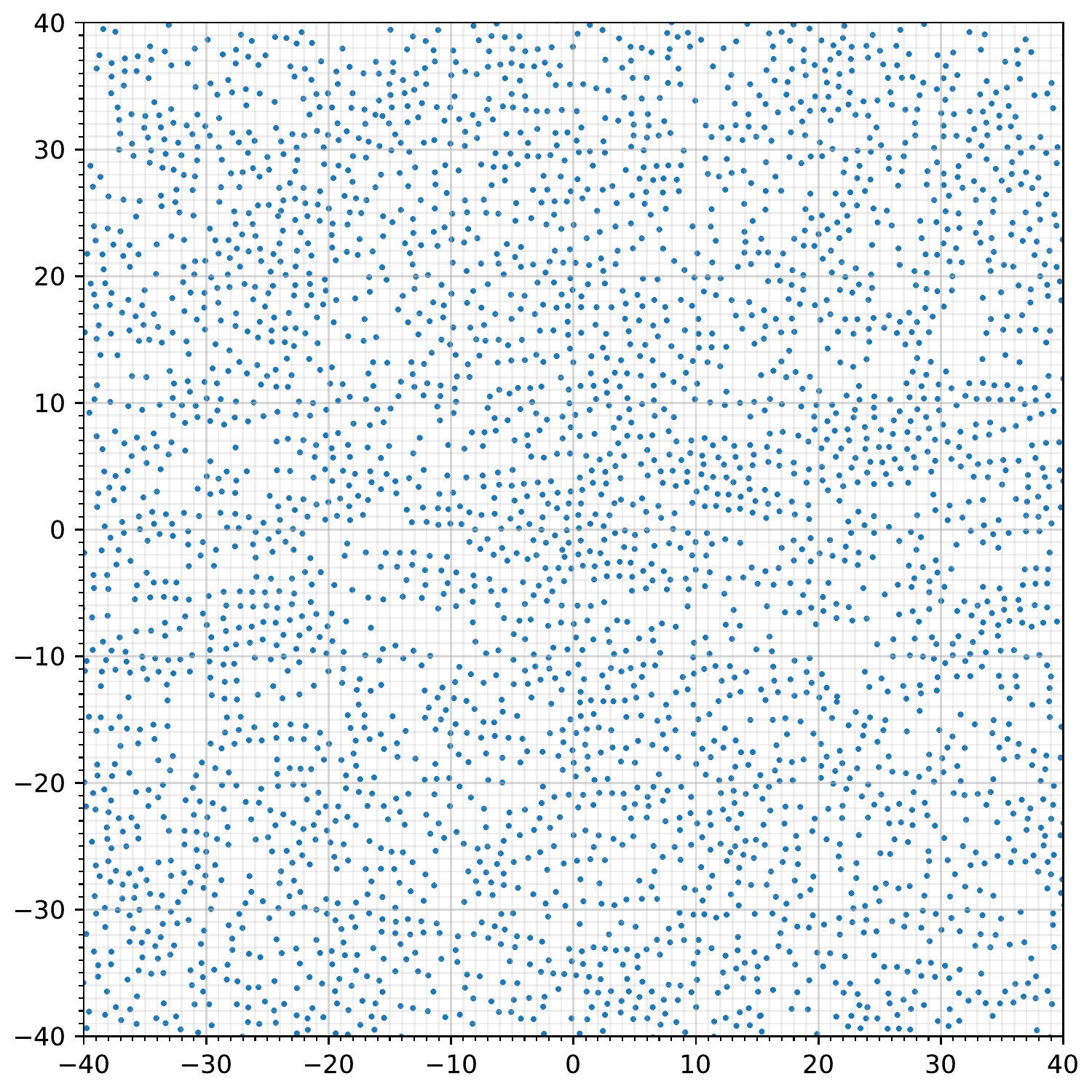}
	\caption{Final}
\end{subfigure}
\caption{A typical sampling optimization trajectory. Starting from the sampling configuration on the left (uniform point process), we obtain the sampling scheme on the right after $10^7$ iterations. The trajectory in the center corresponds to the $10^7$ iterations of a gradient descent with fixed step size. Notice that the points clusters have disappeared, but that the scheme is still essentially uniform, while we would expect the low frequencies to be sampled more densely.}
\label{fig:illustrations_2D}
\end{figure*}

\section{Introduction}

Finding efficient Fourier sampling schemes is a critical issue in communications and imaging. 
This led to various theories including the celebrated Shannon-Nyquist theorems for bandlimited signals and compressed sensing for sparse signals.
Unfortunately - in most practical cases - the signals to reconstruct are quite loosely described by these generic classes. 
For instance, magnetic resonance images of brains or knees have a rich structure due to the underlying object.
It is therefore tempting to optimize a sampling scheme directly for a given dataset rather than relying on a rough mathematical model.
The recent progresses in Graphical Processing Units (GPU) programming, automatic differentiation and machine learning make this idea even more tantalizing.
In the sole field of Magnetic Resonance Imaging (MRI), the following list of references \cite{gozcu2018learning,jin2019self,sherry2019learning,bahadir2019learning,zibetti2021fast,wang2021b,weiss2019pilot,wang2021b,gossard2020off,loktyushin2021mrzero,peng2022learning,aggarwal2020j} illustrates this novel trend.

Unfortunately, most of the above works report (more or less explicitly) optimization issues. Fig.~\ref{fig:illustrations_2D} illustrates one of them. In this example, we tried to optimize a sampling scheme for a single image from the fastMRI challenge \cite{zbontar2018fastmri}. To this end, we minimize the $\ell^2$ reconstruction error using a simple back-projection reconstructor with a subsampling factor of $2$. The trajectory of a gradient descent is displayed in Fig.~\ref{fig:traj_ini_2D}. As can be seen, the final sampling set covers approximately uniformly the Fourier domain, while we would expect the low frequencies to be sampled more densely. This likely highlights the presence of a spurious minimizer.

The aim of this paper is to explain this phenomenon from a mathematical perspective and to bring some solutions to mitigate the difficulties.
We focus on linear reconstruction methods, which simplifies the analysis and we highlight the critical role of the non uniform Fourier transform as an oscillation generator.
We expect that some of the arguments can be reused for more complex nonlinear reconstruction methods, which suffer from the same experimental issues.
We also focus on optimization schemes that continuously optimize the positions of some sampling locations.
These techniques have the advantage of not relying on a grid, which is an essential feature for various applications such as magnetic resonance imaging or radio-interferometry.
In addition, they spark the hope of avoiding the curse of dimensionality encountered in combinatorial problems.
We show that this dream is not realistic, but that the situation improves by considering large signals datasets and specific variable metric techniques.
We conclude the paper by illustrating our findings on 1D experiments.

\section{Notation}

In this paper, we will focus on discrete 1D signals, for the ease of exposition.
However, the main arguments apply to arbitrary dimensions and continuous signals as well.

We consider a signal $u$ as a vector of $\C^N$ with $N\in 2\Nbb$. 
We let $\mathcal{N}=\left\llbracket-\frac{N}{2},\frac{N}{2}-1\right\rrbracket$. 
An alternative way to represent a signal $u\in\C^N$ is to use a discrete measure $\mu$ of the form:
\begin{equation}\label{eq:def_mu}
\mu=\sum_{n\in \Nc} u_n \delta_{\frac n N}.
\end{equation}
Given a location $\xi\in\R$, we define:
\begin{equation}
\hat u(\xi)\eqdef\frac{1}{\sqrt N}\sum_{n\in \Nc} u_n e^{-2\iota\pi\left\langle \xi,\frac n N \right\rangle},
\end{equation}
which can be seen as the continuous Fourier transform of the measure $\mu$.
We consider $\Xi=[\xi_1,\hdots,\xi_M] \in \R^{M}$ a set of $M$ locations. The Fourier transform $\hat u (\Xi)\in \C^M$ at the locations $\Xi$ can be written as a matrix-vector product of the form $\hat u(\Xi)=A(\Xi)^* u$ with the normalized Vandermonde matrix $A(\Xi)\in\C^{N\times M}$ defined by
\begin{equation*}
    A(\Xi)_{n,m} \eqdef \frac 1{\sqrt N} e^{2\iota\pi\left\langle \xi_m,\frac{n}{N} \right\rangle}.
\end{equation*}
In what follows, we let $a(\xi)\in \C^N$ denote the vector defined for all $n\in \Nc$ by 
\begin{equation*}
    a(\xi)[n] \eqdef \frac 1{\sqrt N} e^{2\iota\pi\left\langle \xi,\frac n {N}\right\rangle},
\end{equation*}
so that 
\begin{equation*}
    A(\Xi) = [a(\xi_1),\hdots, a(\xi_M)].
\end{equation*}
The matrix $A(\Xi)^*$ can be seen as the nonuniform Fourier transform \cite{oppenheim1971computation} from the grid to the set of sampling locations $\Xi$.
We let $(A(\Xi)^*)^+$ denote the pseudo-inverse of $A(\Xi)^*$.

\section{Preliminaries}\label{sec:preliminaries}

Below, we first describe the precise mathematical setting and then turn to some preliminary results.

\subsection{The setting}

Let $u\in \C^N$ denote a signal. 
We assume that a sampling device allows to pick $M$ frequencies $\xi_1, \hdots, \xi_M$ in $\R$, yielding the set of measurements $y=A(\Xi)^*u+w$ with $w\sim\Nc(0,\sigma^2\Id)$ a white Gaussian noise.
A vast amount of reconstruction techniques have been designed in the literature to reconstruct $u$ from $y$.
A generic reconstructor can be defined as a mapping $\Rc:(\C^M\times \R^M)\to \C^N$ that takes as an input a measurement $y\in \C^M$ and a sampling scheme $\Xi\in \R^M$ and outputs a reconstructed signal $\Rc(y,\Xi)$.
Given a collection of signals $u_1,\hdots, u_P$ and a reconstructor $\Rc$, a natural framework to find the best sampling scheme $\Xi$ is to solve the following optimization problem:
\begin{equation}\label{eq:function_sum}
    \inf_{\Xi\in \R^M} \frac{1}{2P}\sum_{p=1}^P \E_w( \|\Rc(A(\Xi)^*u_p +w,\Xi) - u_p\|_2^2 ).
\end{equation}
This problem can be attacked with first order methods that continuously optimize the sampling locations $\xi_m$, see for instance \cite{weiss2019pilot,gossard2020off,wang2021b}.
In this work, we will concentrate on three simple linear reconstruction methods of the form
$\Rc(y,\Xi)=R(\Xi)y$:
\paragraph{The back-projection method} which consists in defining the reconstructor as $R_1(\Xi)=A(\Xi)$ or
	\begin{equation}\label{eq:recon_adj}
	    \Rc_1(y,\Xi)\eqdef A(\Xi)y.
	\end{equation}
\paragraph{The pseudo-inverse method} where the reconstructor is defined with $R_2(\Xi)=\left(A(\Xi)^*\right)^+$ or
	\begin{equation}\label{eq:recon_inv}
	    \Rc_2(y,\Xi)\eqdef \left(A(\Xi)^*\right)^+ y.
	\end{equation}
\paragraph{The Tikhonov method} (or regularized inverse) which consists in solving the following quadratic problem:
	\begin{equation}\label{eq:recon_tik}
	    \Rc_3(y,\Xi) \eqdef (1+\lambda)\argmin_{f\in\C^N} \f12\| A(\Xi)^*f - y\|_2^2+\frac\lambda 2 \|f\|_2^2
	\end{equation}
    for $\lambda>0$. Hence 
    \begin{equation}
        R_3(\Xi)=(1+\lambda)\left(A(\Xi)A(\Xi)^*+\lambda\Id\right)^{-1}A(\Xi).
    \end{equation}
    The multiplication by $(1+\lambda)$ is there to compensate the bias introduced by the regularization and will later simplify the expressions. 
A similar analysis can be carried out for the more standard solver $R_3(\Xi)=\left(A(\Xi)A(\Xi)^*+\lambda\Id\right)^{-1}A(\Xi)$, but it leads to significantly more complicated formulas, which we prefer avoiding for the sake of readability.

These techniques are quite popular in the actual practice.
We restrict our analysis to linear reconstructors of the type \eqref{eq:recon_adj}, \eqref{eq:recon_inv} and \eqref{eq:recon_tik} for simplicity reasons.
Note that \eqref{eq:recon_inv} corresponds to the limit case of \eqref{eq:recon_tik} when $\lambda$ tends to zero. Numerical experiments reveal that the optimization issues raised in Theorems~\ref{thm:main1} and \ref{thm:gradient_flatness} also apply to nonlinear reconstructors such as sparsity promoting convex penalties.
However, the techniques used in the proofs do not directly extend to this framework.

We first analyze the problem with a single image $u$ in the dataset, i.e. $P=1$.
Let us define three cost functions $J_1$, $J_2$ and $J_3$ which respectively correspond to the back-projection, the pseudo-inverse and the regularized inverse.
\begin{definition}[Cost function]
    Given a signal $u$, a sampling scheme $\Xi$ and a reconstruction method $R(\Xi)$, the cost function reads
    \begin{equation}\label{eq:def_J}
        J(\Xi) \eqdef \E_w\left( \f12\left\| R(\Xi)(A(\Xi)^*u+w)-u \right\|_2^2 \right)
    \end{equation}
    where $w\sim \Nc\left(0, \sigma^2\Id\right)$ is white Gaussian noise.
\end{definition}

\subsection{Elementary observations}\label{sec:elementary_observations}

We will make use of the following definitions.
\begin{definition}[The min distance]

Given a set of sampling points $\Xi$, the min distance $\md(\Xi)$ is defined by 
\begin{equation*}
\md(\Xi) \eqdef \min_{m\neq m'} \dist(\xi_m, \xi_{m'})
\end{equation*}
where $\dist$ is the distance on the torus defined for $(\xi_1,\xi_2) \in \R^2$ as
\begin{equation}
\dist(\xi_1,\xi_2) \eqdef \inf_{k\in\Z} \|\xi_1 - \xi_{2}-kN\|_\infty.
\end{equation}
\end{definition}

\begin{definition}[Subgrid]
Throughout the paper, we say that $\Xi \in [-N/2,N/2[^M$ is a \emph{subgrid} if $\xi_m-\xi_{m'}\in \Z^*$ for all $m\neq m'$.
\end{definition}

\begin{proposition}[$J$ is $N$-periodic]
We have 
\begin{equation}
J(\Xi \mod N) = J(\Xi).
\end{equation}
\end{proposition}
\begin{proof}
Let $n=k N$ with $k\in \Nbb$. The proof simply stems from the fact that $a(\xi+n)=a(\xi)$.
\end{proof}

The previous proposition shows that we can restrict our attention to frequencies $\xi$ belonging to the set $[-N/2,N/2[$.
\begin{proposition}[Existence of minimizers]
For any $M\in \Nbb$ and any $u\in \C^N$, there exists at least one minimizer of $J$ on $[-N/2,N/2[^M$.
\end{proposition}
\begin{proof}
We start by noticing that $J$ is a $C^\infty$ function since it is defined as a composition of $C^\infty$ functions. Hence it is also continuous on $[-N/2,N/2]^M$. This yields the existence of at least one minimizer.
\end{proof}

Now we proceed to a reformulation of the problem by rearranging the terms involved in the definition of $J$.

\begin{proposition}\label{prop:decompR}
    The reconstructors associated to $J_1$, $J_2$ and $J_3$ defined in \eqref{eq:def_J} can be expressed as
    \begin{equation}\label{eq:ReqAQ}
        R(\Xi) = A(\Xi)Q(\Xi)
    \end{equation}
    (i.e. the solution lives in $\ran(A)$) with: 
    \begin{itemize}
        \item $Q_1(\Xi)=\Id$
        \item $Q_2(\Xi)=(A(\Xi)^*A(\Xi))^+$
        \item $Q_3(\Xi)=(1+\lambda)(A(\Xi)^*A(\Xi)+\lambda\Id)^{-1}$.
    \end{itemize} 
\end{proposition}
\begin{proof}
    For $Q_1$, there is nothing to prove.
    For $Q_2$, we use one of the standard properties of the pseudo-inverse.
    For $Q_3$, we use the equality $A\left(A^*A+\lambda\Id\right) = \left(AA^*+\lambda\Id\right)A$ and then left multiply by $\left(AA^*+\lambda\Id\right)^{-1}$ and right multiply by $\left(A^*A+\lambda\Id\right)^{-1}$.
\end{proof}

\begin{proposition}\label{prop:decompJ}
    Letting $\hat u(\Xi)=A(\Xi)^*u$, we have
    \begin{align}
        J(\Xi) &= \f12\|u\|_2^2 - \langle Q(\Xi)\hat u(\Xi), \hat u(\Xi)\rangle \label{eq:generic_decomp} \\
        &\quad + \f12 \|R(\Xi)\hat u(\Xi)\|_2^2 + \f12 \E_w\left( \|R(\Xi)w\|_2^2 \right) \nonumber
    \end{align}
\end{proposition}
\begin{proof}
    We drop the dependency in $\Xi$ to simplify the notation.
    \begin{align*}
        2 J &= \|u\|_2^2 + \|RA^*u\|_2^2 + \E_w\left(\|Rw\|_2^2\right) \\
        &\quad + 2\E_w\left(\Re\langle RA^*u-u,w \rangle\right) -2\Re\langle RA^*u,u \rangle \\
        &= \|u\|_2^2 + \|RA^*u\|_2^2 + \E_w\left(\|Rw\|_2^2\right) -2\langle Q\hat u,\hat u \rangle
    \end{align*}
    where we used $Q(\Xi)^*=Q(\Xi)$ and $\E_w(w)=0$.
\end{proof}

Equation~\eqref{eq:generic_decomp} greatly simplifies when $\Xi$ is a subgrid. 
Let us define the following function

\begin{equation}\label{eq:Jtilde}
    \tilde J(\Xi) \eqdef \f12\| u\|_2^2 - \f12\|\hat u(\Xi)\|_2^2 +\frac{\sigma^2M}2.
\end{equation}

\begin{proposition}\label{prop:decomp_subgrid}
    When $\Xi$ is a subgrid, $J(\Xi)=\tilde J(\Xi)$.
\end{proposition}
\begin{proof}
    When $\Xi$ is a subgrid, we have $A(\Xi)^* A(\Xi) = \Id$ and we use the decomposition of Proposition~\ref{prop:decompJ} with $Q(\Xi)=\Id$, $R(\Xi)^*R(\Xi)=\Id$.
\end{proof}

\section{Theoretical issues}\label{sec:main_results}

In this section, we give the main theoretical results of the paper.

\subsection{Spurious minimizers}\label{sec:spurious_min}

The aim of this section is to illustrate a common situation where the function $J$ possesses a combinatorial number of minimizers.
We construct examples where the function $\tilde J$ defined in \eqref{eq:Jtilde} is very oscillatory, while $J-\tilde J$ is of small amplitude.
The function $J$ is close to $\tilde J$ not only for subgrids as in Proposition~\ref{prop:decomp_subgrid} but also for well-spread schemes.
Following the proof of Proposition~\ref{prop:decomp_subgrid}, and the decomposition of Proposition~\ref{prop:decompJ}, it is sufficient to control how close $Q(\Xi)$ and $R(\Xi)^*R(\Xi)$ are to $\Id$.
This is the aim of the following proposition.

\begin{proposition}[Bound on $Q$ and $R^*R$]\label{prop:bnd_q_rr}
    Consider a sampling pattern $\Xi$ such that $\md(\Xi) > 1$ and set $\epsilon=1/\md(\Xi)$. Then
    \begin{eqnarray}
        -a_i\Id \preccurlyeq& Q_i-\Id &\preccurlyeq a_i\Id \label{eq:bnd_Q} \\
        -b_i\Id \preccurlyeq& R_i^*R_i-\Id &\preccurlyeq b_i\Id, \label{eq:bnd_RR}
    \end{eqnarray}
    with
    \begin{align*}
            a_1 &= 0, \quad a_2 = \frac\epsilon{1-\epsilon}, \quad a_3 = \frac{\epsilon}{1-\epsilon} \\ 
            b_1 &= \epsilon, \quad b_2 = \frac\epsilon{1-\epsilon}, \quad b_3 = \frac{4\epsilon}{(1-\epsilon)^2}
    \end{align*}
\end{proposition}
The proof is postponed to Section~\ref{sec:proof_bnd_q_rr}.

\begin{theorem}[A combinatorial number of minimizers \label{thm:main1}]
    Set a number of samples $M\in \Nbb$ and consider a vector $u\in \C^N$ such that the following properties are verified
    \begin{enumerate}[i)]
      \item The modulus $|\hat u|^2$ possesses a subset of $K\geq M$ local maximizers $Z = \{\zeta_1,\hdots, \zeta_K\}$ separated by a distance at least $\delta=\md(Z)$ with $\delta>1+2r$ for some $r>0$. 
      \item The modulus $|\hat u|^2$ is locally strictly concave for each $\zeta_k$:
        \begin{equation*}
            |\hat u|^2(\zeta_k+h)\leq |\hat u|^2(\zeta_k)-\frac c 2 h^2, \forall h \in [-r,r]
        \end{equation*}
      for some $c>0$.
      \item For any subset $\bar \Xi$ of $M$ distinct points in $Z$, we have
    \begin{equation}\label{eq:condition_cryptic}
        \frac{cr^2}{2} > \left(b+2a\right)\|\hat u(\bar \Xi)\|_2^2+b M\sigma^2 
    \end{equation}
        where $a$ and $b$ are given in Proposition~\ref{prop:bnd_q_rr} with $\epsilon=\frac{1}{\delta-2r}$. 
    \end{enumerate}
    Then, the function $J$ possesses at least $\binom{K}{M} \cdot M\,!$ local minimizers.
\end{theorem}

The proof of Theorem~\ref{thm:main1} is postponed to Section~\ref{sec:proof_thm_main1_gen}.
The conditions in Theorem~\ref{thm:main1} may look cryptic at first sight.
We first show a simple example of a function $u$ that verifies the hypotheses and leads to a huge number of critical points.

\begin{corollary}\label{cor:example_spurious}
    Assume that $N\in 4\Nbb$ and define $u\in \C^N$ as follows
    \begin{equation}\label{eq:def_u_oscillatory}
    u[n] = 
    \begin{cases}
    \sqrt{N}/2 & \textrm{ if } n=\pm N/4,\\
    0 & \textrm{otherwise}.
    \end{cases}
    \end{equation}

    Let $M=\lfloor \eta\sqrt N \rfloor$ with $\eta = \frac{\pi^2 \sqrt{2}}{256 \cdot (20+16\sigma^2)}$ then all the functions $J_i$ possess a number of minimizers larger than $M! \cdot \left(\frac{1}{2\eta}\right)^{\eta\sqrt{N}}$.

    For $\sigma\leq 1$, the bound holds for $\eta = 3\cdot 10^{-3}$.
    For $\sigma=0$ and $J_1$ the bound can be increased to $\eta=1.09\cdot 10^{-1}$.
\end{corollary}

\begin{proof}
    The choice of $u$ in \eqref{eq:def_u_oscillatory} leads to the oscillatory function $\hat u(\xi) = \cos\left(\frac{\pi}{2}\xi\right)$.
    The modulus $|\hat u|$ is maximal at every point $\xi\in 2\Nbb$. Let $\xi_0\in 2\Nbb$ and set $r=\frac{1}{4}$. For any $\xi\in [\xi_0-r,\xi_0+r]$, we have
    \begin{align*} 
    (|\hat u|^2)''(\xi)& = \frac{\pi^2}{2} \left( \sin^2\left(\frac{\pi}{2}\xi\right) -\cos^2\left(\frac{\pi}{2}\xi\right)\right) \\
    & \leq -\frac{\pi^2}{2\sqrt{2}}.
    \end{align*}

    Let $p\in \Nbb$. The conditions i) and ii) of Theorem~\ref{thm:main1} are satisfied with $Z=2p\Nbb\cap [-N/2,N/2[$, $K=\left \lfloor N/(2p) \right \rfloor$, $r=1/4$, $c=\frac{\pi^2\sqrt{2}}{8}$, $\delta = 2p$.
    Further notice that for every set $\bar \Xi\in Z^M$, $\|\hat u(\bar \Xi)\|_2^2=M$.

    For this example, the condition \eqref{eq:condition_cryptic} therefore reads
    \begin{equation}
    M < \frac{\pi^2 \sqrt{2}}{256}\cdot \left(\frac{1}{b+2a+b\sigma^2}\right).
    \end{equation}
    As long as this condition is satisfied, Theorem~\ref{thm:main1} allows to conclude on the existence of $\binom{\lfloor N/2p \rfloor}{M} \cdot M\,!$ maximizers. 

    Now, if $\delta-2r\geq 2$, we can coarsely simplify the bounds in Proposition~\ref{prop:bnd_q_rr} as 
    \begin{equation*}
    a\leq \frac{2}{\delta-2r}\quad  \mbox{and} \quad b\leq \frac{16}{\delta-2r}. 
    \end{equation*}
    Hence, a combinatorial number of minimizers is granted given that 
    \begin{equation}
    M < \frac{\pi^2 \sqrt{2}}{256}\cdot \left(\frac{\delta-2r}{20+16\sigma^2}\right).
    \end{equation}

    Now, take $p=\lfloor \sqrt{N} \rfloor$ and $M = \lfloor \eta \cdot \sqrt{N} \rfloor$ with $\eta = \frac{\pi^2 \sqrt{2}}{128 \cdot (20+16\sigma^2)}$.
    Then Theorem \ref{thm:main1} yields a number of minimizers larger than $\binom{\lfloor \sqrt{N}/2 \rfloor}{ \lfloor \eta\cdot \sqrt{N} \rfloor } \cdot M\,!$.
    Using the standard bound
    \begin{equation}
    \binom{n}{k} \geq \left(\frac{n}{k}\right)^k
    \end{equation}
    yields a number of minimizers larger than $\left(\frac{1}{2\eta}\right)^{\eta\sqrt{N}}$.

    In particular for $\sigma<1$ this yields $\eta=0.003$.
    The bound can be increased to $\eta=0.109$ for $\sigma=0$ and $J_1$.
\end{proof}

\subsection{Numerical illustration of Theorem~\ref{thm:main1}}

In this section we illustrate Theorem~\ref{thm:main1} through numerical examples in Fig.~\ref{fig:energy_profile_M2}.
We first consider the noiseless settings $\sigma=0$ and illustrate the existence of spurious minimizers for the back-projection and the pseudo-inverse methods.

We introduce the following function
\begin{equation}
F(\Xi) \eqdef \f12 \sum_{m=1}^M |\hat u(\xi_m)|^2 = \frac{1}{2}\|\hat u(\Xi)\|_2^2,
\end{equation}
which somehow measures the energy captured within a sampling scheme $\Xi$.
We also introduce the functions $G_1$ and $G_2$ such that $J_1=\f12\|u\|_2^2-F+G_1$ and $J_2=\f12\|u\|_2^2-F+G_2$. Using Proposition~\ref{prop:decompJ}, we have
\begin{equation}
G_1(\Xi) = \frac{1}{2} \left\langle \left(A(\Xi)^* A(\Xi) - \Id \right) \hat u(\Xi) , \hat u(\Xi) \right \rangle,
\end{equation}
and
\begin{equation}
G_2(\Xi) = \frac{1}{2} \left\langle \left( \Id - (A(\Xi)^* A(\Xi))^+ \right) \hat u(\Xi) , \hat u(\Xi) \right \rangle.
\end{equation}

From the left to the right, we used three different 1D signals: a high frequency cosine, a low frequency sine and a Gaussian.
We plot the different energy landscapes, for $M=2$ measurements at locations $\Xi=\{\xi_1,\xi_2\}$ and $N=16$.
From the top to the bottom, we display the functions $J_1$, $J_2$, $G_1$, $G_2$, $-F$ and the modulus of the Fourier transform $\xi\mapsto|\hat u(\xi)|$.
In order to understand the effect of the signal's structure, the local minima of $J_1$, $J_2$, $G_1$, $G_2$ and $-F$ are represented with red dots.

First notice that the cost functions are symmetric with respect to the diagonal. This simply reflects the fact that permutation of points lead to the same energy, and this illustrates the factor $M!$ in Theorem~\ref{thm:main1}.

As can be seen in all cases, the functions $G_1$ and $G_2$ vanish far away from the diagonal (see Proposition~\ref{prop:bnd_q_rr}).
These point configurations correspond to well-spread sampling schemes.
On the contrary, the function $-F$ can have a large amplitude even outside the diagonal. These two properties are the main ingredients to prove Theorem~\ref{thm:main1}.

The left column (high frequency cosine), corresponds to the example in Corollary~\ref{cor:example_spurious}. 
We see a number of minimizers that seems quadratic in $N$ for $M=2$. 
The center column (low frequency sine) shows that the number of minimizers decreases with a higher regularity of the signal, by reducing the oscillations in $F$. 
On the right (Gaussian function), we illustrate a case where $F$ has only one local maximum. 
Even in this case, the function $J$ still has valleys with shallow local minima. 
The same phenomenon appears in the center (low frequency sine). 
Notice that this phenomenon is not captured by Theorem~\ref{thm:main1}, which only relies on local maximizers of $F$. 
In these two examples, the oscillations are induced by the function $G$, which we do not explore in this paper. 

\begin{figure*}[htbp]
\centering
\begin{tabular}{c c}
\rotatebox[origin=c]{90}{\small$J_1=\f12 \|u\|_2^2-F+G_1$} &
\includegraphics[valign=m,width=.31\linewidth]{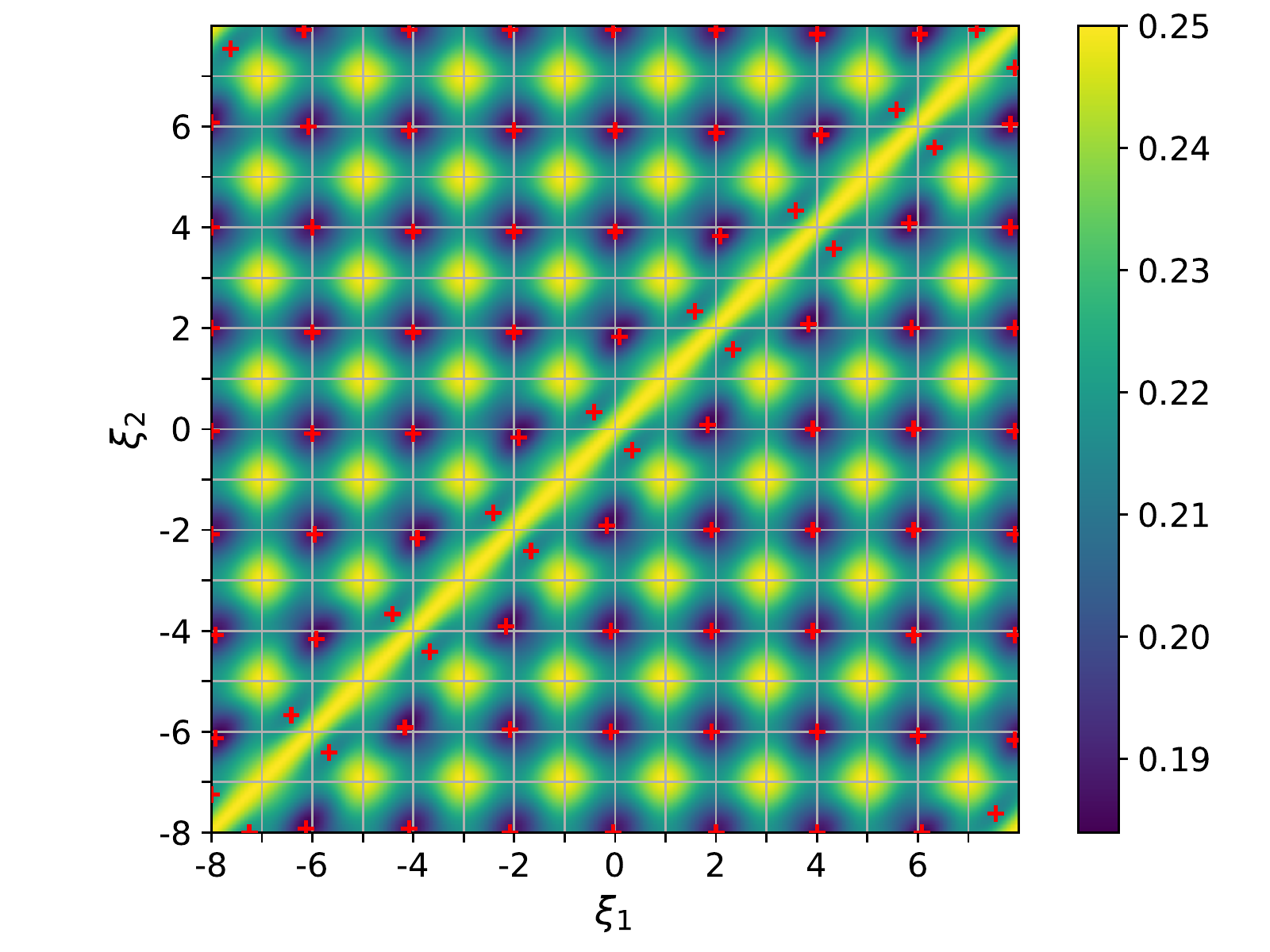}
\includegraphics[valign=m,width=.31\linewidth]{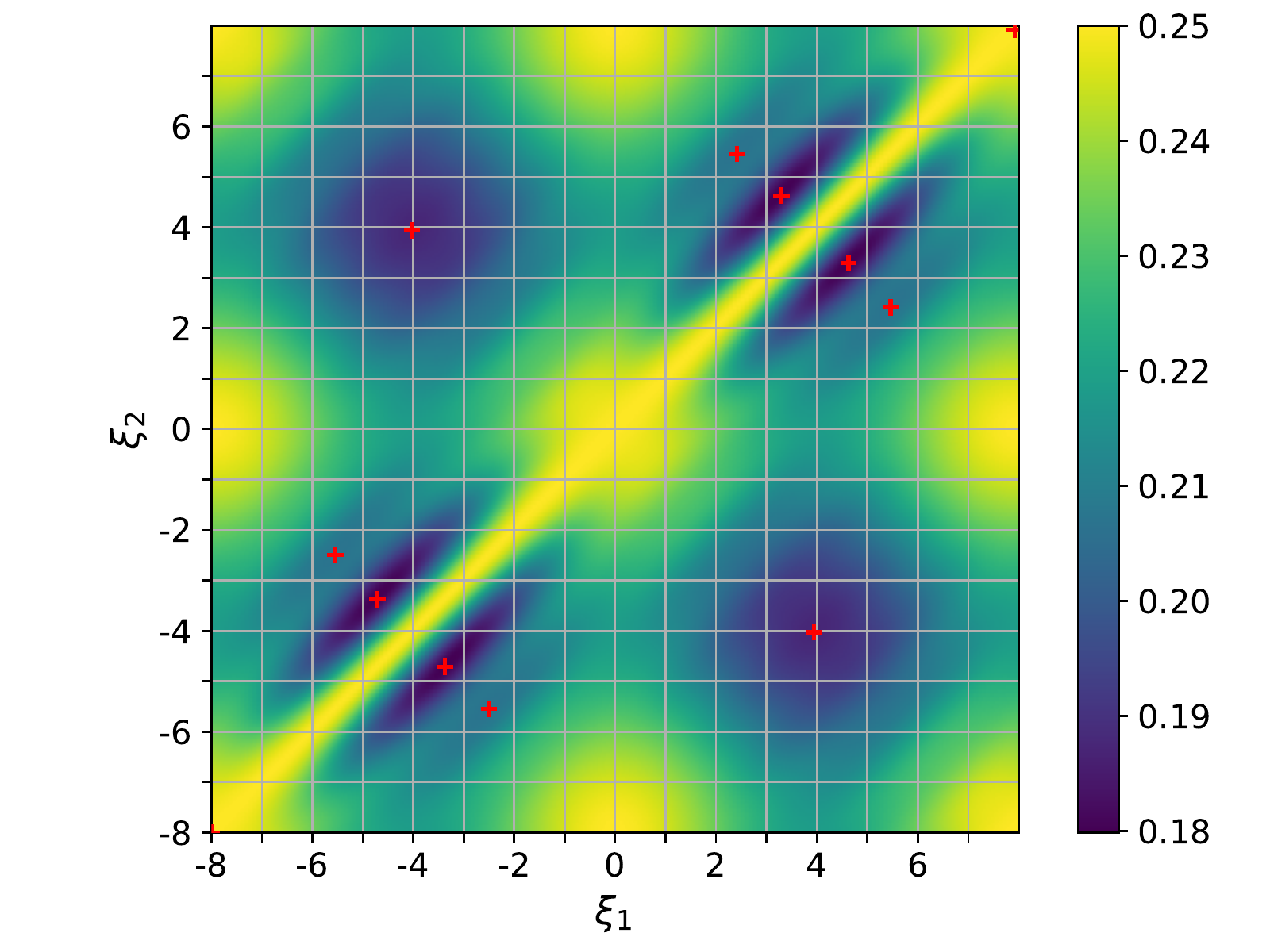}
\includegraphics[valign=m,width=.31\linewidth]{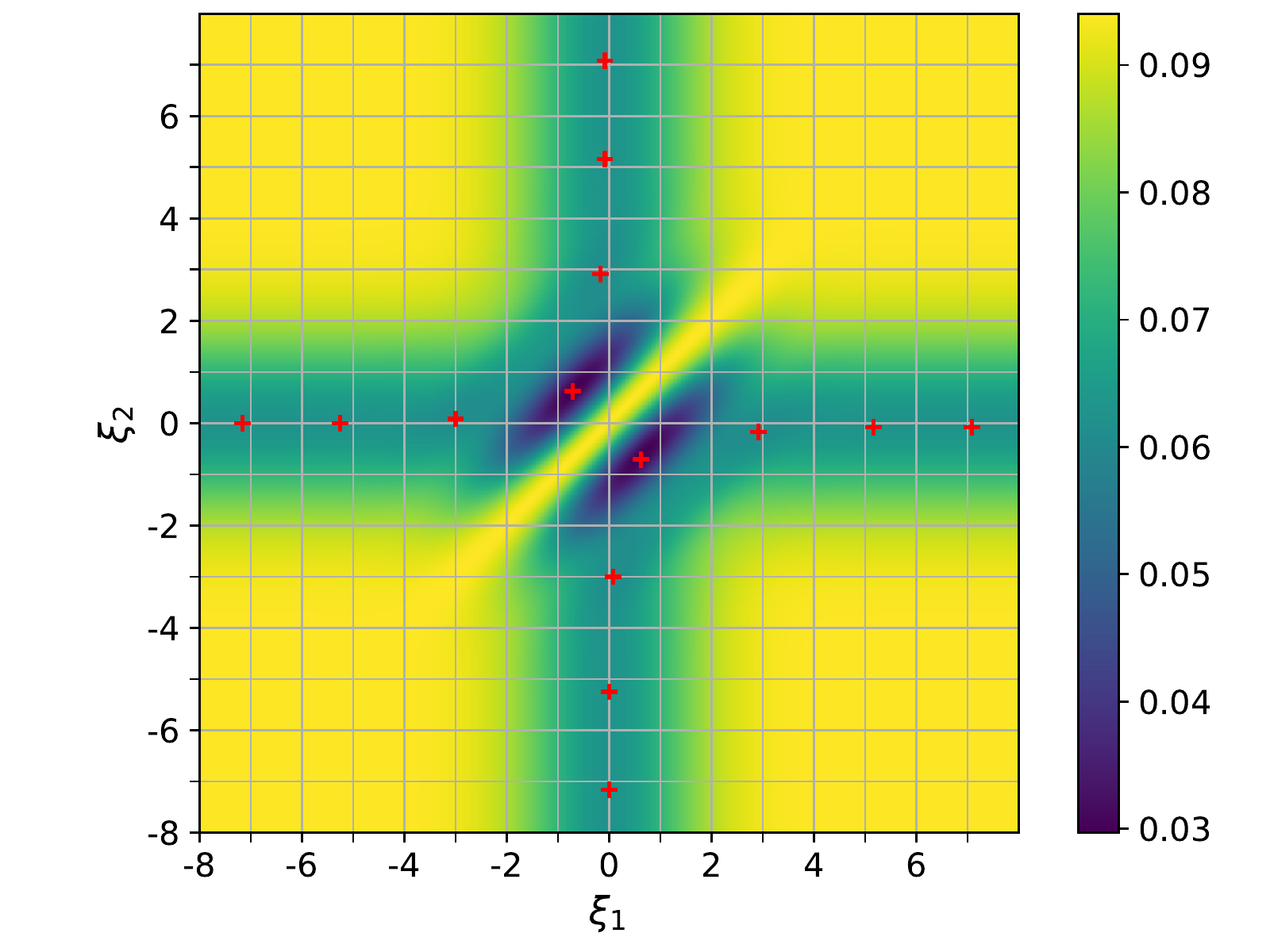} \\
\rotatebox[origin=c]{90}{\small$J_2=\f12 \|u\|_2^2-F+G_2$} &
\includegraphics[valign=m,width=.31\linewidth]{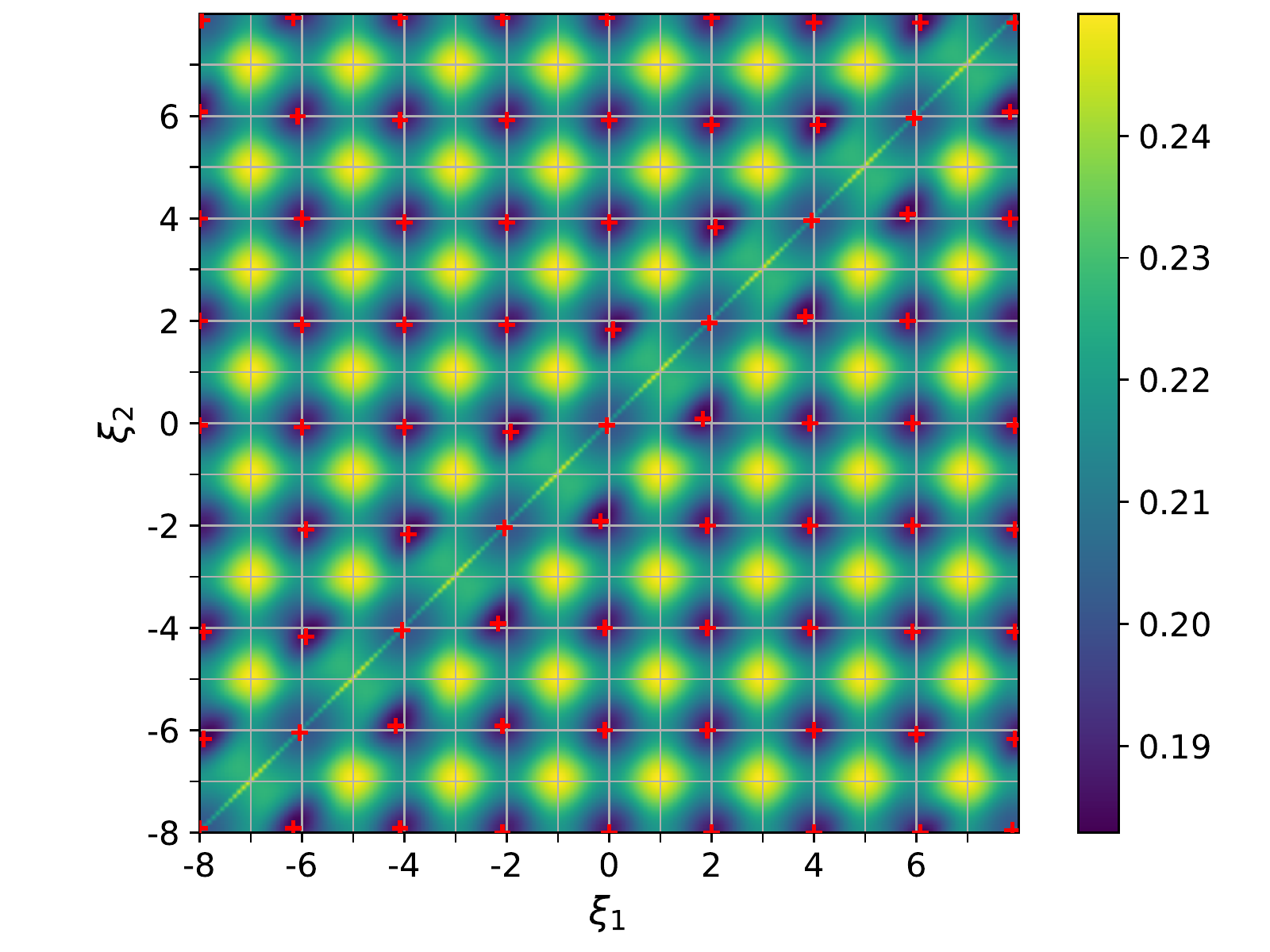}
\includegraphics[valign=m,width=.31\linewidth]{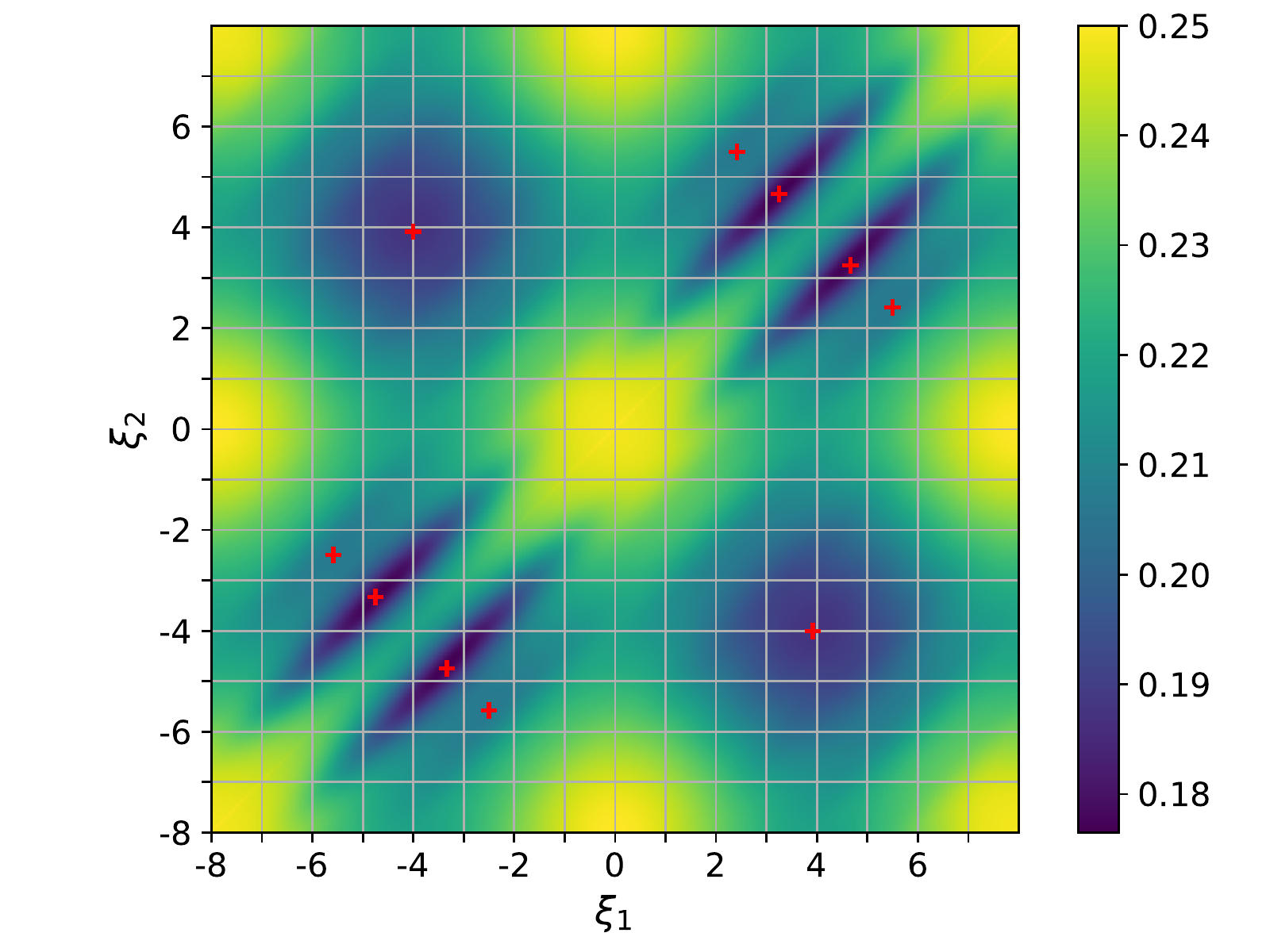}
\includegraphics[valign=m,width=.31\linewidth]{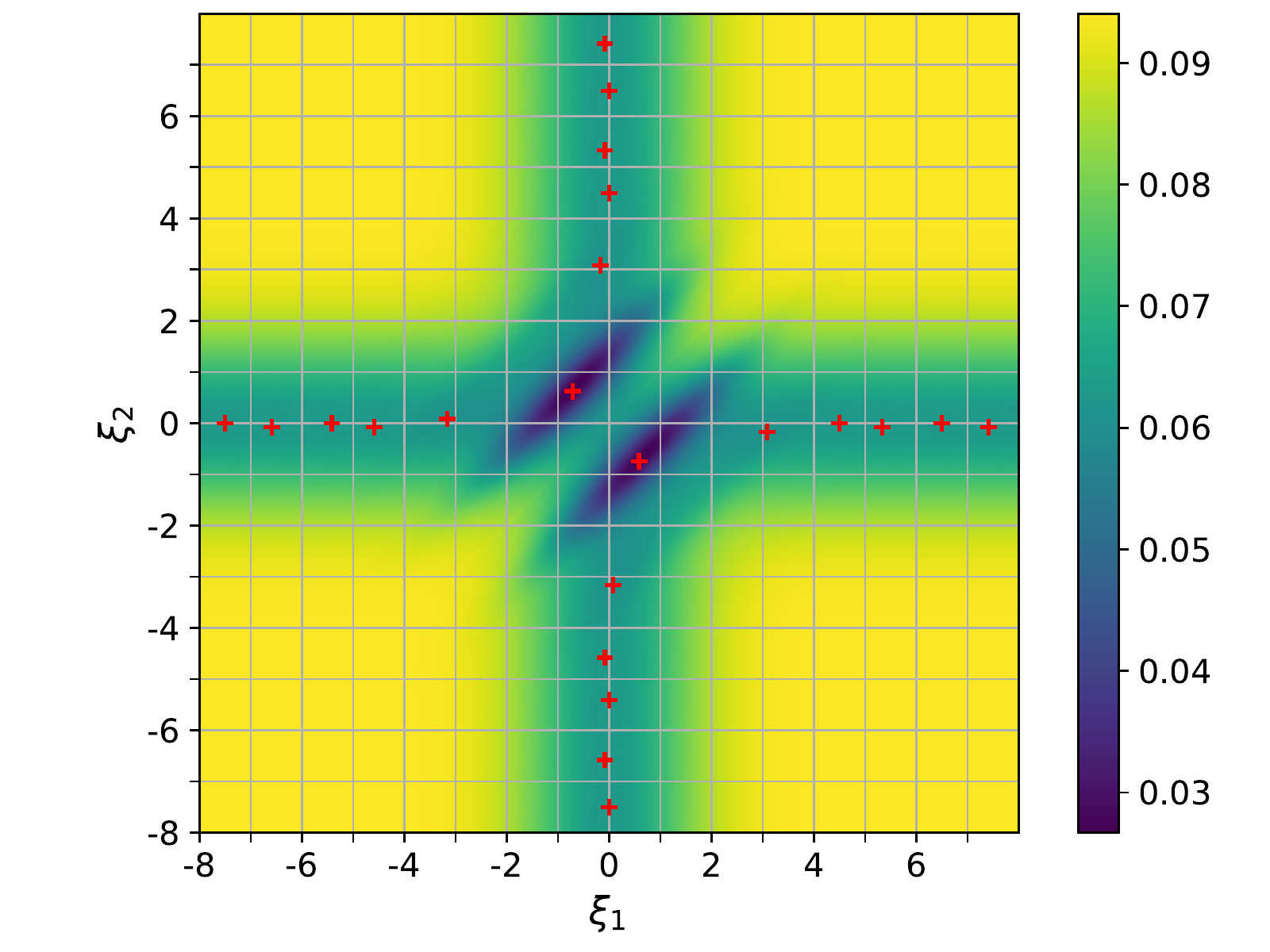} \\
\rotatebox[origin=c]{90}{\small$G_1$} &
\includegraphics[valign=m,width=.31\linewidth]{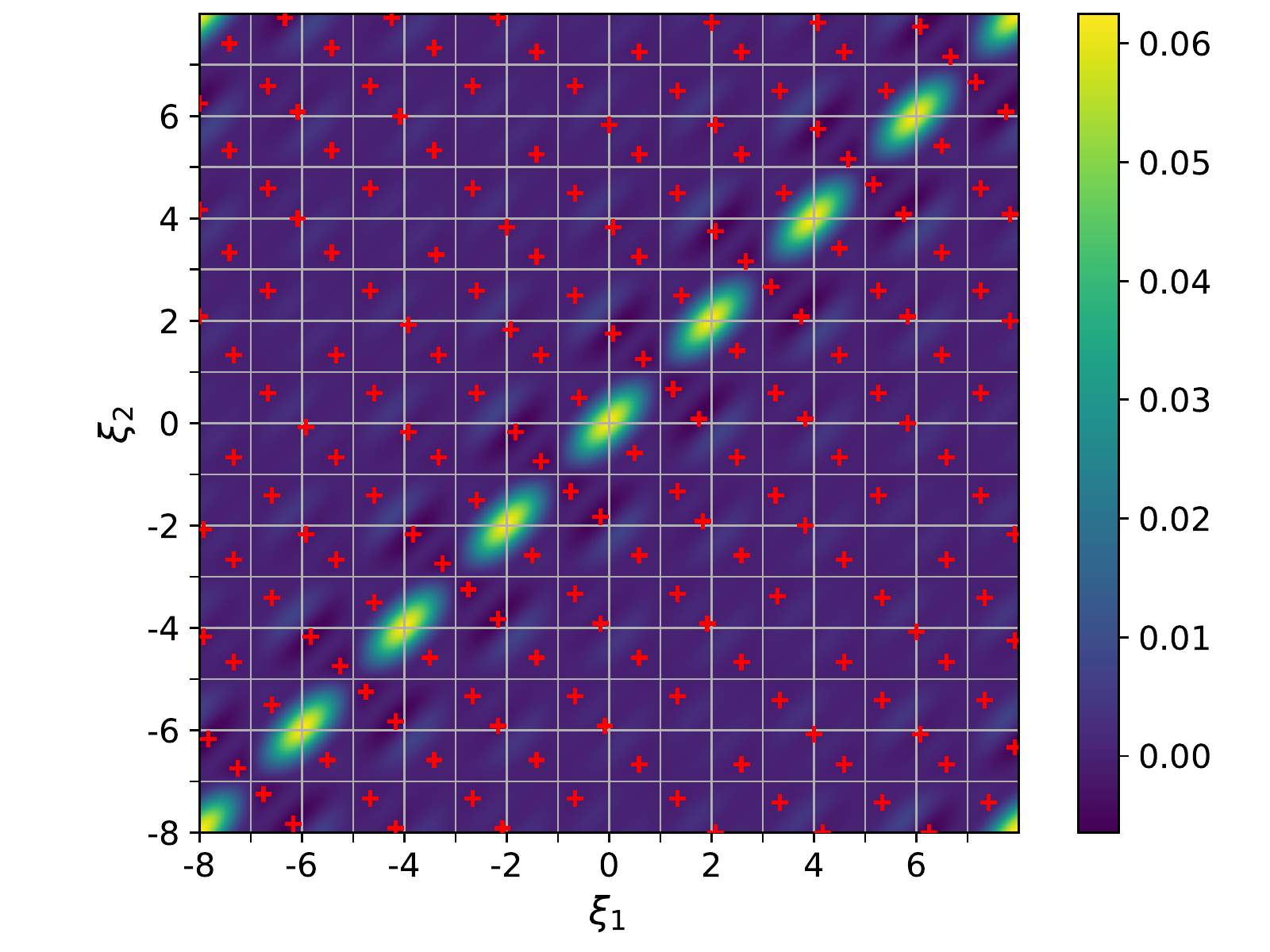}
\includegraphics[valign=m,width=.31\linewidth]{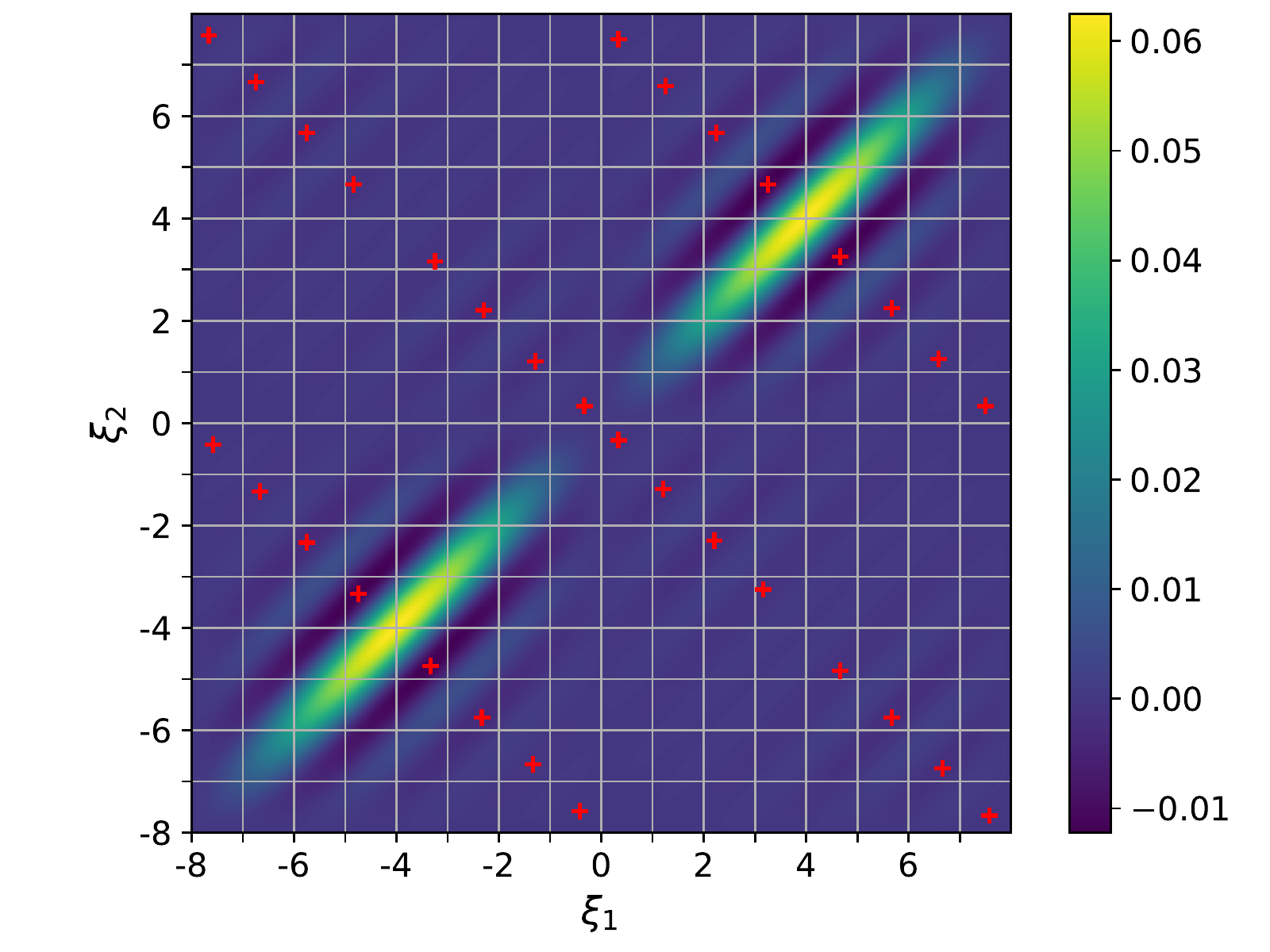}
\includegraphics[valign=m,width=.31\linewidth]{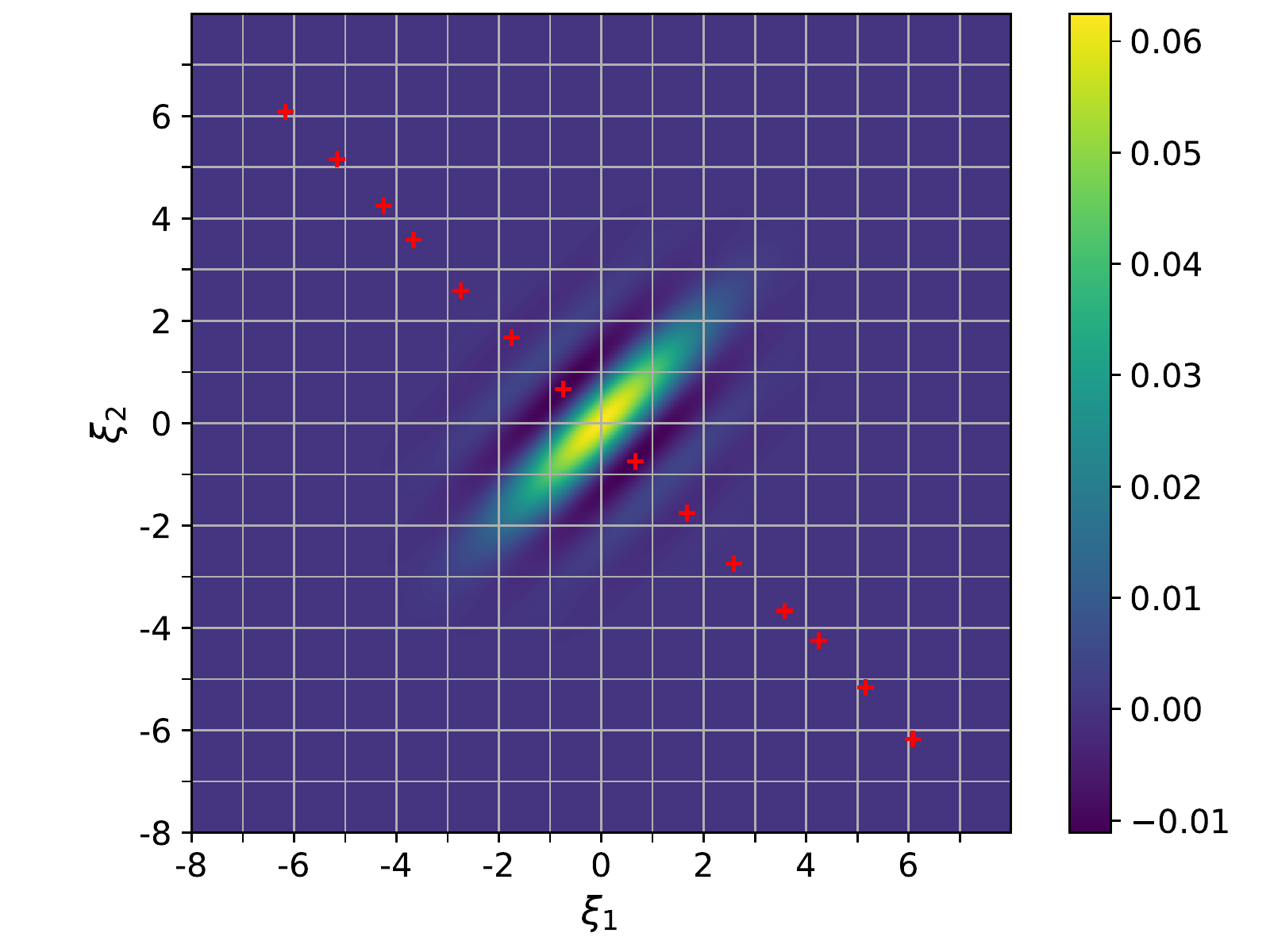} \\
\rotatebox[origin=c]{90}{\small$G_2$} &
\includegraphics[valign=m,width=.31\linewidth]{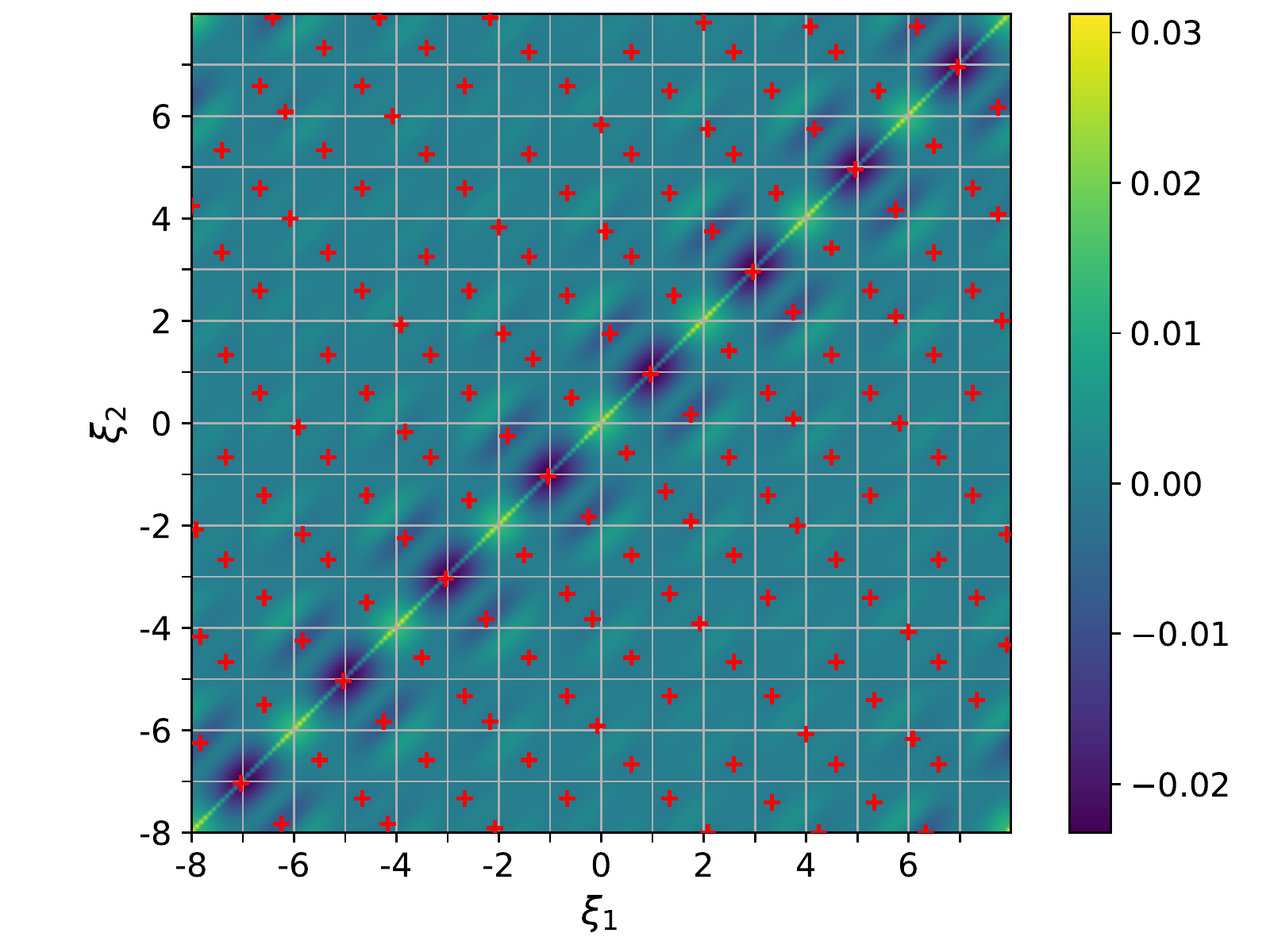}
\includegraphics[valign=m,width=.31\linewidth]{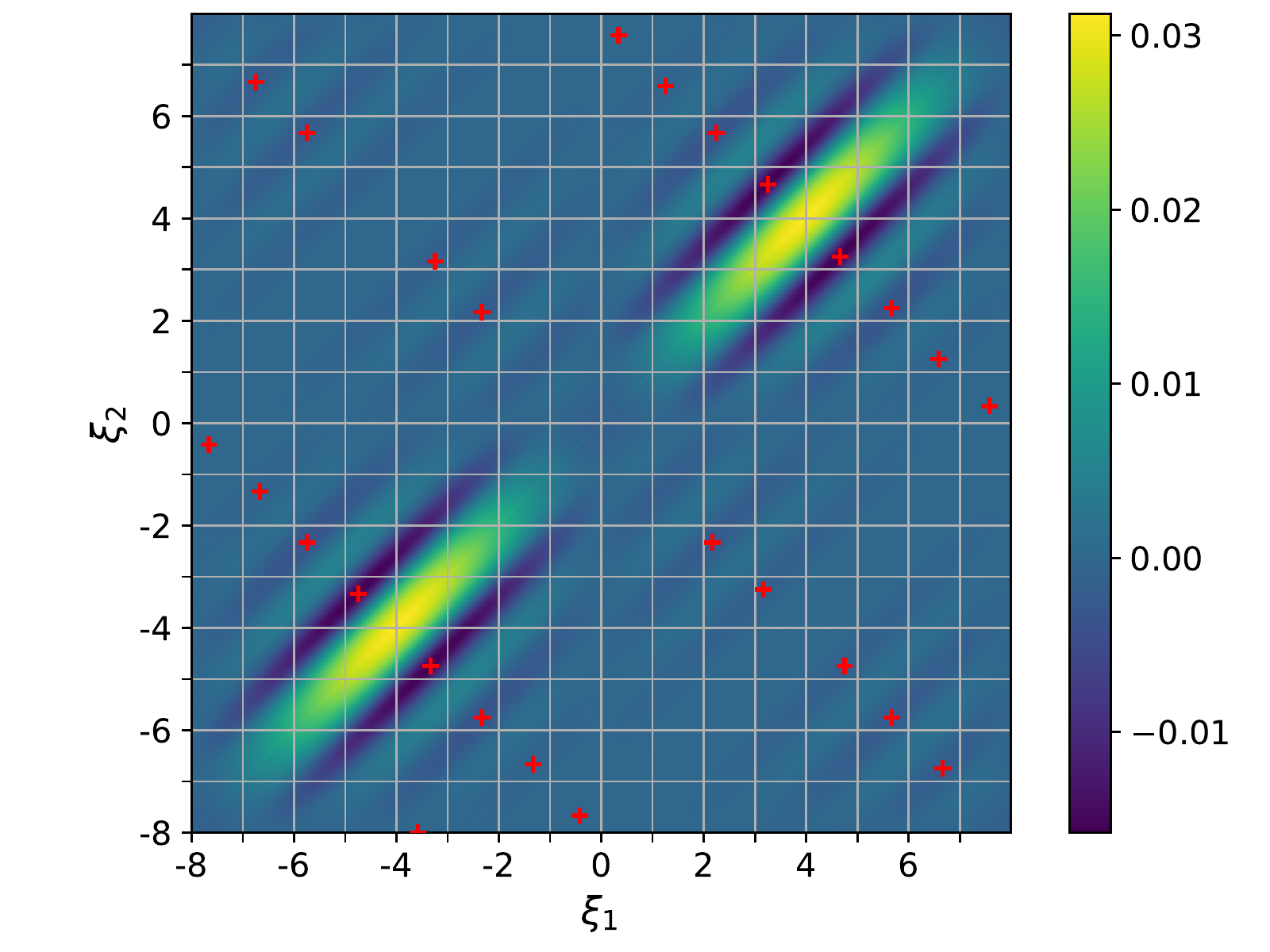}
\includegraphics[valign=m,width=.31\linewidth]{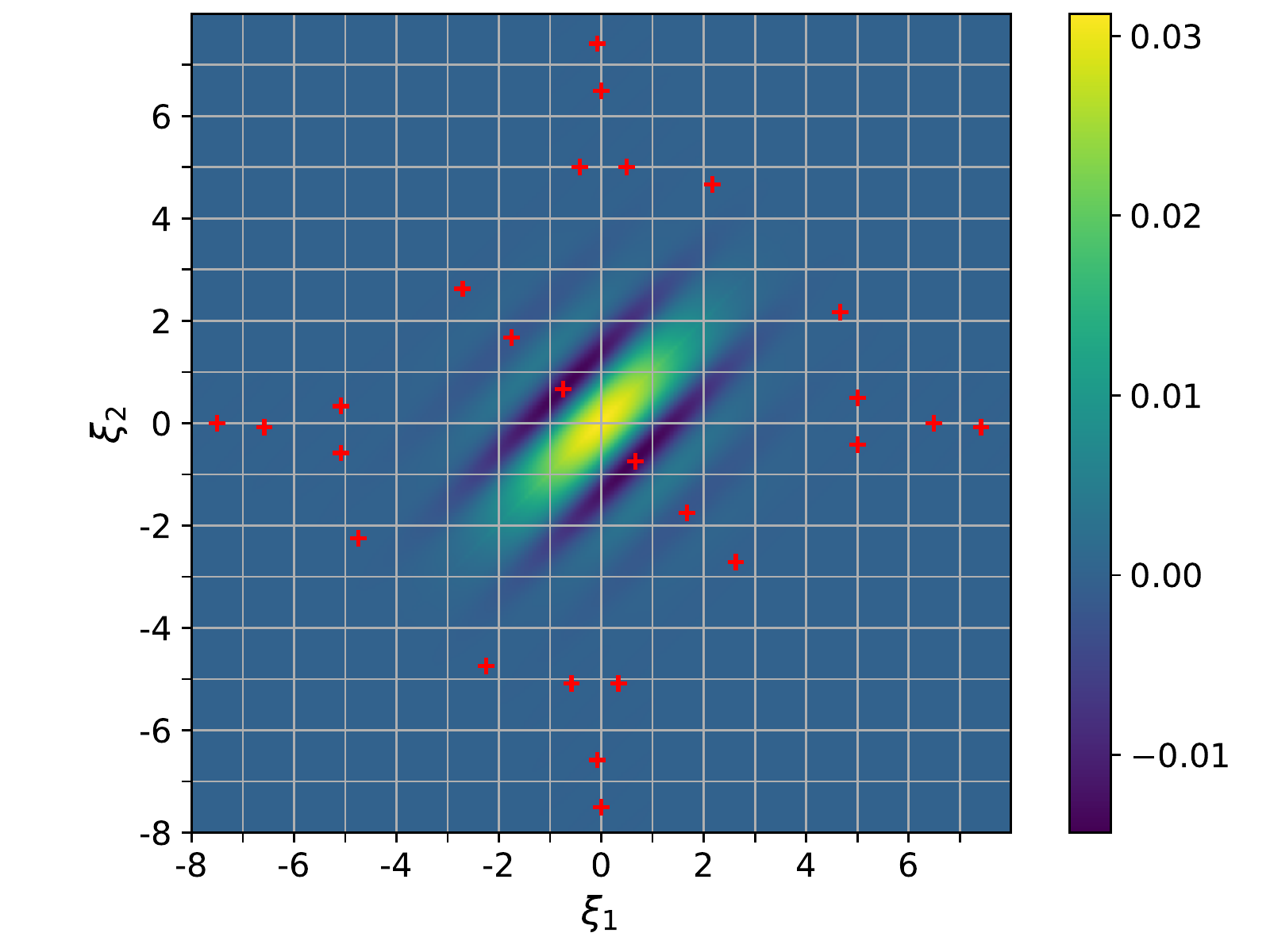} \\
\rotatebox[origin=c]{90}{\small$-F$} &
\includegraphics[valign=m,width=.31\linewidth]{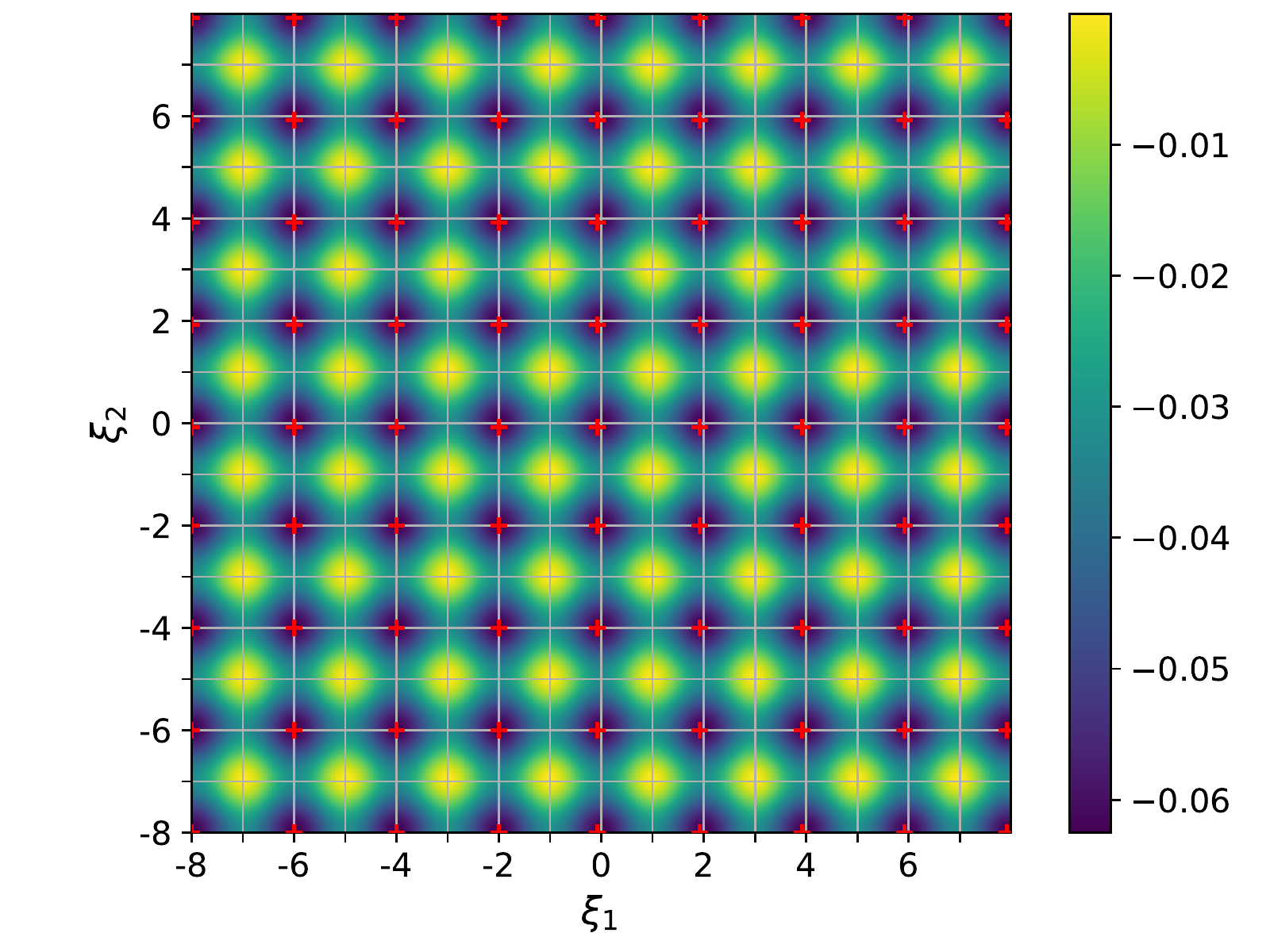}
\includegraphics[valign=m,width=.31\linewidth]{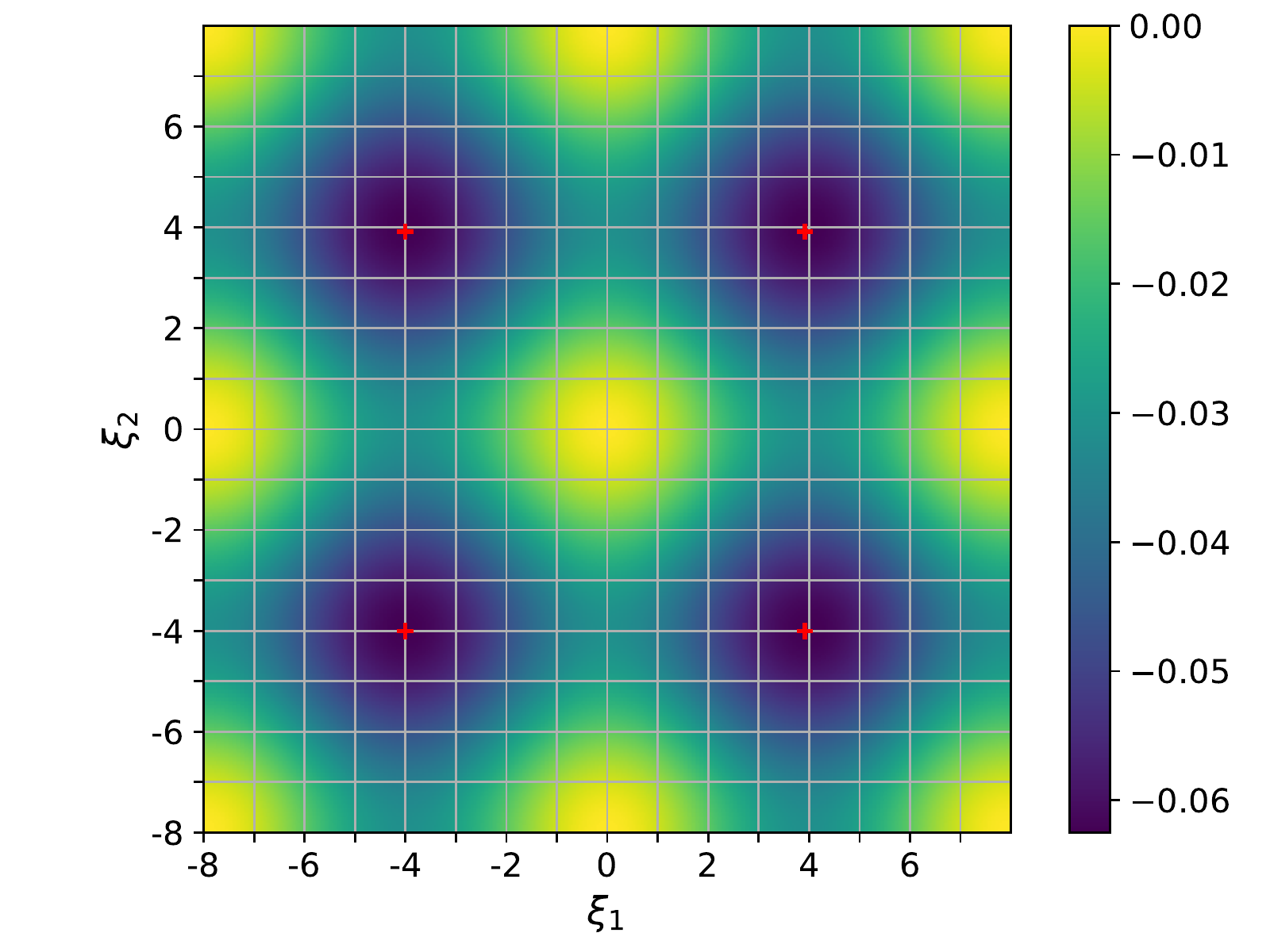}
\includegraphics[valign=m,width=.31\linewidth]{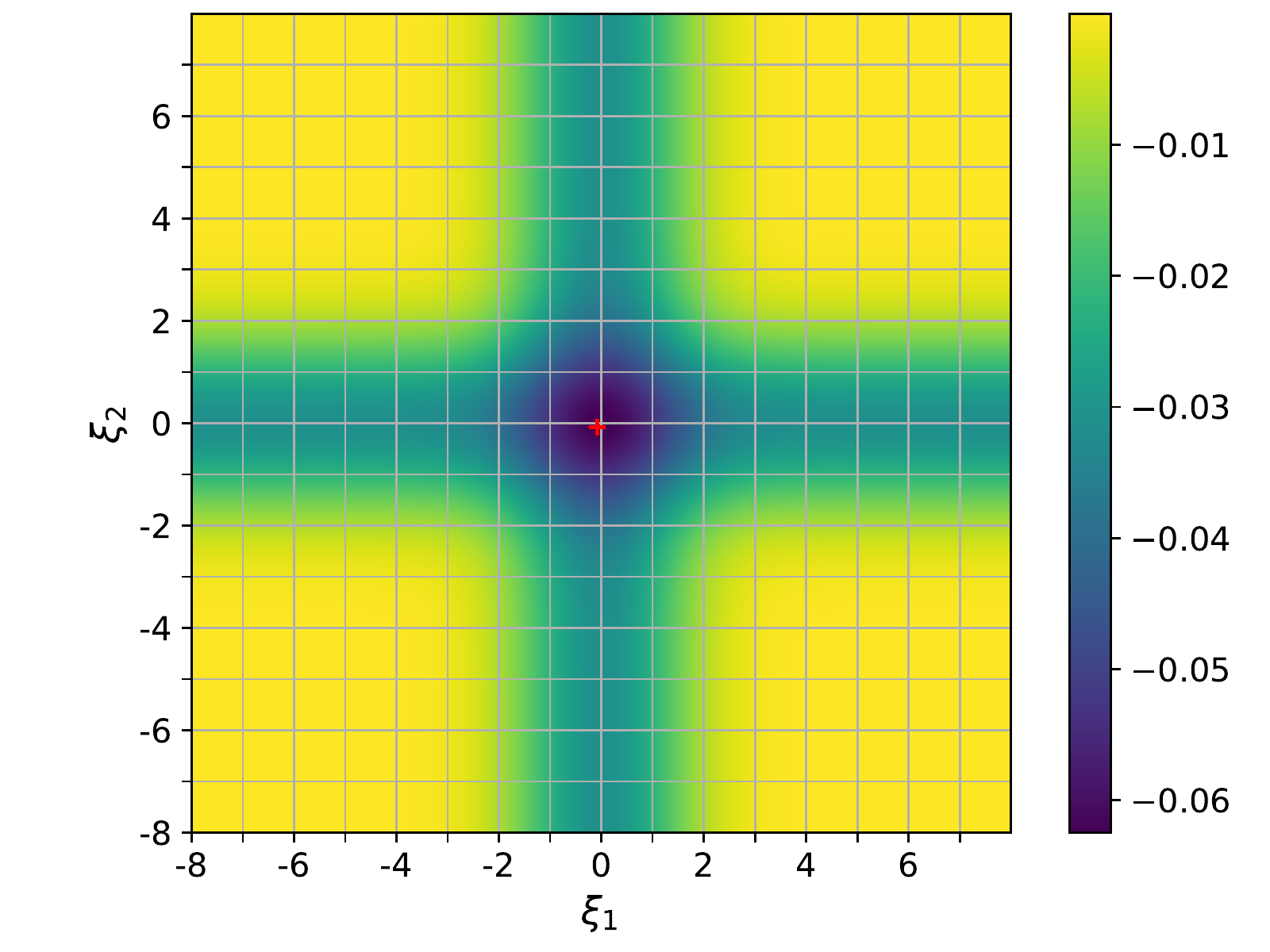} \\
\rotatebox[origin=c]{90}{\small$|\hat u|$} &
\includegraphics[valign=m,width=.31\linewidth]{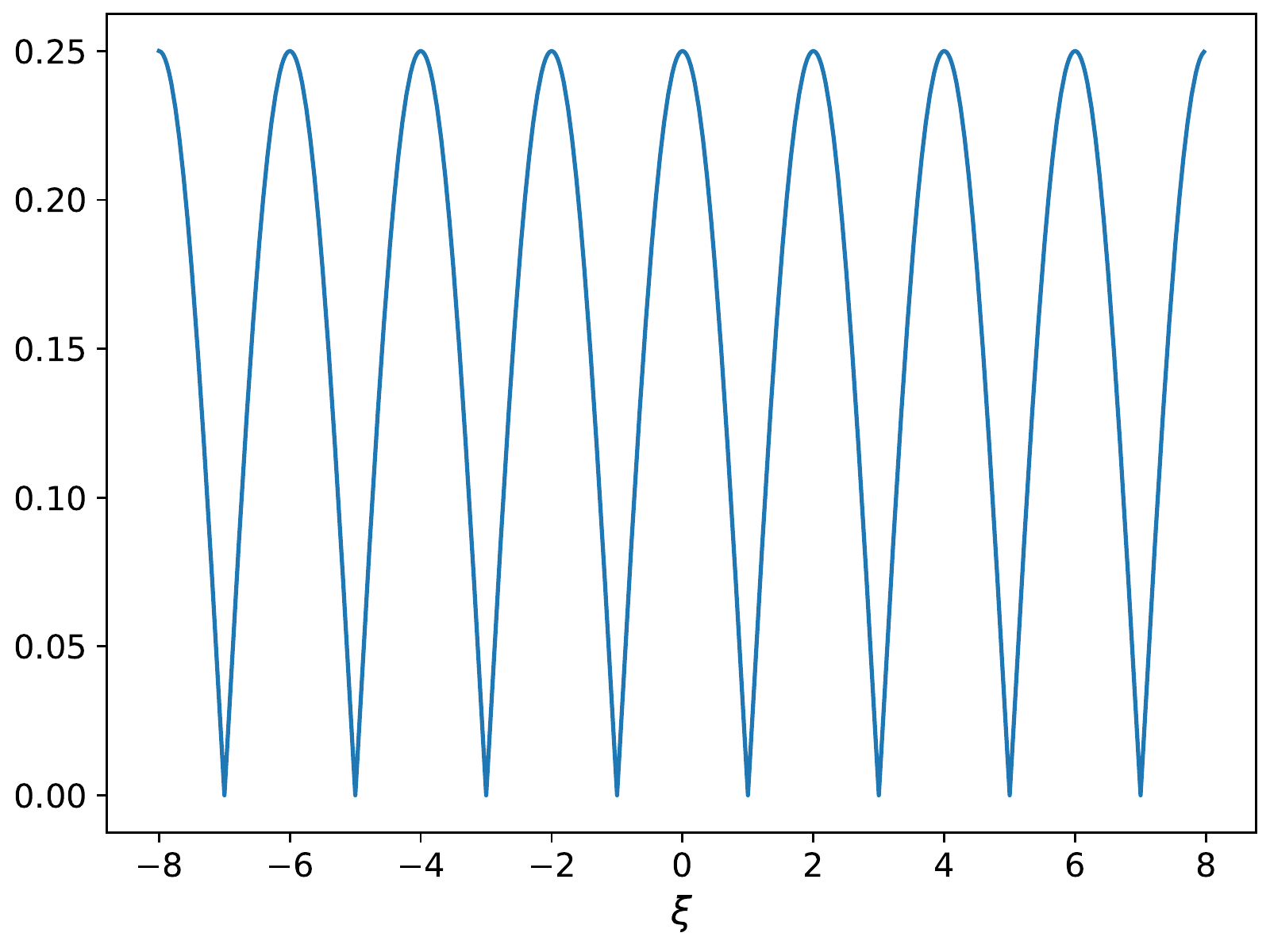}
\includegraphics[valign=m,width=.31\linewidth]{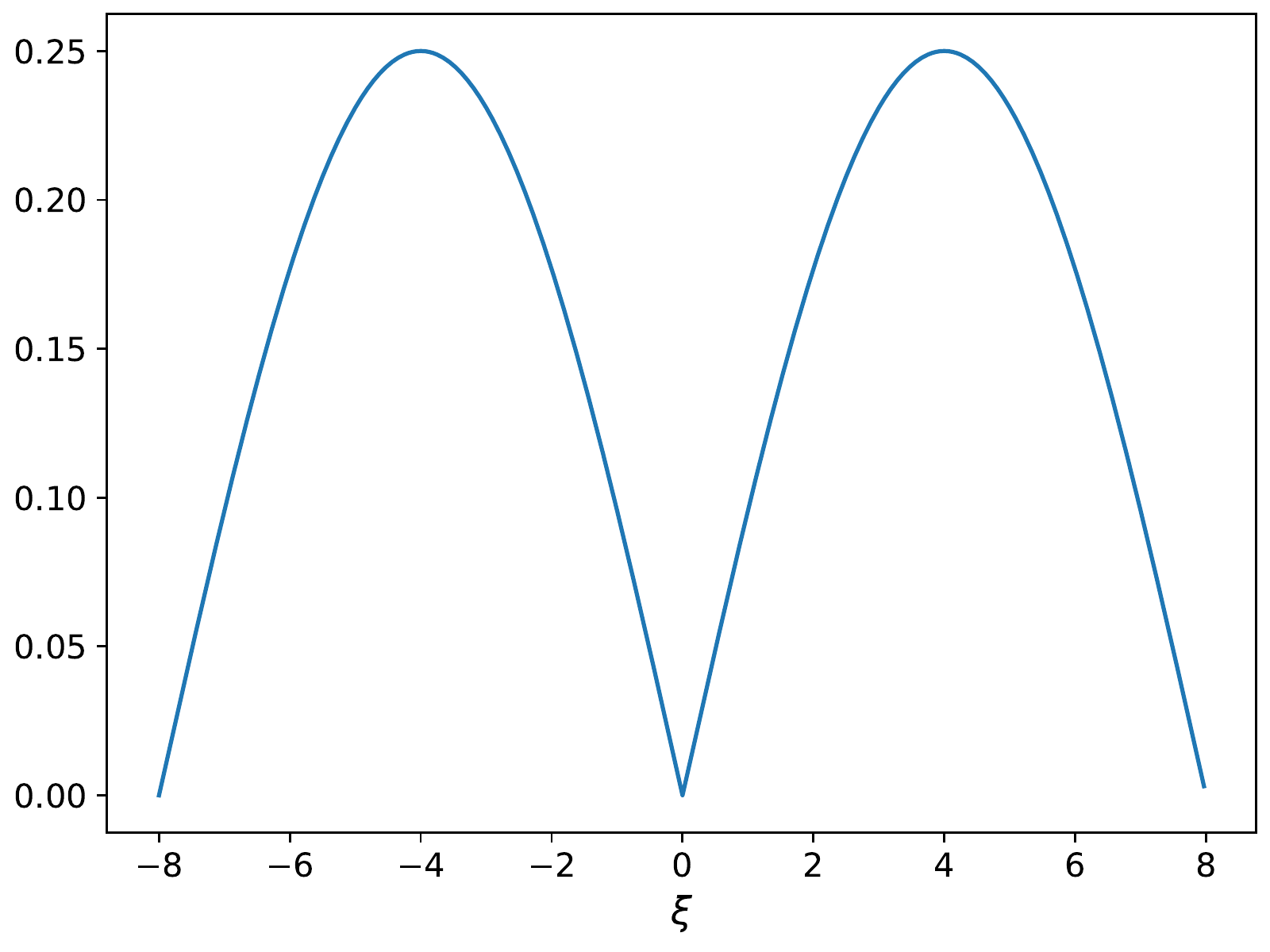}
\includegraphics[valign=m,width=.31\linewidth]{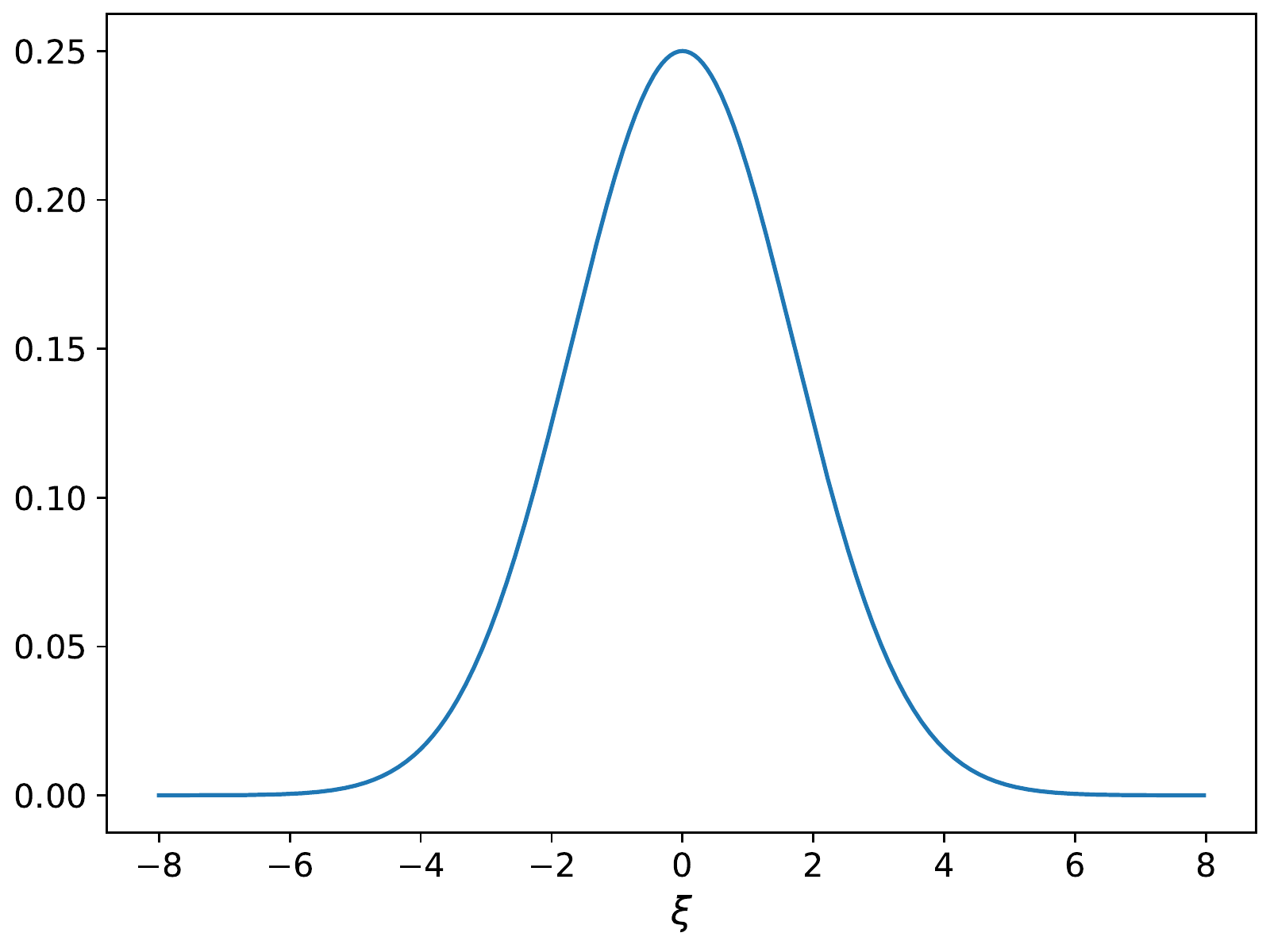}
\end{tabular}
\caption{The energy profile for $M=2$ and three different signals $\hat u$: a high frequency cosine, a low frequency sine and a Gaussian (from left to right). From top to bottom, we represent $J_1$, $J_2$, $G_1$, $G_2$, $F$ and $|\hat u|$. The red dots represent local minima.}
\label{fig:energy_profile_M2}
\end{figure*}

In Fig.~\ref{fig:energy_profile_noise}, the energy profile of $J_3$ is displayed with the low frequency signal (see Fig.~\ref{fig:energy_profile_M2} center column) for $M=2$ and for various noise levels $\sigma$ and regularization parameters $\lambda$.

Neither the noise, nor the regularization parameter $\lambda$ are able to remove the local minimizers of $J_3$ which are displayed by red dots.
The last column is a critical case where the noise prevails over the signal and the reconstruction error is high (typically $\sim 0.18$ in the noiseless setting and $\sim 0.5$ with $\sigma=5\times 10^{-1}$).

\begin{figure*}[htbp]
    \centering
    \begin{tabular}{@{}cc@{}c@{}c@{}c@{}}
    \rotatebox[origin=c]{90}{\small $\lambda=10^{1}$} &
    \includegraphics[valign=m,width=.24\linewidth]{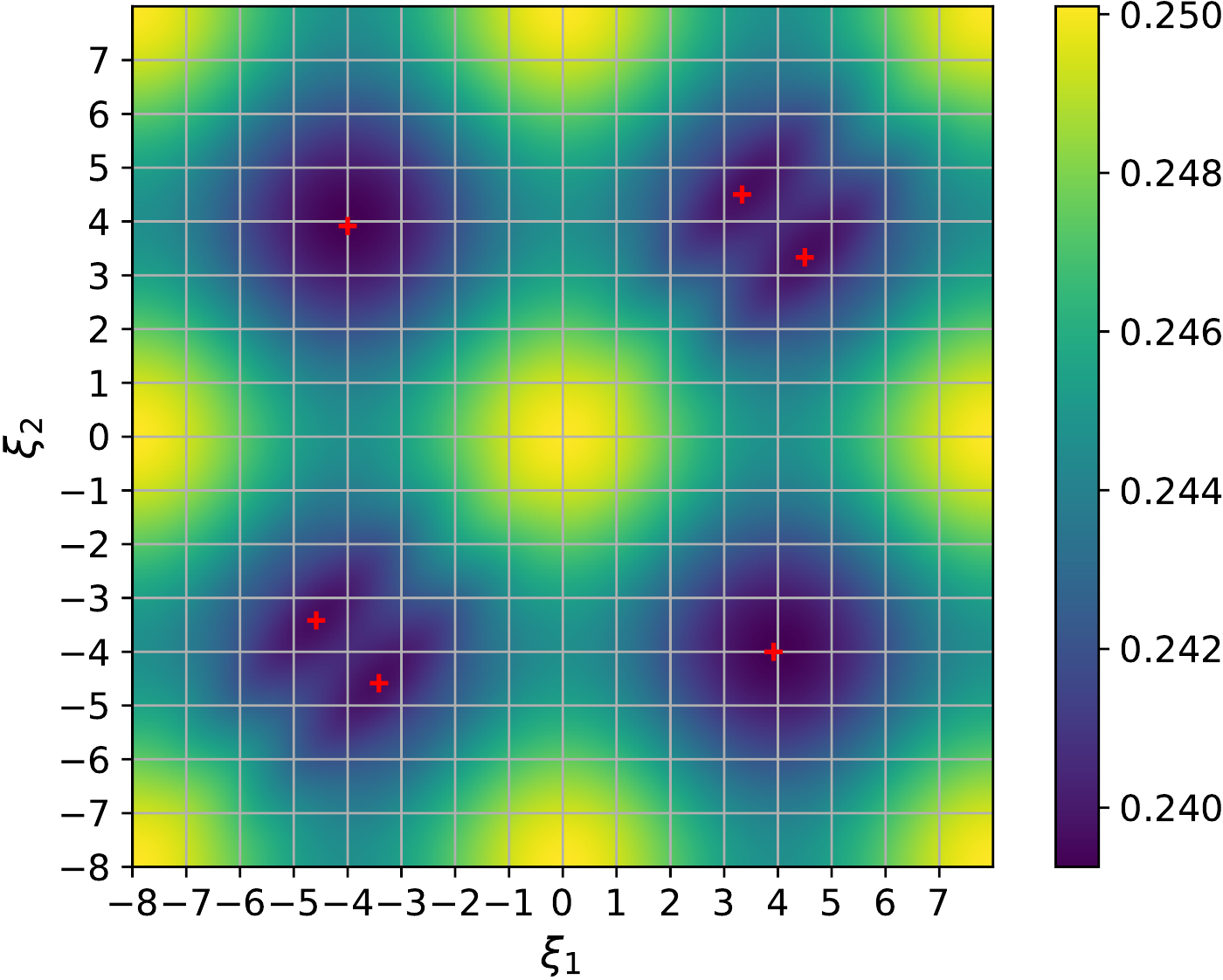} &
    \includegraphics[valign=m,width=.24\linewidth]{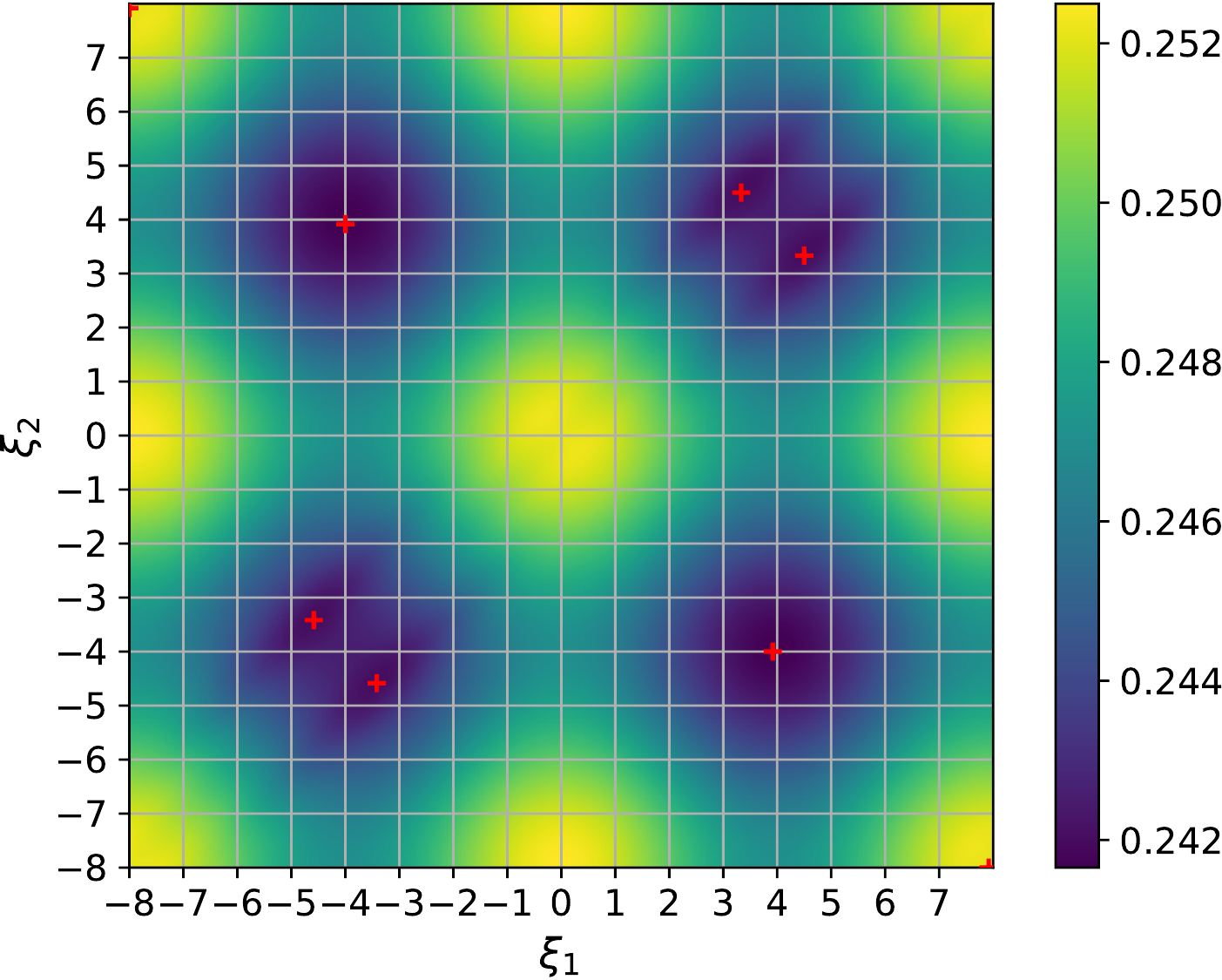} &
    \includegraphics[valign=m,width=.24\linewidth]{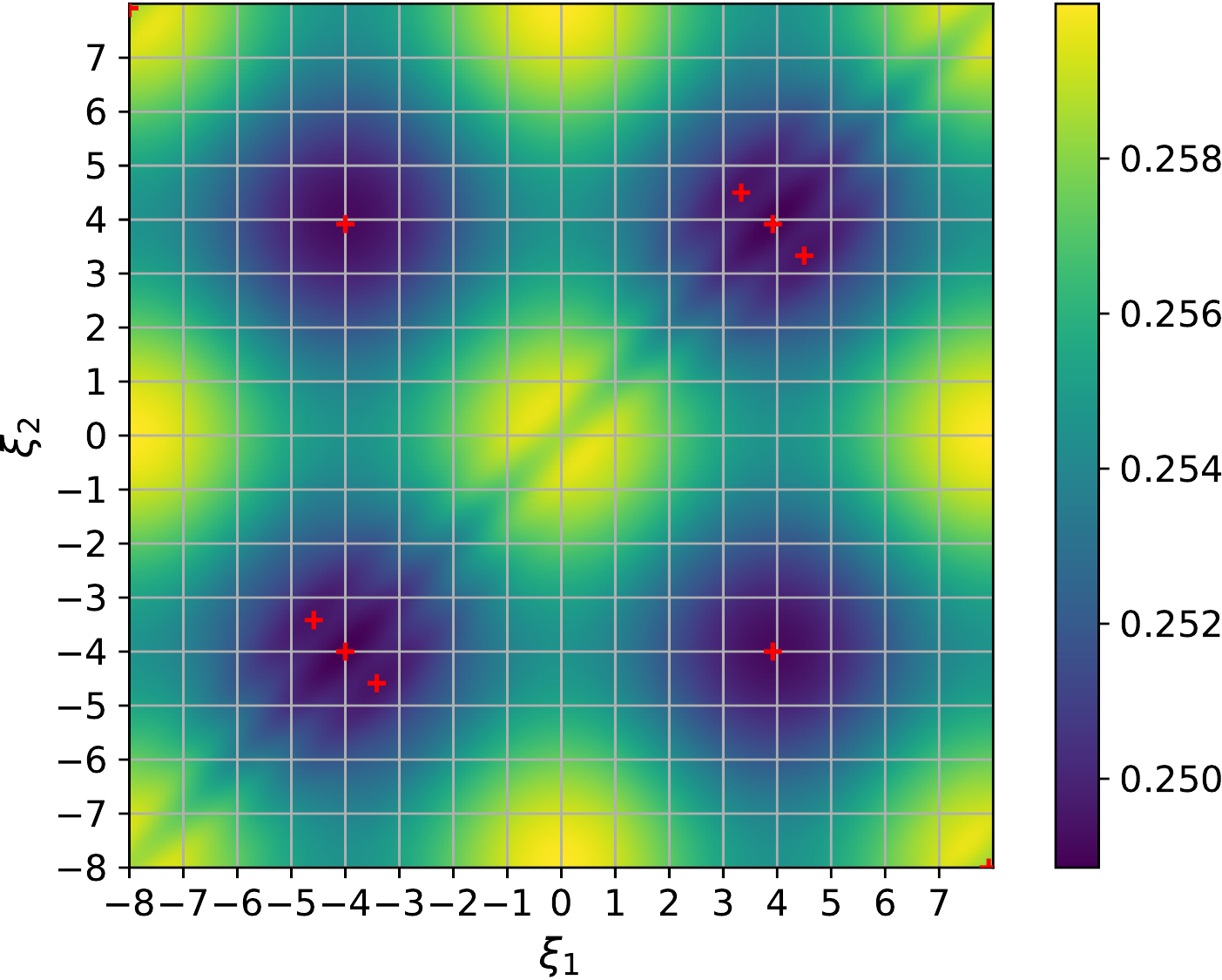} &
    \includegraphics[valign=m,width=.24\linewidth]{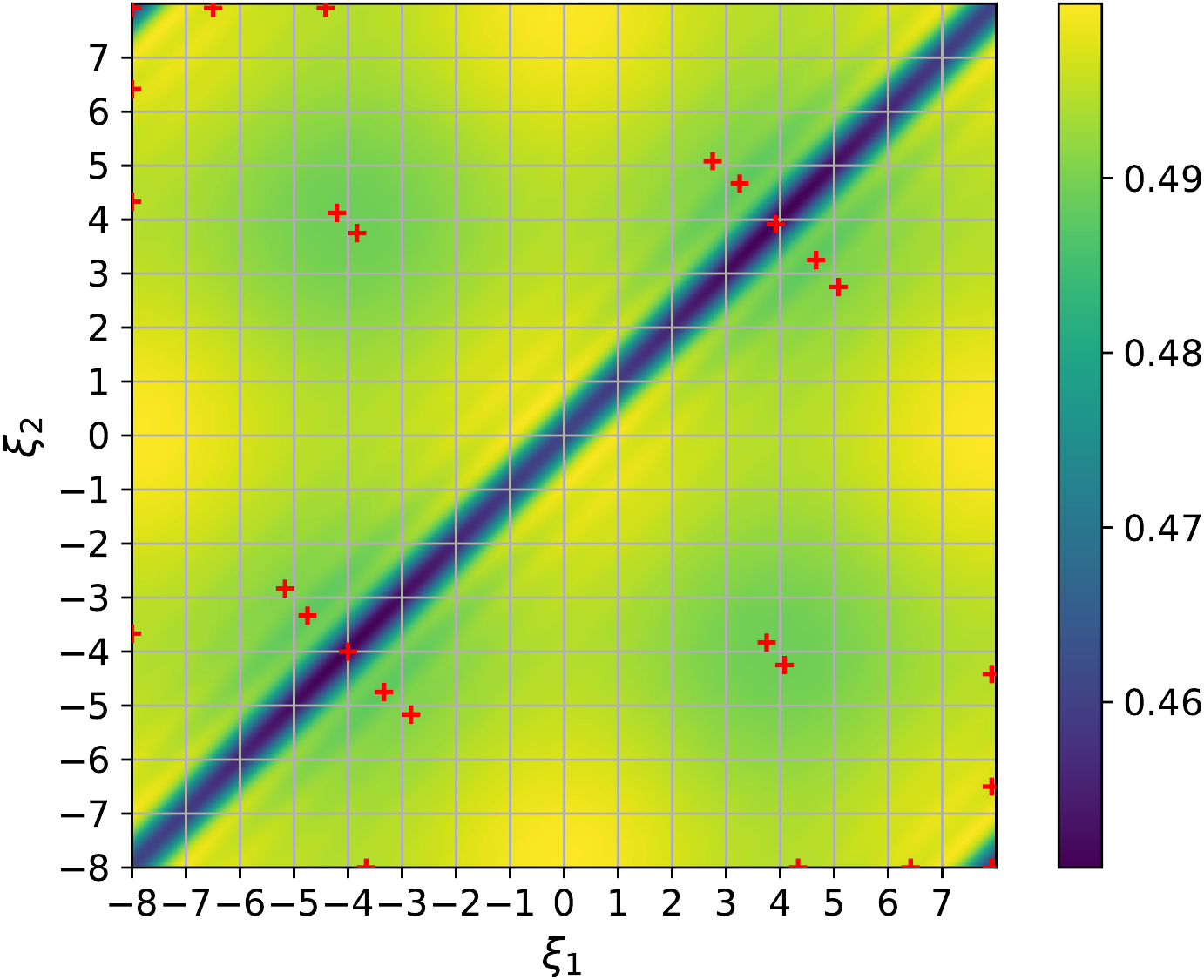} \\

    \rotatebox[origin=c]{90}{\small $\lambda=10^{0}$} &
    \includegraphics[valign=m,width=.24\linewidth]{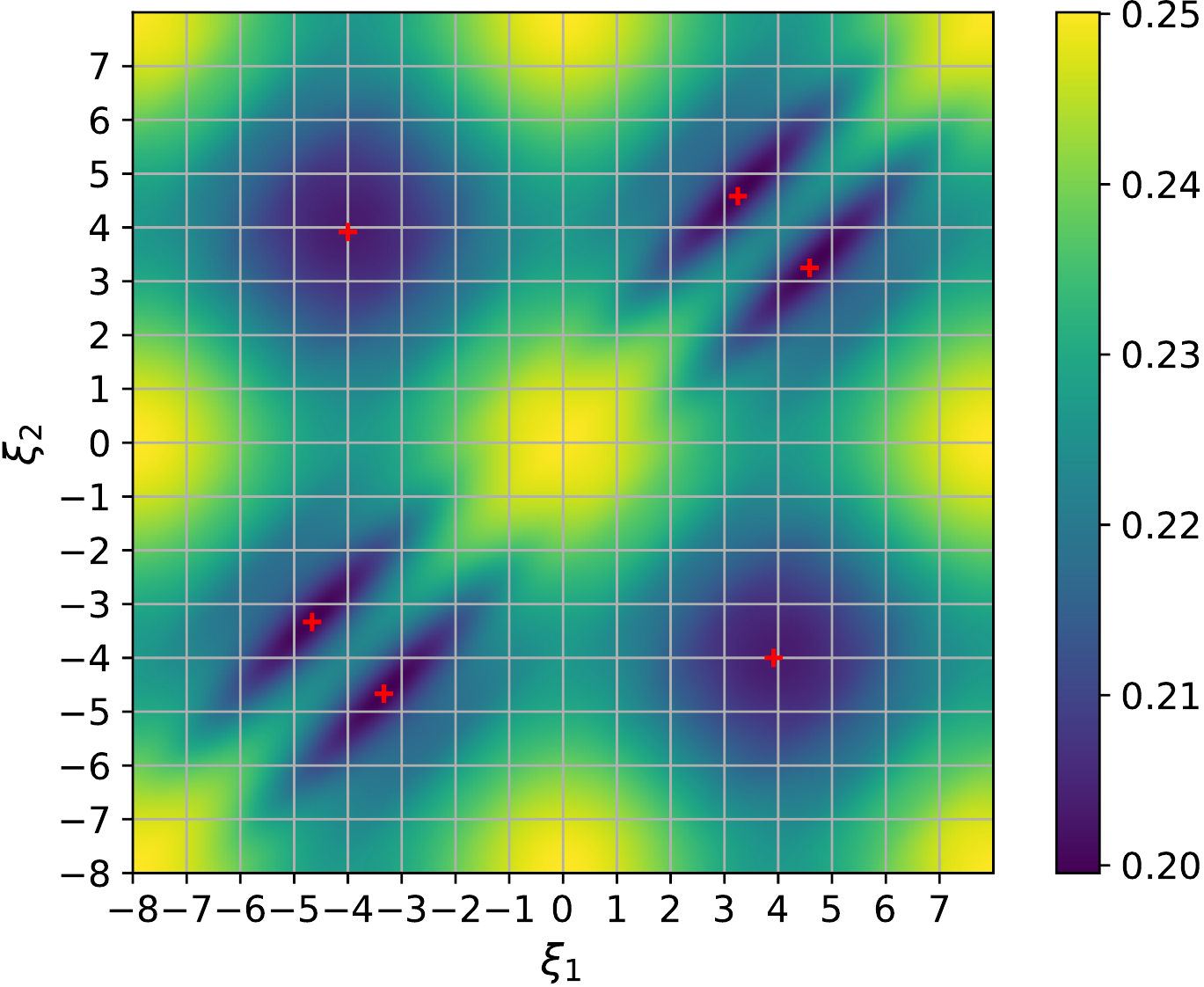} &
    \includegraphics[valign=m,width=.24\linewidth]{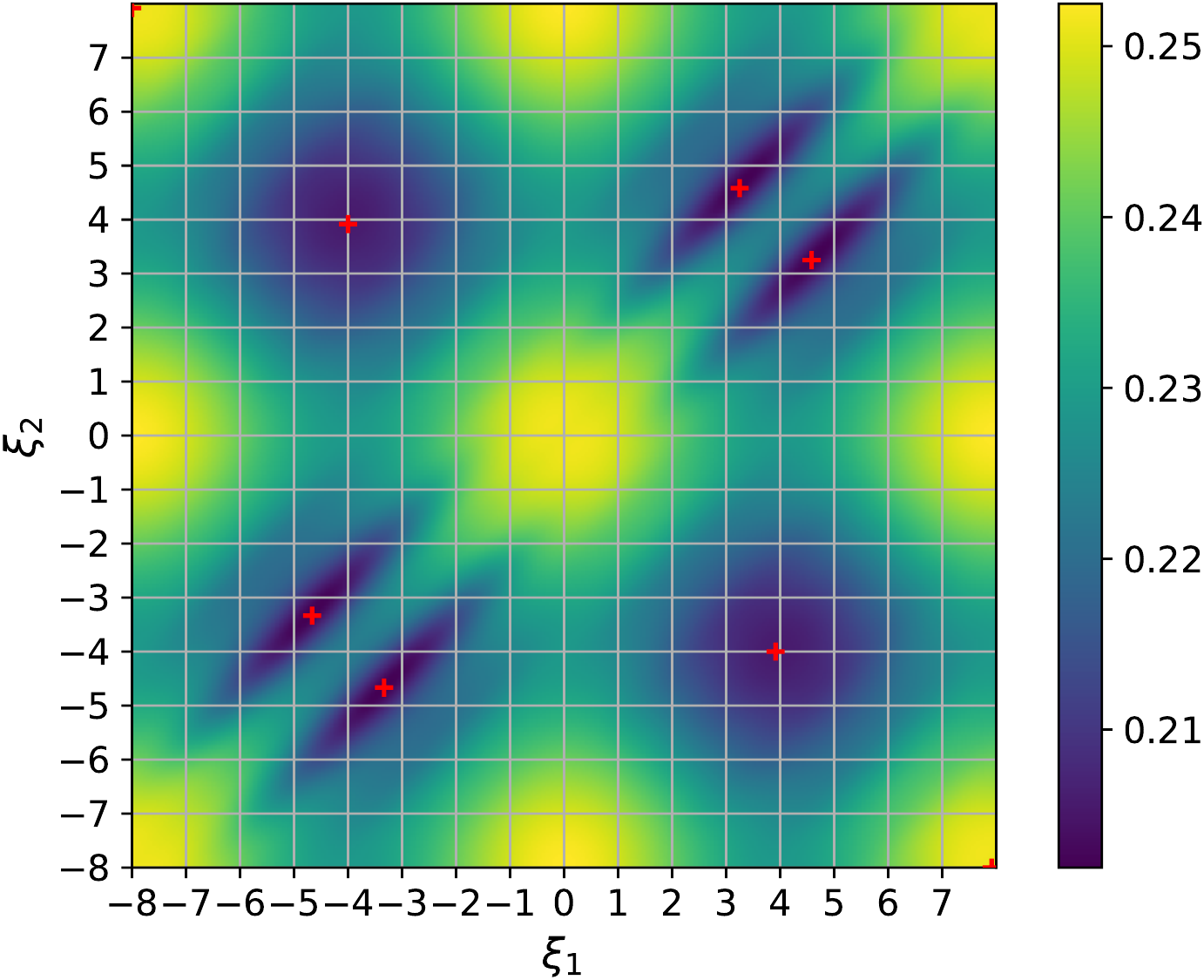} &
    \includegraphics[valign=m,width=.24\linewidth]{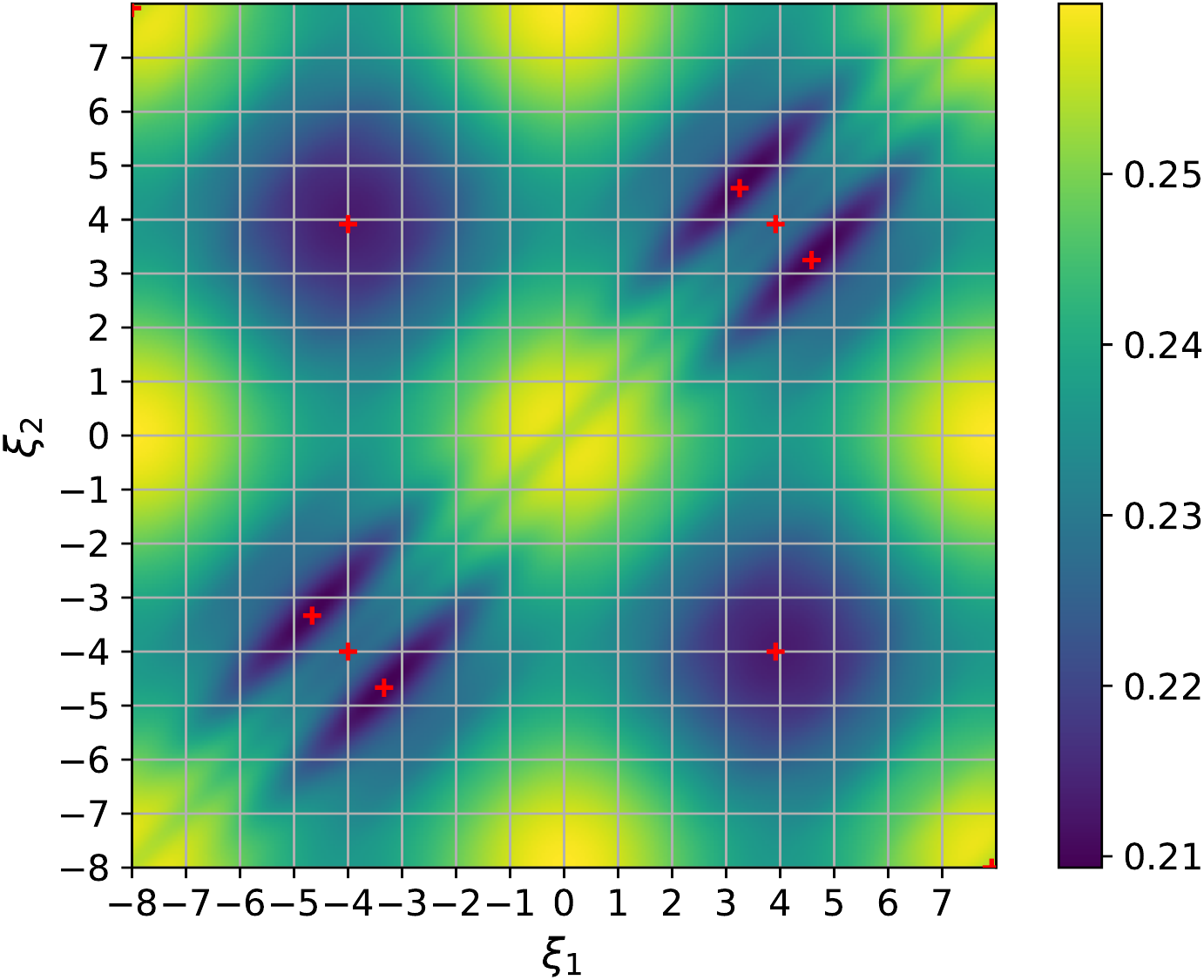} &
    \includegraphics[valign=m,width=.24\linewidth]{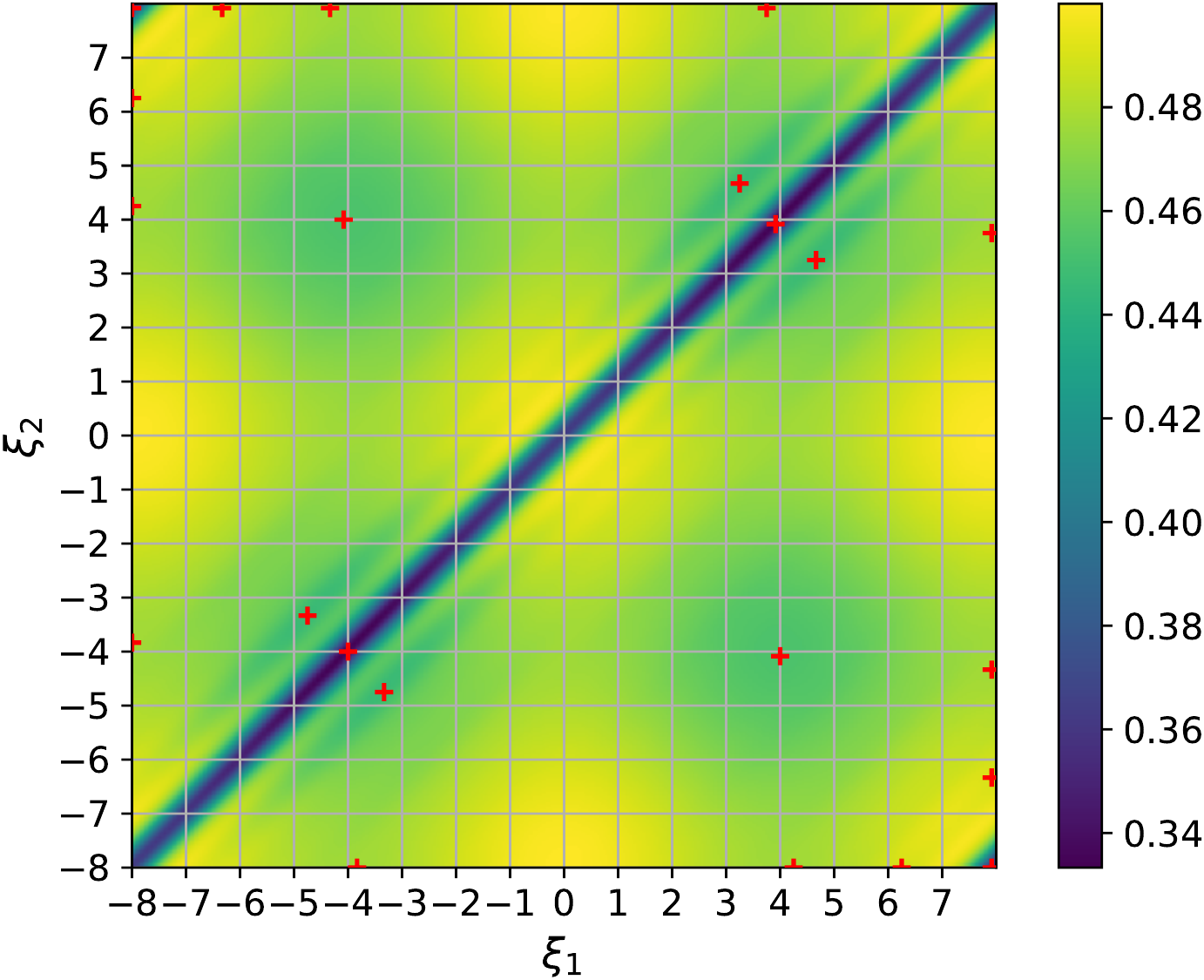} \\

    \rotatebox[origin=c]{90}{\small $\lambda=10^{-1}$} &
    \includegraphics[valign=m,width=.24\linewidth]{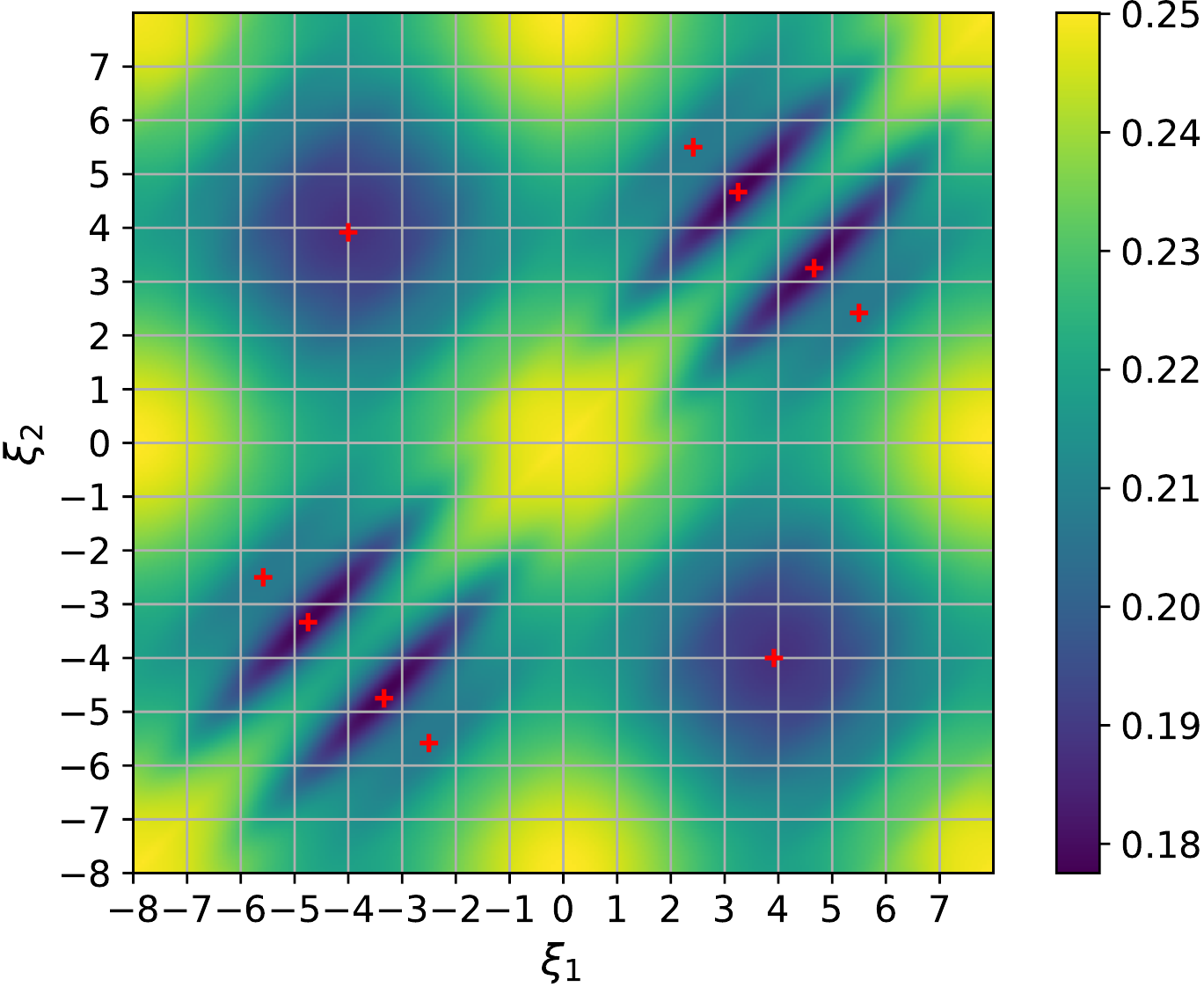} &
    \includegraphics[valign=m,width=.24\linewidth]{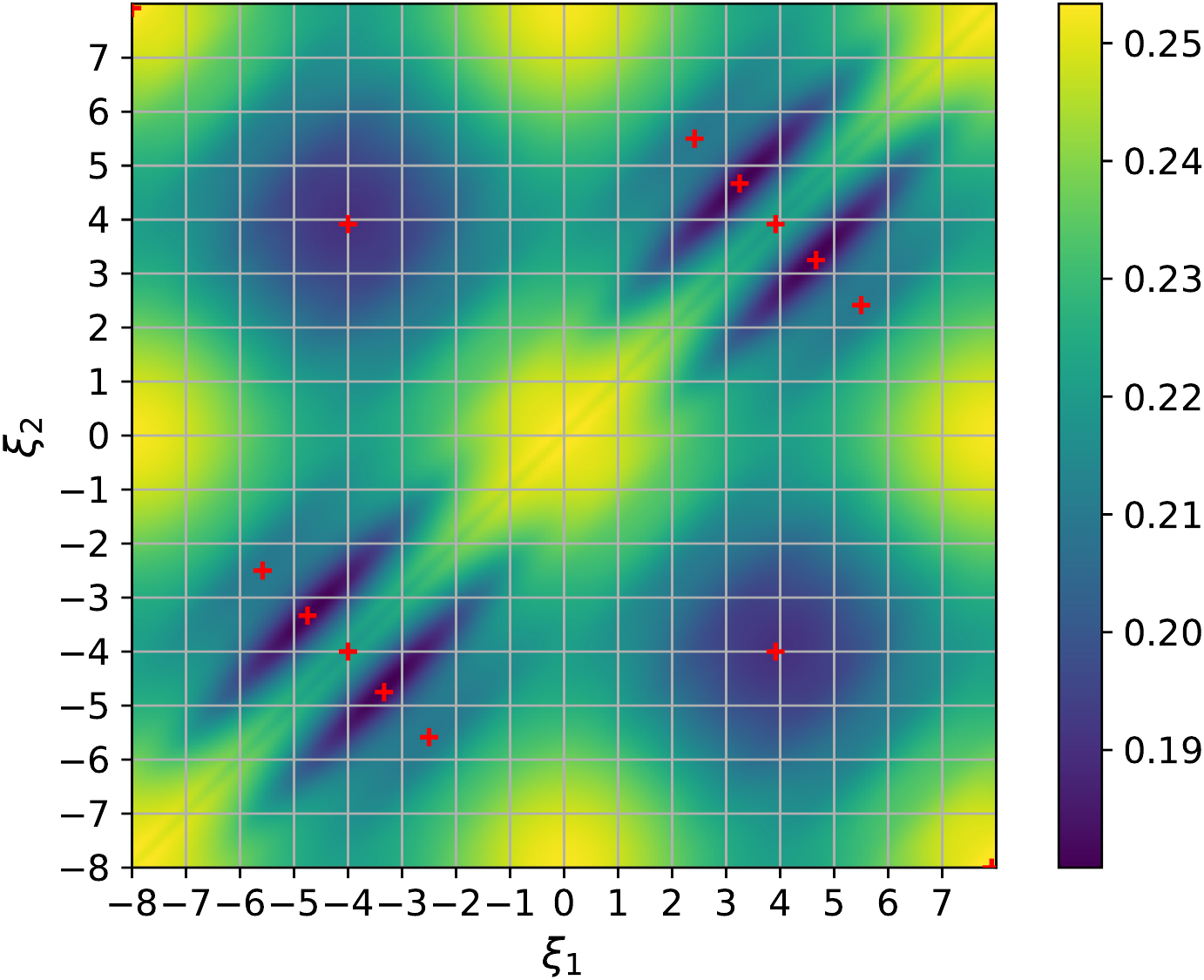} &
    \includegraphics[valign=m,width=.24\linewidth]{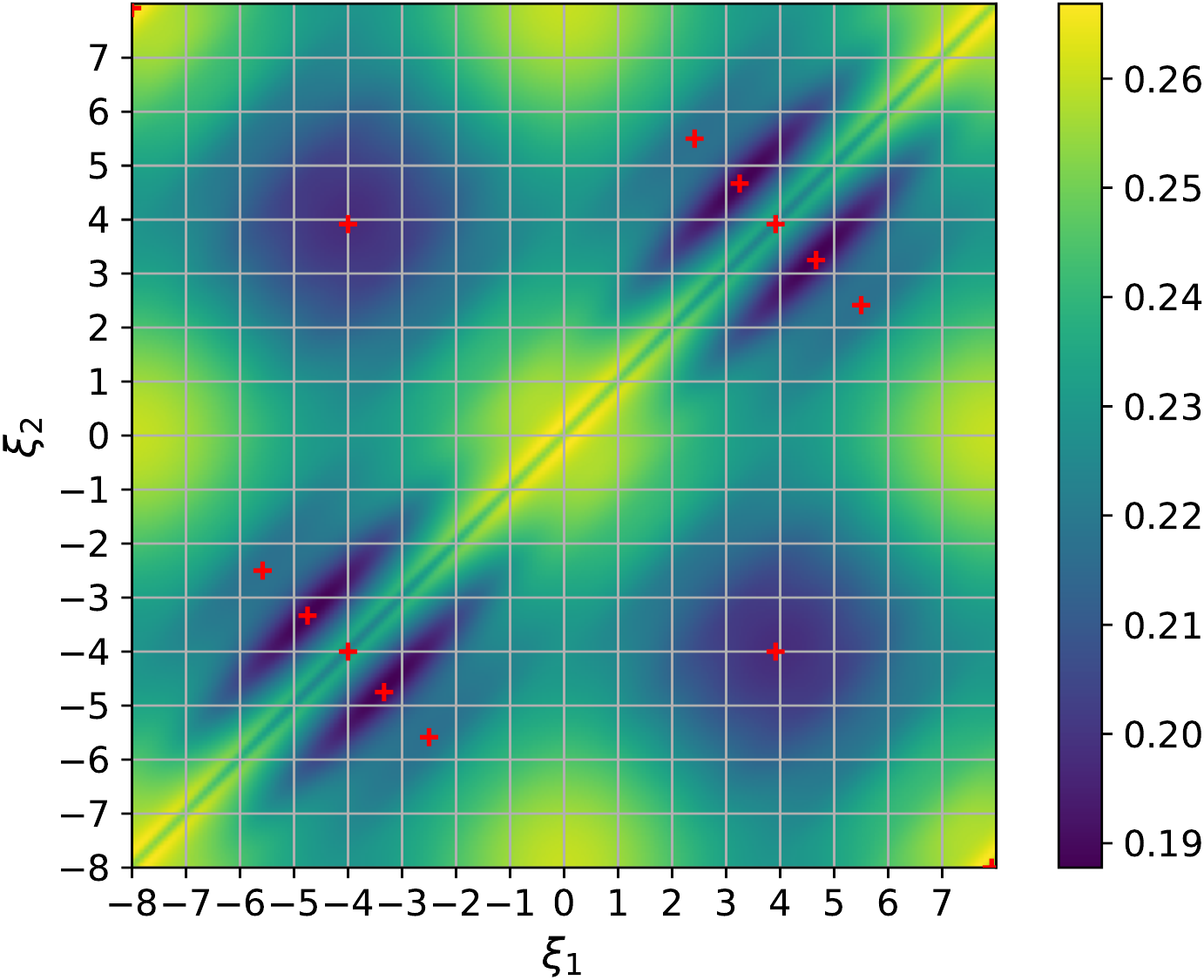} &
    \includegraphics[valign=m,width=.24\linewidth]{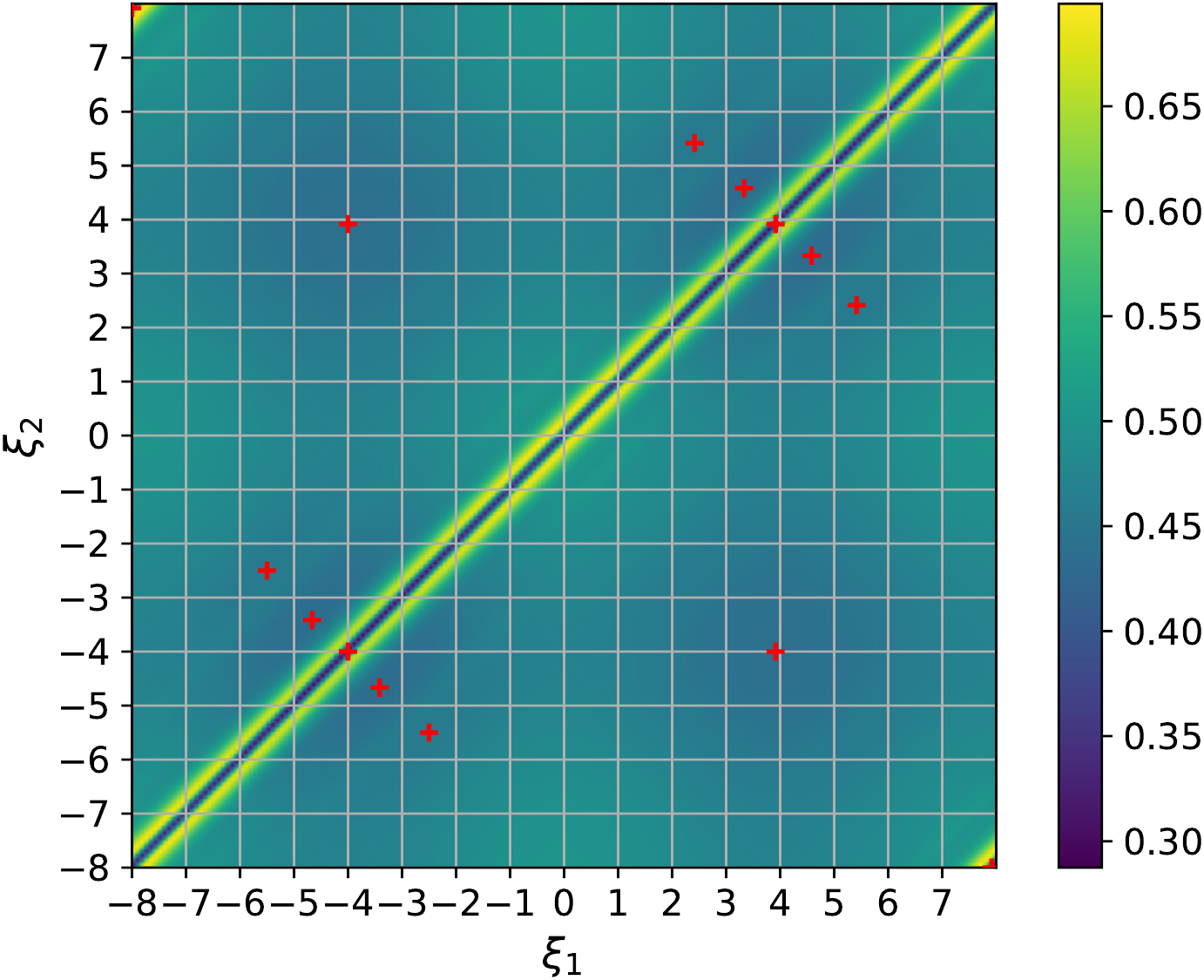} \\

    \rotatebox[origin=c]{90}{\small $\lambda=10^{-2}$} &
    \includegraphics[valign=m,width=.24\linewidth]{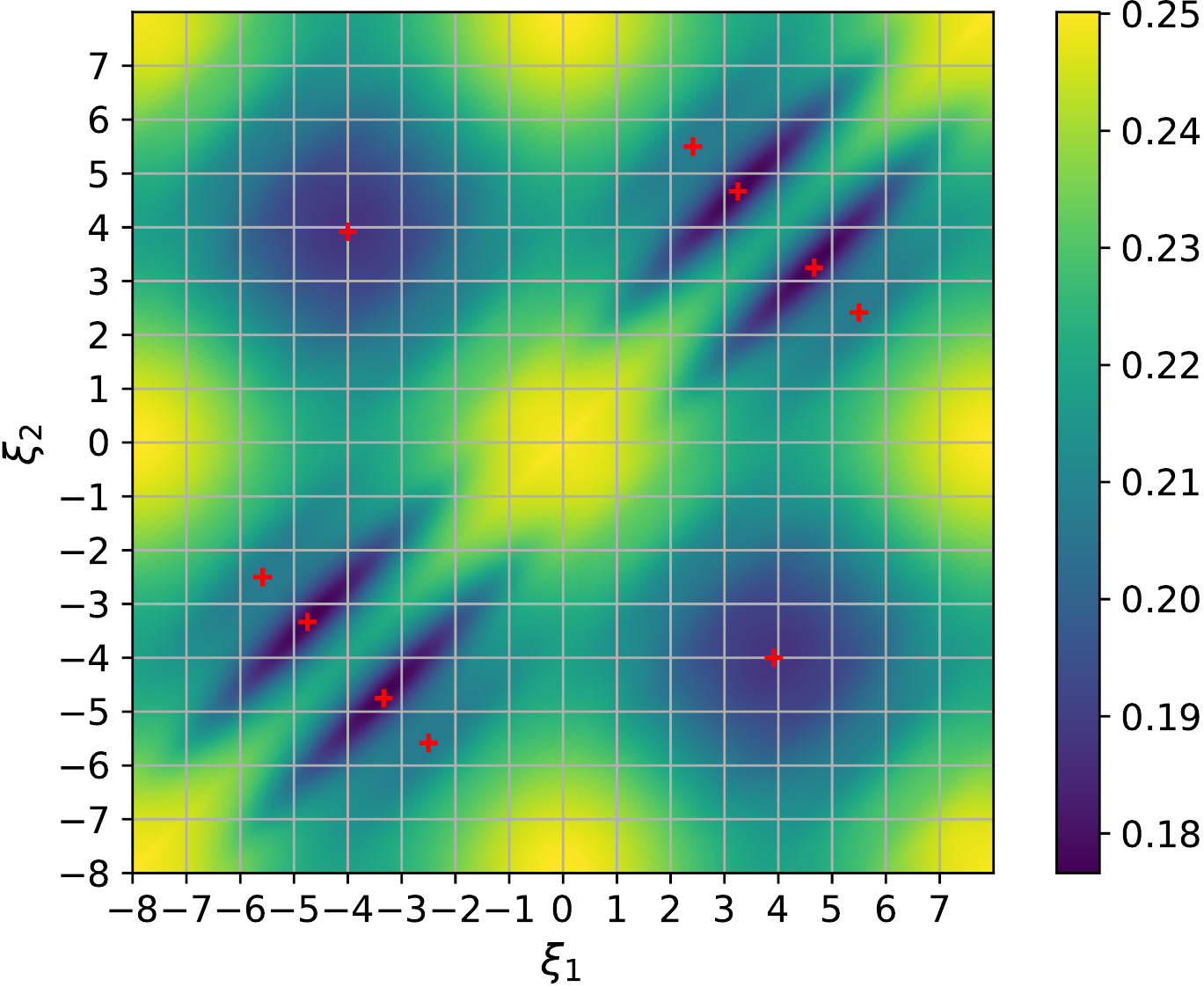} &
    \includegraphics[valign=m,width=.24\linewidth]{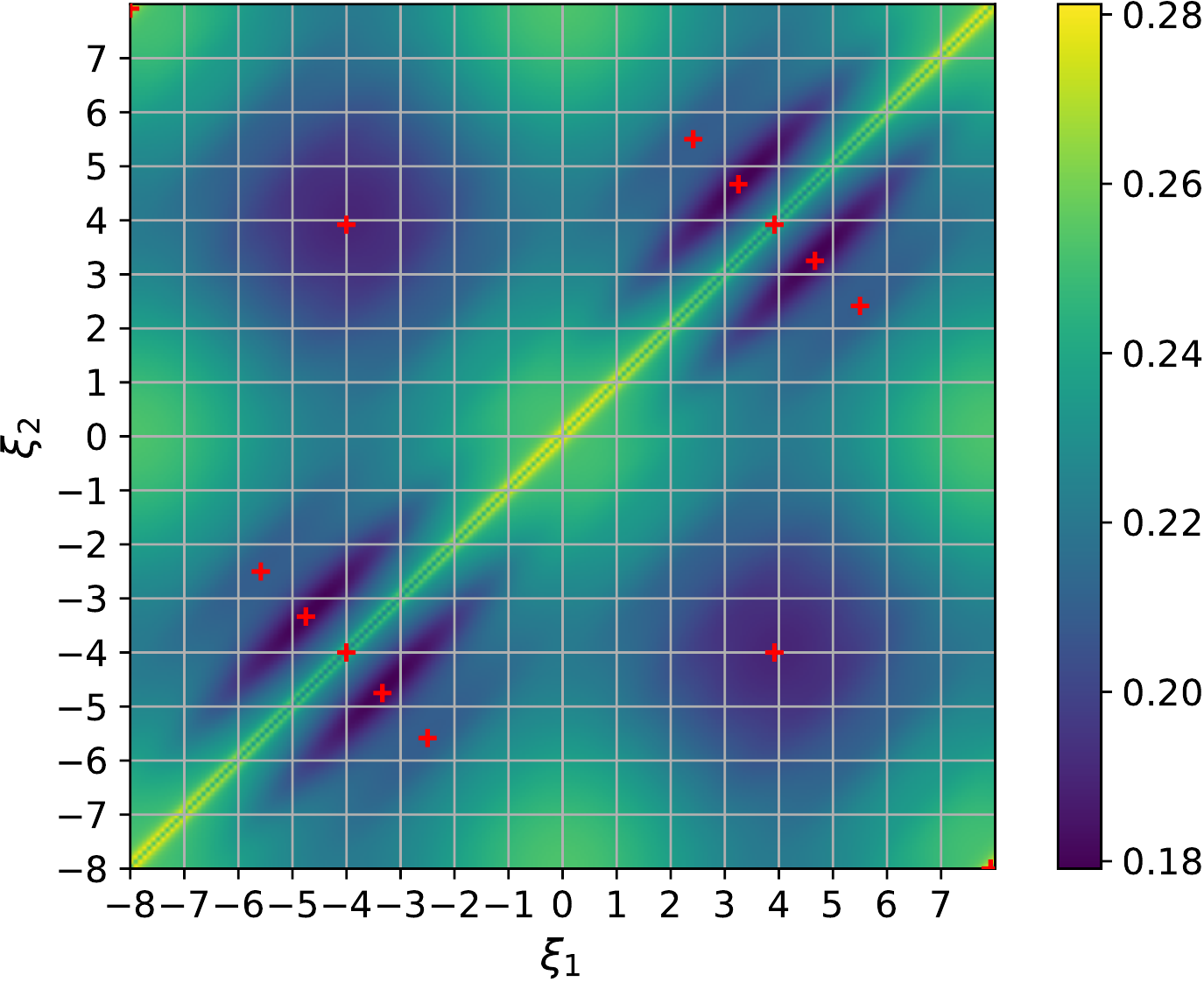} &
    \includegraphics[valign=m,width=.24\linewidth]{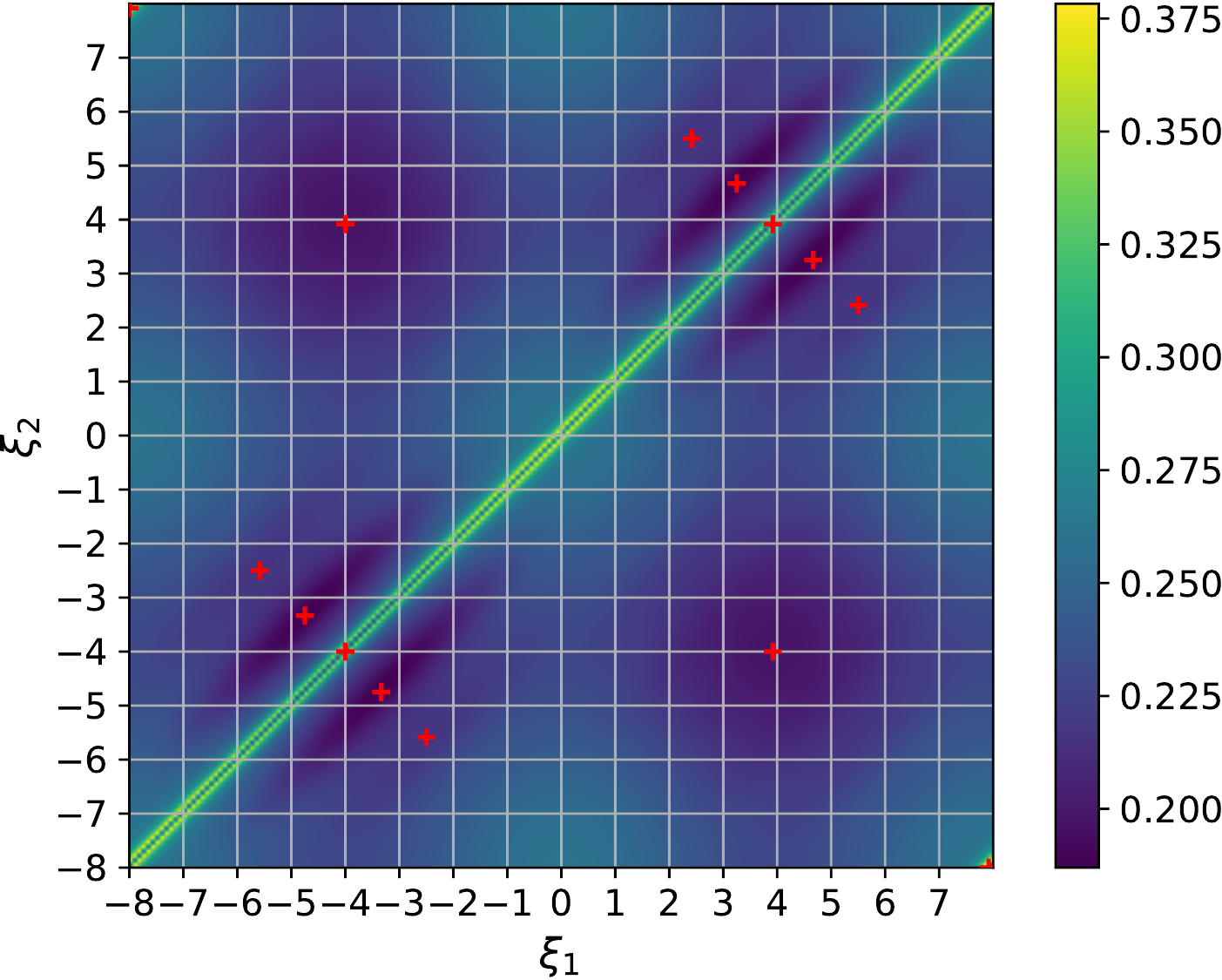} &
    \includegraphics[valign=m,width=.24\linewidth]{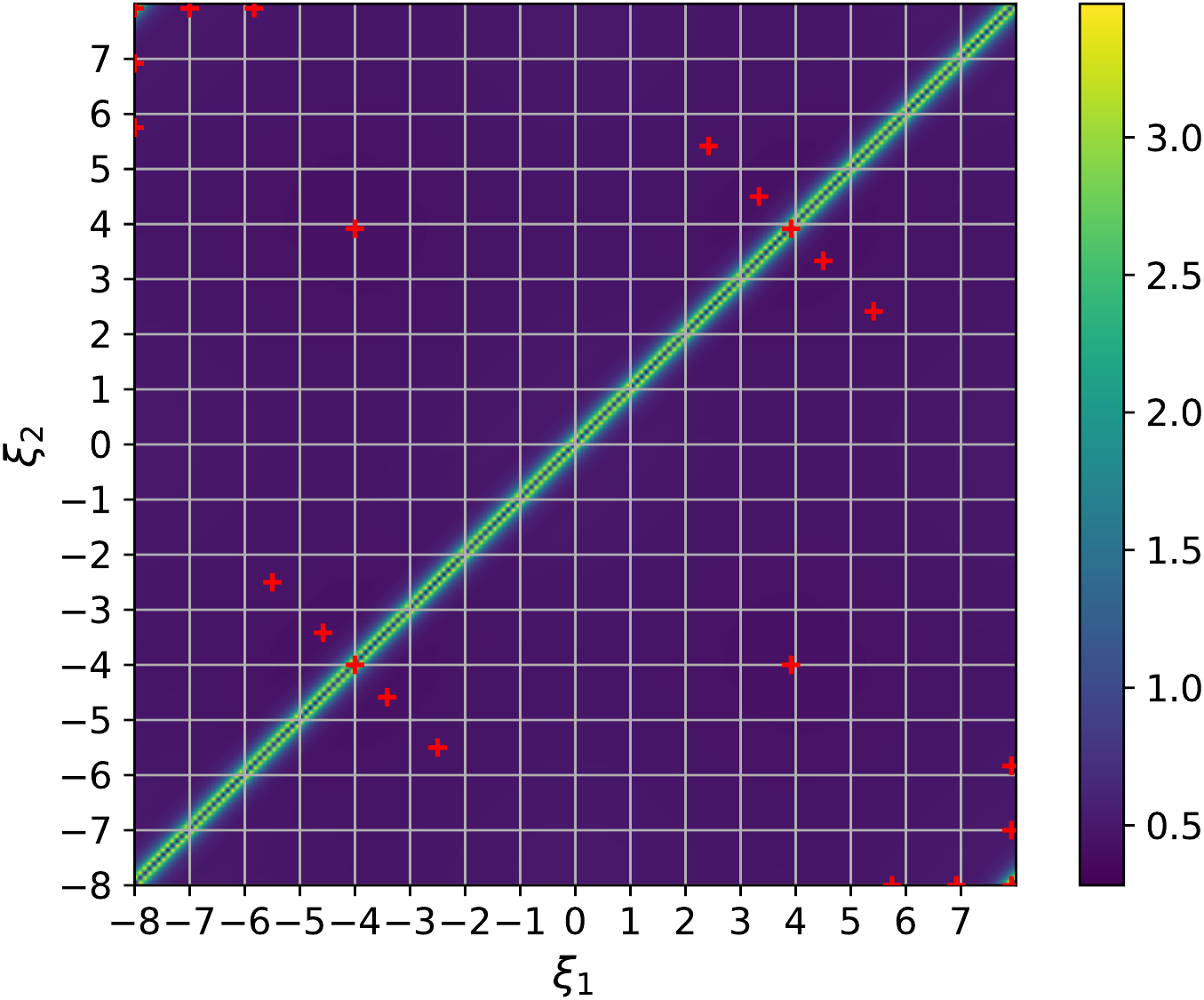} \\

    & $\sigma=1\times 10^{-2}$ &
    $\sigma=5\times 10^{-2}$ &
    $\sigma=1\times 10^{-1}$ &
    $\sigma=5\times 10^{-1}$
    \end{tabular}
    \caption{
    The energy profile of $J_3$ for $M=2$ with the low frequency signal (see center column of Fig.~\ref{fig:energy_profile_M2}) for different noise levels $\sigma$ and regularization $\lambda$.
    The red dots represent local minima.}
    \label{fig:energy_profile_noise}
\end{figure*}

\subsection{Flatness for high frequencies}\label{sec:flatness_HF}

In this paragraph we show that the partial derivatives of the cost function may vanish, for indexes corresponding to high frequencies. This explains another practical difficulty in Fourier sampling optimization: without using variable metric techniques, the sampling points located in the high frequencies move very slowly. Though our proof only applies to the function $J_1$, this effect also seems to occurs for $J_2$. See for instance the four corners of Fig.~\ref{fig:energy_profile_M2}, right.

\begin{proposition}\label{prop:expression_gradient}
Letting $r$ denote the residual error function
\begin{equation}
r(\Xi) = A(\Xi)A(\Xi)^*u-u,
\end{equation}
the gradient of the cost function $J_1$ reads:
\begin{equation}\label{eq:expression_grad_J1}
\nabla J_1(\Xi) = \Re\left(\nabla\left(\hat u(\Xi)\odot \overline{\hat r(\Xi)}\right)\right),
\end{equation}
where $\nabla$ in the right-hand-side denotes the usual derivative in 1D or the gradient in higher dimension and where $\odot$ is the coordinate-wise (Hadamard) product.
\end{proposition}
The proof of Proposition~\ref{prop:expression_gradient} is postponed to Section~\ref{sec:proof_expression_gradient}.

\begin{theorem}[Vanishing gradients for high frequencies]\label{thm:gradient_flatness}
Consider a signal $u\in\C^N$ and a point configuration $\Xi\in \R^M$. Under the decay assumptions
\begin{equation}
|\hat u(\xi)|\lesssim\frac{1}{|\xi|^\alpha}\quad\textrm{ and }\quad|\hat u'(\xi)|\lesssim\frac{1}{|\xi|^\alpha},
\end{equation}
with $\alpha>0$, we have
\begin{equation}
\left|\frac{\partial J_1(\Xi)}{\partial\xi_m}\right| \lesssim \frac{\|\hat u(\Xi)\|_1}{\md(\Xi) |\xi_m|^\alpha}.
\end{equation}
\end{theorem}
The decay assumption appear naturally in the continuous setting, when considering signals $u$ from Sobolev spaces $H^k$ with $k$ derivatives in $L^2$.

\section{Escaping the minimizers}\label{sec:solutions}

In this section we propose some solutions to mitigate the issues raised in Section~\ref{sec:main_results} and we illustrate them numerically.

\subsection{The effect of using a large dataset}\label{sec:effect_large_dataset}

In Theorem~\ref{thm:main1}, we proved existence of many local minimizers in the case $P=1$, which corresponds to a unique signal. Let us now assume that we have access to $P$ signals $u_1,\hdots,u_P$ in $\C^N$.
The analysis carried out to prove Theorem~\ref{thm:main1} can be replicated verbatim.
The only difference being that every occurence of $|\hat u|^2$ must be replaced by $\rho_P \eqdef \frac{1}{P}\sum_{p=1}^P |\hat u_p|^2$.
The function $\rho_P$ can be understood as the average power spectral density of the family $u_1,\hdots, u_P$.
As highlighted in Theorem~\ref{thm:main1}, two important factors that can create spurious minimizers are i) the number $K$ of strict maximizers of $\rho_P$ and ii) the curvature $c$ at these maximizers.

As $P$ increases, we typically expect the density $\rho_P$ to become smoother.
This effect is illustrated for a simple family of shifted and dilated rectangular functions in Fig.~\ref{fig:energy_profile_batch_1D}.
As can be seen, both the number of maxima and the curvature $c$ of $\rho_P$ in Theorem~\ref{thm:main1} decay with $P$.
For $N=128$, we display the average power spectral density for $P$ ranging from $1$ to $10^3$.
Each signal is defined by
\begin{equation}\label{eq:def_signal_square}
u[n]=\int_{n-\f12}^{n+\f12}\ind_{[a,b]}(x)\,\mathrm{d}x,
\end{equation}
where $a$ and $b$ are drawn uniformly in the range $[-N/2+1,N/2-1]$. The discrete signals are then renormalized so that $\|u\|_2=1$.

The same experiment can be reproduced in a more relevant framework from a practical viewpoint. The average power spectral density for 2D knee images of the fastMRI database \cite{zbontar2018fastmri} are represented in Fig.~\ref{fig:energy_profile_batch_2D}.
The image are of size $320\times 320$. The local maximizers are computed and displayed with red dots in Fig.~\ref{fig:energy_profile_batch_2D}.
In that case, increasing the family size $P$ reduces the number of maximizers at a slow rate.
Indeed they slightly increase from $13$k points in the case $P=1$ to $14$k in the case $P=100$ and then start to decrease to $11$k for $P=10000$. However, the curvature $c$ decays much faster.
As a conclusion, we see that \emph{using large families of signals can reduce asymptotically the number and the size of the basins of attraction of some spurious minimizers}.

\begin{figure}[h]
\centering
\includegraphics[width=\linewidth]{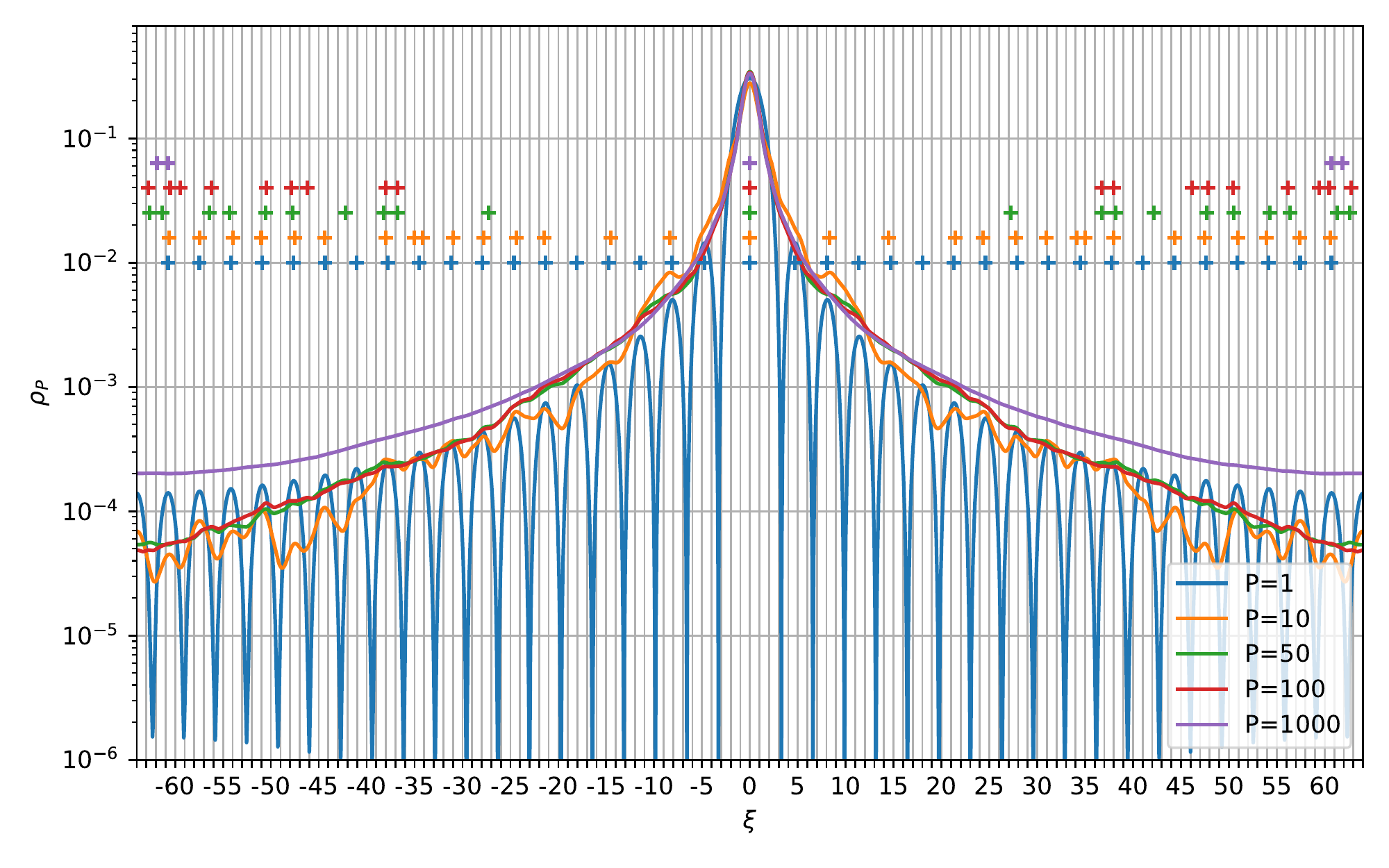}
\caption{Average power spectral density $\rho_P$ for families of rectangular functions with different sizes $P$. The dots represent local maxima of $\rho_P$ for different values of $P$.}
\label{fig:energy_profile_batch_1D}
\end{figure}

\begin{figure*}[h]
\centering
\begin{subfigure}{.325\textwidth}
  \centering
  \includegraphics[width=\linewidth]{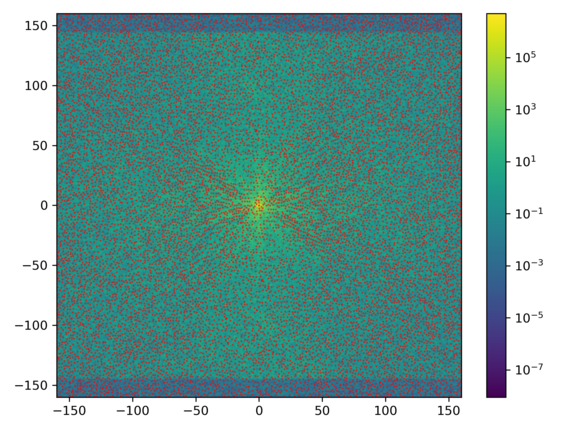}
  \caption{$P=10^0$ -- $13746$ maxima}
\end{subfigure}
\begin{subfigure}{.325\textwidth}
  \centering
  \includegraphics[width=\linewidth]{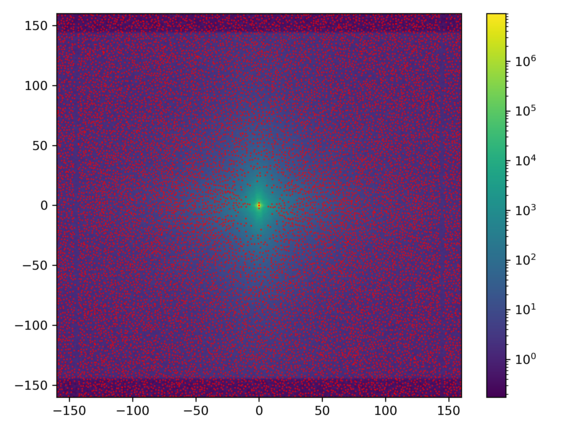}
  \caption{$P=10^2$ -- $14888$ maxima}
\end{subfigure}
\begin{subfigure}{.325\textwidth}
  \centering
  \includegraphics[width=\linewidth]{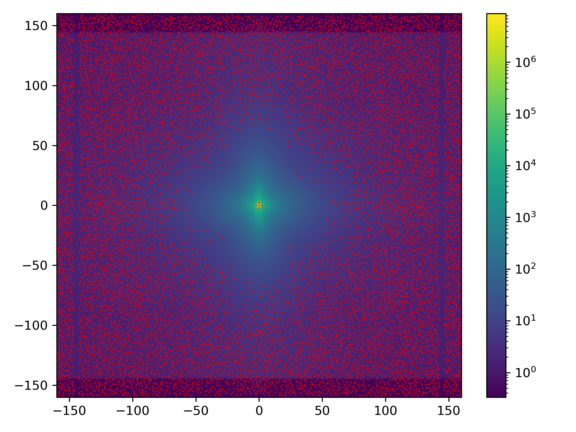}
  \caption{$P=10^4$ -- $11592$ maxima}
\end{subfigure}
\caption{Average power spectral density $\rho_P$ for a subset of images from the knee dataset of fastMRI. The image size is $N=320$ and the red dots represent local maximizers.}
\label{fig:energy_profile_batch_2D}
\end{figure*}

\subsection{Stochastic gradient descent}

When using a large family of signals, the cost function \eqref{eq:function_sum} naturally lends itself to the use of stochastic gradient descents (SGD), see \cite{wang2021b,weiss2019pilot} that address large MRI datasets.
Contrarily to a deterministic gradient descent, which is known to converge to critical points under mild regularity conditions, the stochastic gradient with a fixed step size does not converge. The method is known to end up frolicking in the neighborhood of local critical points \cite{bottou2010large}. The radius of the neighborhood depends on the stochastic gradient variance and on the step-size. Intuitively, \emph{using stochastic gradients algorithms should therefore allow escaping local minimizers}. We will showcase this effect in the forthcoming numerical experiments.

\subsection{Variable metric}\label{sec:variable_metric}

In Section~\ref{sec:flatness_HF}, Theorem~\ref{thm:gradient_flatness} states that the gradient of $J_1$ might vanish in the high frequency domain. Using second order information is a well known remedy to mitigate this effect. 
In this work, we propose a simple method which corresponds to a variable diagonal metric with well-chosen coefficients. 

As shown in Theorem~\ref{thm:gradient_flatness}, the gradient vanishes with a rate depending on the Fourier transform magnitude $|\hat u|$. For a dataset, this decay is somewhat captured by the average power spectral density $\rho_P(\xi)\eqdef\frac{1}{P}\sum_{p=1}^P |\hat u_p(\xi)|^2$. Hence, we propose to compute $\rho_P$ once and for all on a fine grid ($20\times N$ discretization points in our example). The function $\rho_P$ is then linearly interpolated in between the grid points during the gradient descent. At each gradient iteration we replace $\frac{\partial J_1(\Xi)}{\partial\xi_m}$ by 
\begin{equation}\label{eq:variable_metric}
\frac{1}{\rho_P(\xi_m)^\beta}\frac{\partial J_1(\Xi)}{\partial\xi_m},
\end{equation}
where $\beta$ is a constant that has to be set empirically. From numerical experiments $\beta\in[1,2]$ shows good performance. In all the experiments presented hereafter we use $\beta=1$.
We will see later in the numerical experiments, that \emph{this variable metric significantly accelerates the convergence for sampling points located in high frequencies}.

\subsection{Numerical illustrations}

In this section, we aim at illustrating numerically the different results established previously.
We aim at reconstructing 1D signals of size $N=128$ from $M=64$ measurements in the Fourier domain.
We suppose that $P$ rectangular signals generated using \eqref{eq:def_signal_square} are given.
We illustrate our findings with the back-projection reconstructor associated to the cost function $J_1$, but similar results have been obtained with the pseudo-inverse.
As we are working in 1D with small dimensions $N$ and $M$, at each iteration, the whole matrix $A(\Xi)^*$ is evaluated and the gradient $\nabla J_1$ is computed directly from the analytic expression \eqref{eq:expression_grad_J1}.
We first use a fixed step gradient descent algorithm in order to showcase the convergence dynamics of the algorithm.
The initialization of $\Xi$ is a subgrid with a constant spacing of $2$.
The following experiments are conducted:

\paragraph{Effect of the dataset size $P$}
We first vary the number of signals by taking $P=1$ and $P=1000$.
The evolution of $\Xi$ is displayed in Fig.~\ref{fig:batch_size}, respectively top-left and top-center.
The history of the cost function is given in Fig.~\ref{fig:sto}.
For this experiment, we expect that a good sampling scheme consists of low frequencies sampled at the Shannon-Nyquist rate.
In this regard, the sampling scheme obtained in Fig.~\ref{fig:batch_size} for $P=1000$ is more satisfactory than the one obtained for $P=1$. In the case $P=1000$, the displacement of $\Xi$ is more important, suggesting that some local minima have been discarded.

\paragraph{Variable metric}
We then study, for $P=1000$ the effect of a variable metric gradient descent as described in Section~\ref{sec:variable_metric}.
We also compare this approach to an L-BFGS algorithm \cite{goldfarb1970family} with a line search and with a Hessian estimated using the last $8$ gradients.
In Fig.~\ref{fig:batch_size}, the usual gradient algorithm is at the top-center, the variable metric gradient descent is at the bottom-center and the L-BFGS algorithm is at the bottom-left.
The cost function evolution is displayed in Fig.~\ref{fig:sto}.
Using a variable metric results in a huge speed-up of the algorithm.
This is particularly visible for points $\xi$ located at high frequencies, which is another illustration of Theorem~\ref{thm:gradient_flatness}.
For this example, the L-BFGS algorithm converges slightly faster than the variable metric gradient descent in the early iterations. However, its per-iteration cost is much higher since it uses a line search and a non diagonal metric.
Since the L-BFGS algorithm can be seen as a state-of-the-art quasi-Newton method, the proposed empirical metric \eqref{eq:variable_metric} seems remarkably efficient.

\paragraph{Stochastic gradient descent}
Finally in Fig.~\ref{fig:batch_size} right column, we investigate the use of a fixed-step stochastic gradient descent algorithm with a batch size of $1$. In that experiment, a new random signal is generated at every iteration using the model \eqref{eq:def_signal_square} and the stochastic gradient is computed with respect to that signal only.
The trajectory of the vanilla SGD is comparable with the one obtained using a deterministic gradient descent for $P=1000$ in Fig.~\ref{fig:batch_size} top-center.
The variable metric trick significantly improves the convergence speed and more importantly, the final points configuration.
As a conclusion, the variable metric SGD algorithm seems to be able to escape spurious minimizers and to take advantage of the averaging effect of the large dataset without the struggle of computing the gradient over a large dataset.

\paragraph{Comparison of the sampling schemes}
The final sampling schemes are not directly comparable in terms of cost function because the objective function is computed over different datasets.
In Table~\ref{tab:optim_1D_cost}, we therefore report the cost function computed on a specific set of signals.
This set contains the $P=1000$ signals that are used in the numerical illustrations of Fig.~\ref{fig:batch_size} center column.
When tested against a large dataset, the final configuration obtained for $P=1$ seems highly sub-optimal.
This effect is most likely due to a convergence to a local minimizer and also to the fact that the sampling scheme is not adapted to a whole family but only to a single signal.
The remarkable observation that can be made from Table~\ref{tab:optim_1D_cost} is that the optimal configuration obtained with the variable metric SGD performs better on the dataset of $P=1000$ signals than the experiment conducted in Fig.~\ref{fig:batch_size} which is taylored for this dataset.
This shows that the the usual deterministic algorithms are stuck in local minima even with large datasets.
On the contrary, the variable metric SGD algorithm seems effective.

These numerical results highlight the effectiveness of the different tricks suggested in this section: the use of a variable metric to handle high frequencies and a stochastic optimization to avoid local minima.

\begin{figure*}[ht]
\centering
\begin{subfigure}[t]{0.325\textwidth}
	\includegraphics[width=\linewidth]{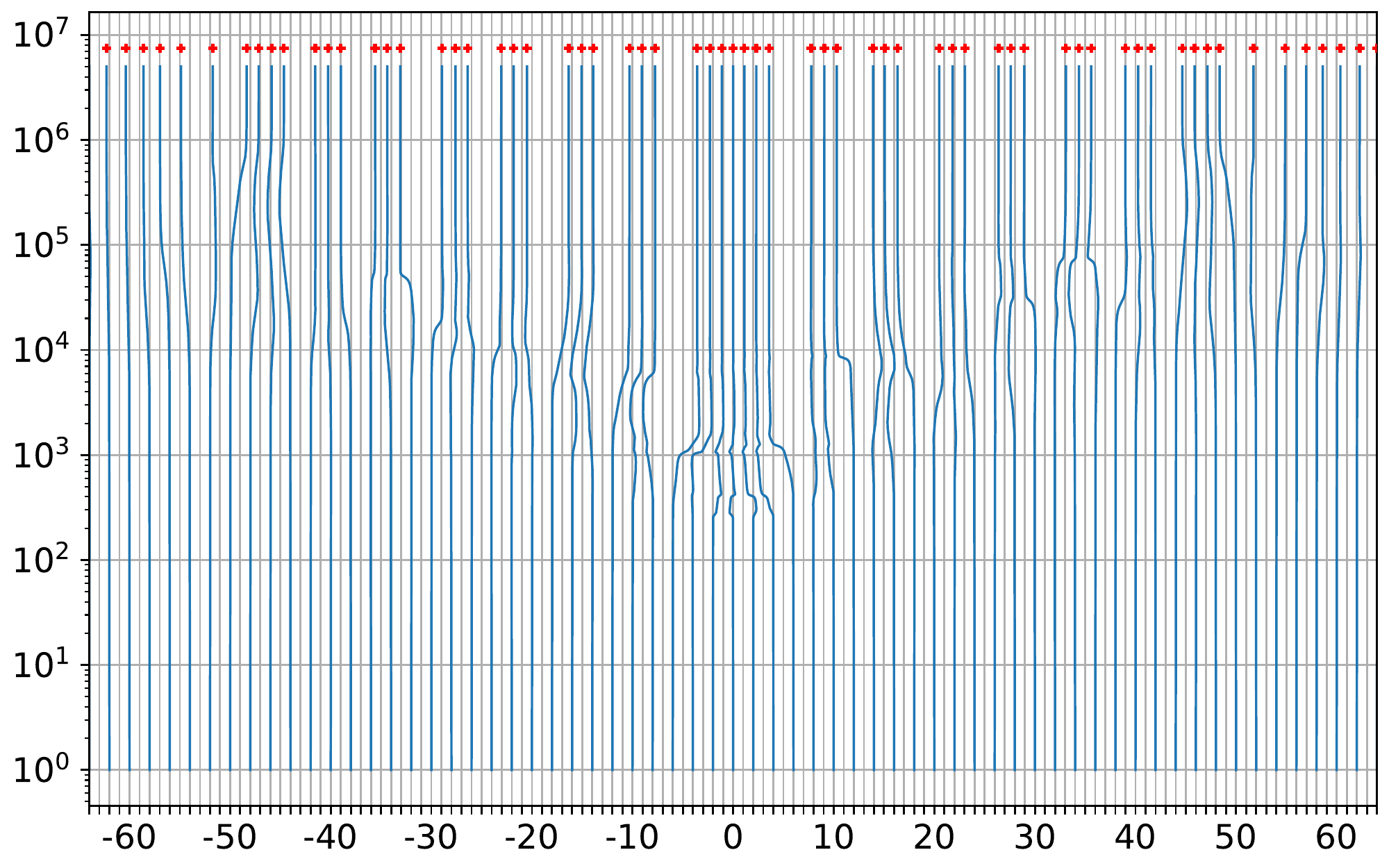}
	\caption{$P=1$ -- vanilla gradient descent}
\end{subfigure}
\begin{subfigure}[t]{0.325\textwidth}
	\includegraphics[width=\linewidth]{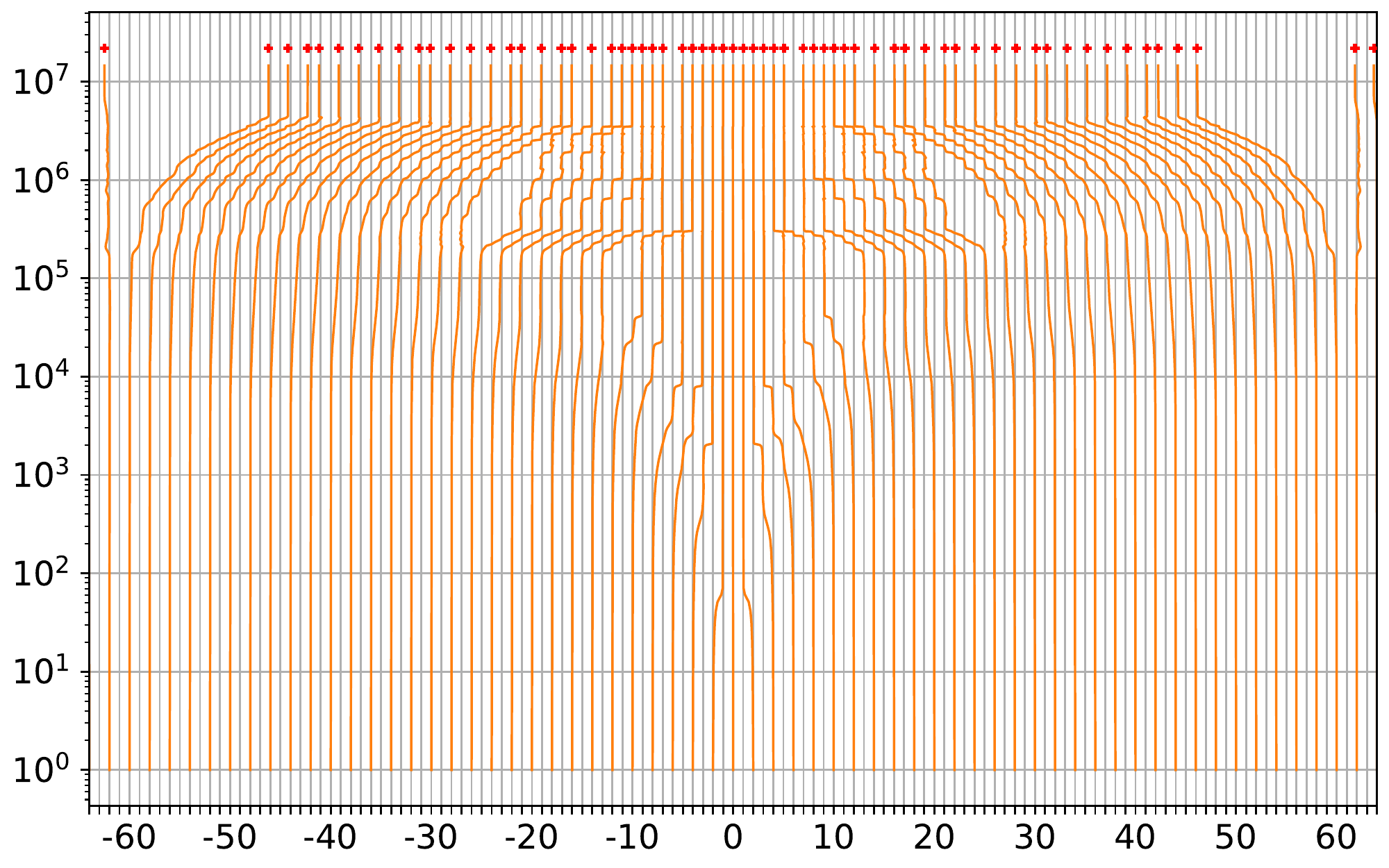}
	\caption{$P=1000$ -- vanilla gradient descent}
\end{subfigure}
\begin{subfigure}[t]{0.325\textwidth}
	\includegraphics[width=\linewidth]{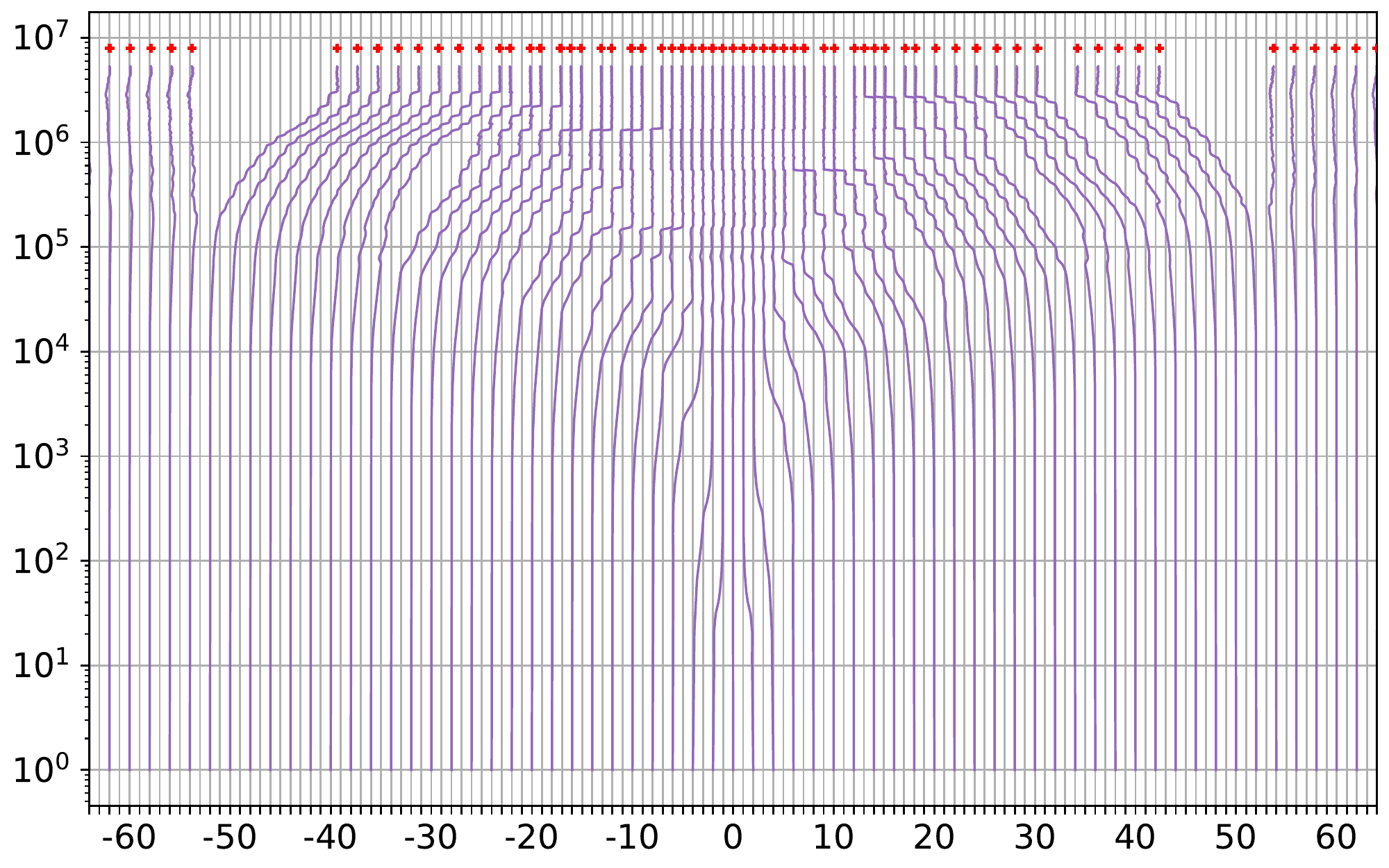}
	\caption{vanilla SGD}
\end{subfigure}
\begin{subfigure}[t]{0.325\textwidth}
	\includegraphics[width=\linewidth]{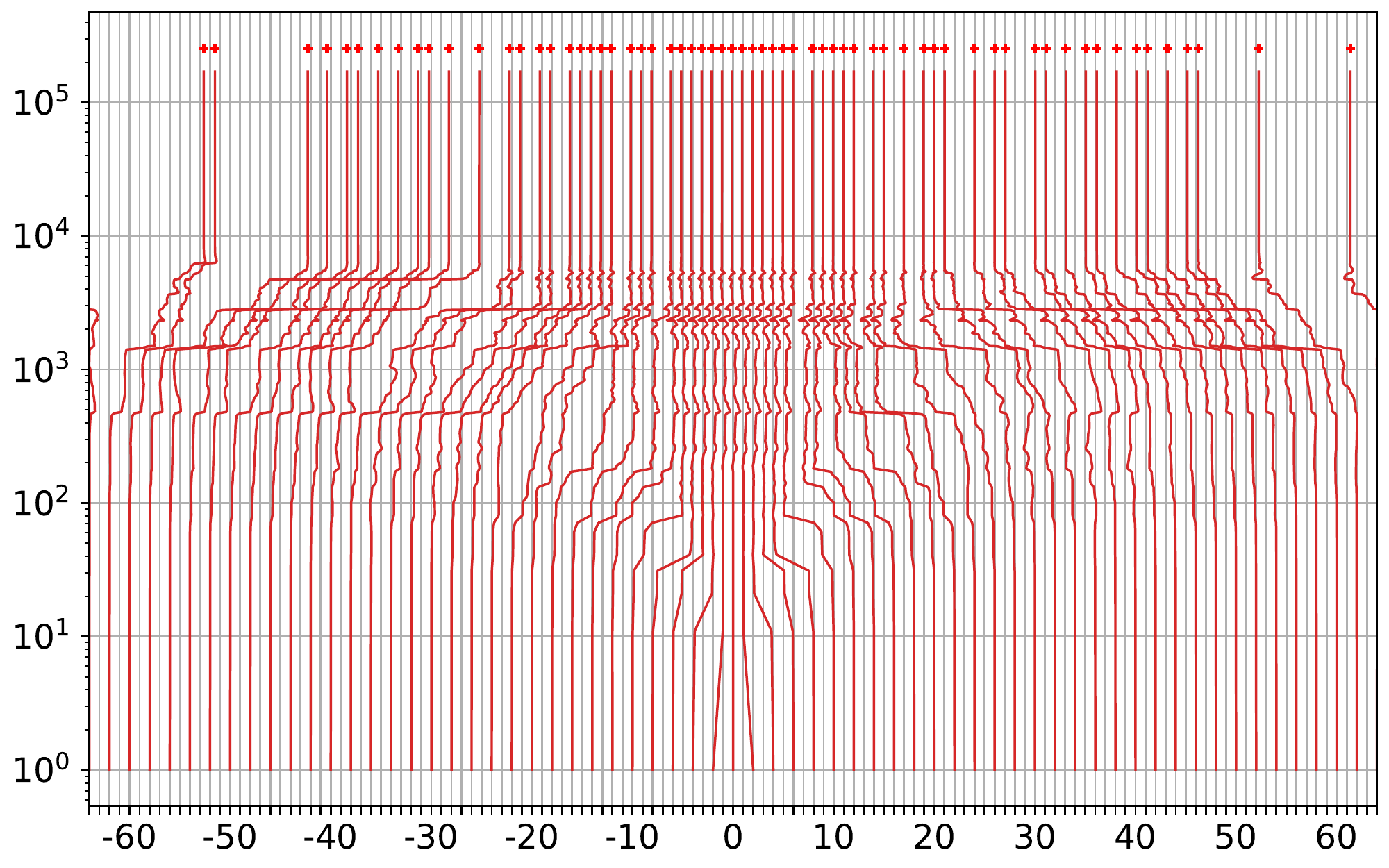}
	\caption{$P=1000$ -- L-BFGS}
\end{subfigure}
\begin{subfigure}[t]{0.325\textwidth}
	\includegraphics[width=\linewidth]{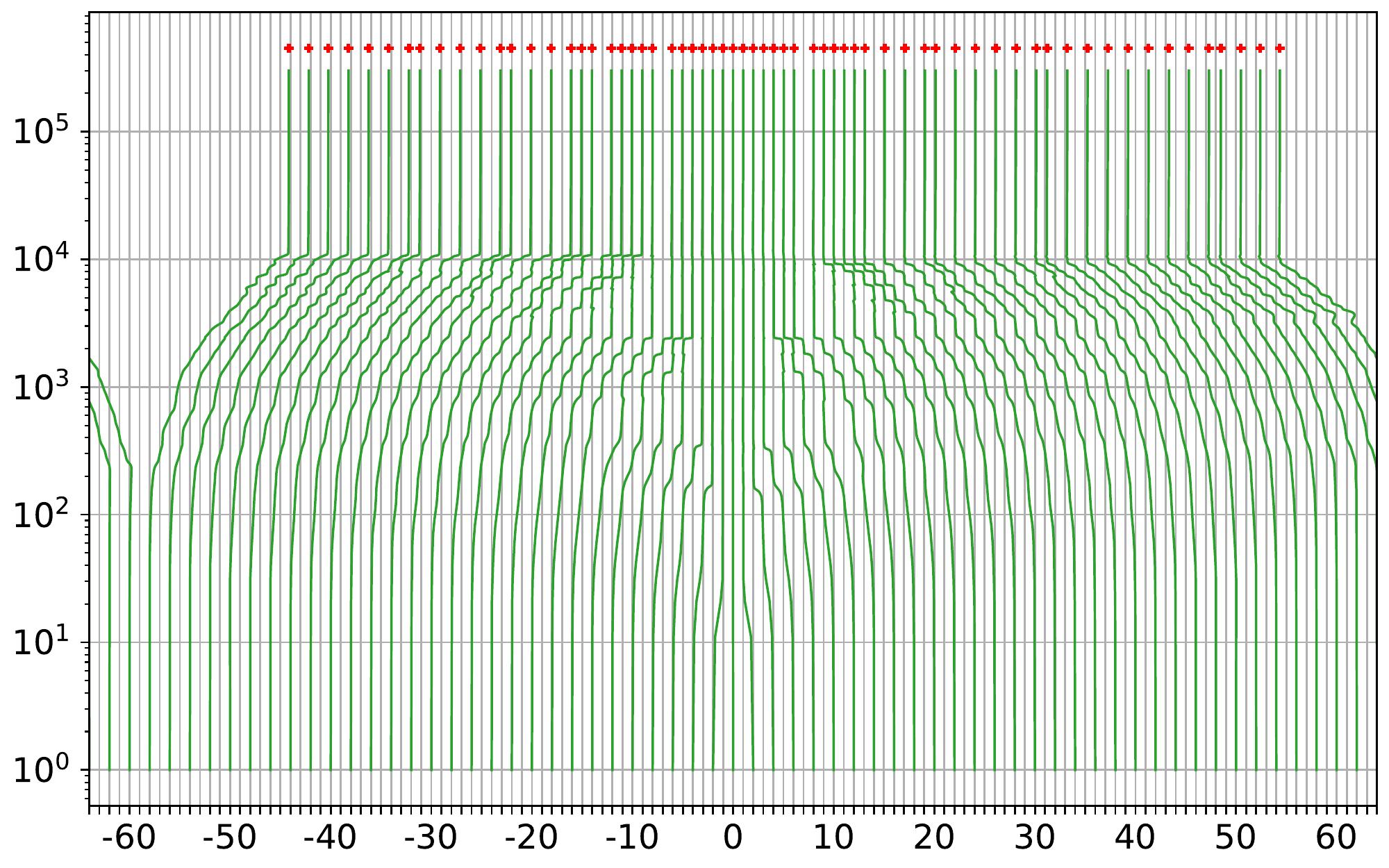}
	\caption{$P=1000$ -- variable metric gradient descent}
\end{subfigure}
\begin{subfigure}[t]{0.325\textwidth}
	\includegraphics[width=\linewidth]{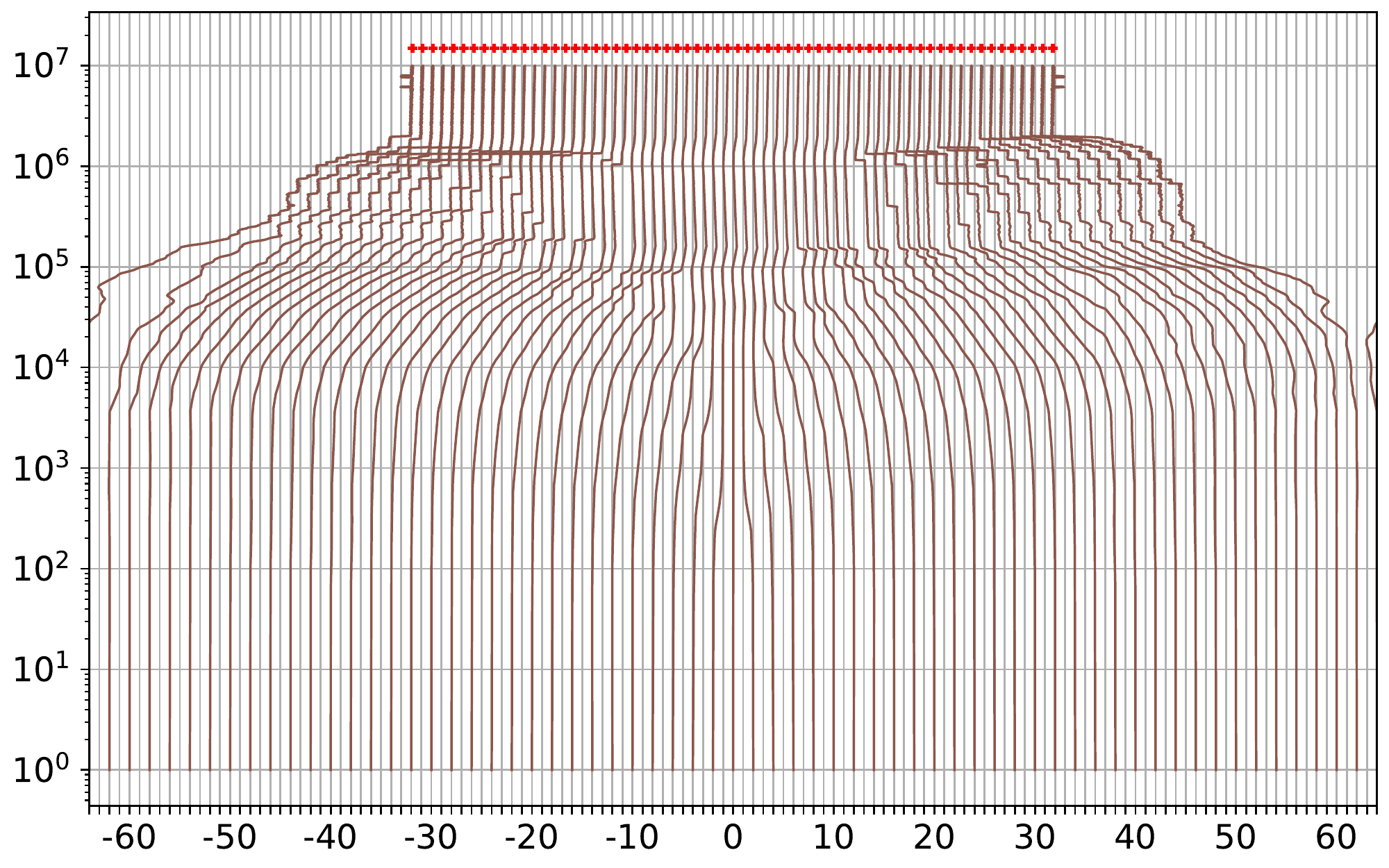}
	\caption{variable metric SGD}
\end{subfigure}
\caption{Trajectories of $\Xi$ the back-projection reconstructor $J_1$ and a fixed-step gradient descent. The iterations are represented on the vertical axis, and the horizontal axis corresponds to $\xi$ and is periodic. The initialization is a uniform subgrid and is seen on the axis $y=0$ of the top and middle figures. Left and center: trajectories of $\Xi$ for different sizes of signals families. The objective function is given in Fig.~\ref{fig:sto}. The right column represents trajectories of $\Xi$ using a stochastic gradient descent with one signal in the batch that is different at each iteration. The trajectories in the stochastic case have been averaged over the last 10000 iterations.}
\label{fig:batch_size}
\end{figure*}

\begin{figure}[ht]
\centering
\includegraphics[width=\linewidth]{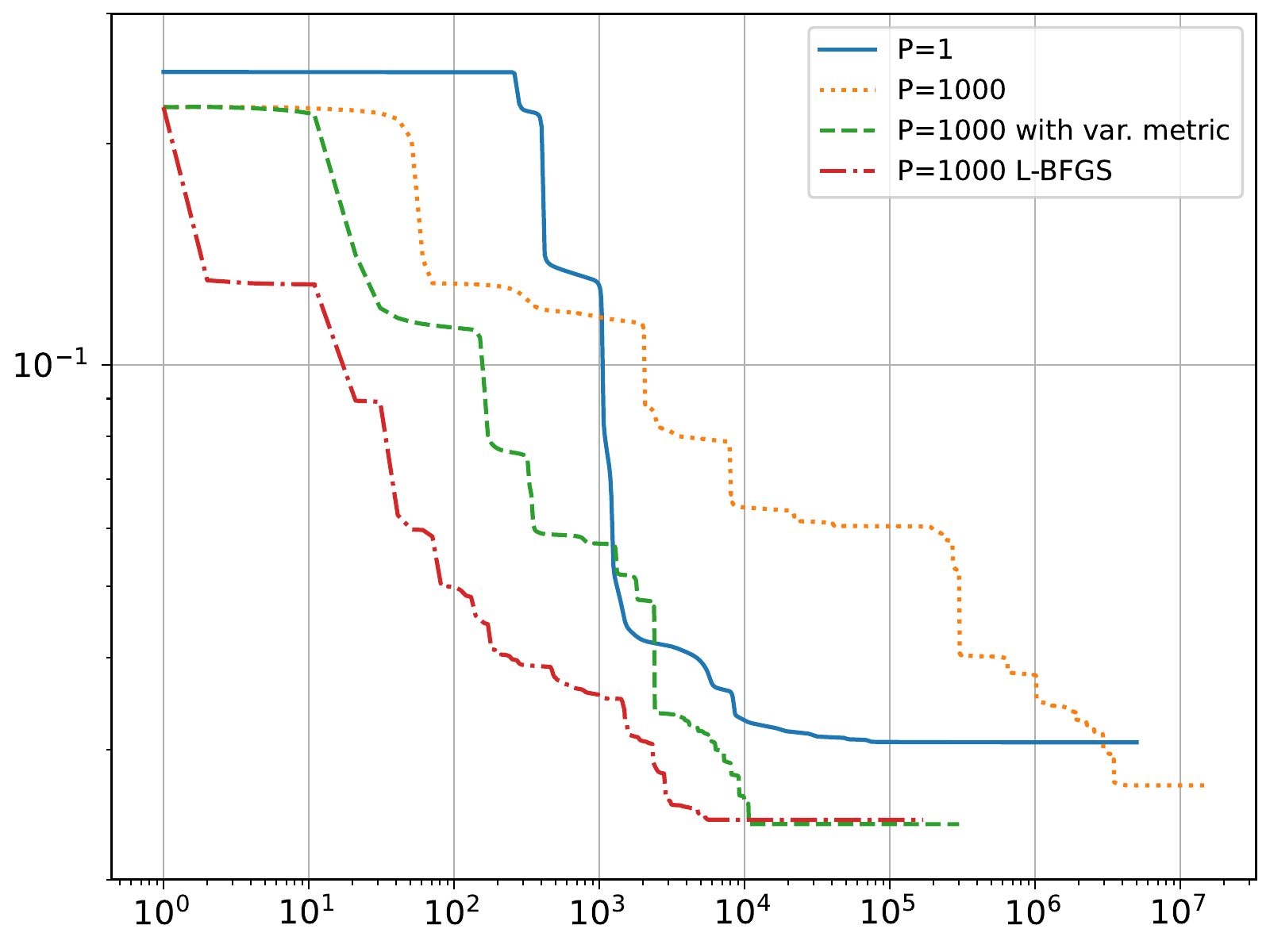}
\caption{Objective function $J_1$ (back-projection) for the different experiments in Fig.~\ref{fig:batch_size} in the deterministic case.}
\label{fig:sto}
\end{figure}

\begin{table*}[h]
\centering
\begin{tabular}{ |p{1.5cm}||p{2cm}|p{2cm}|p{2.5cm}|p{2cm}|p{2cm}|p{2cm}| }
\hline
Test case & $P=1$ & $P=1000$ & $P=1000$ with var. metric & L-BFGS & SGD & SGD with var. metric \\
\hline
Eff. & $9.07\times 10^{-2}$ & $2.68\times 10^{-2}$ & $2.38\times 10^{-2}$ & $2.41\times 10^{-2}$ & $6.63\times 10^{-2}$ & $1.00\times 10^{-2}$ \\
\hline
\end{tabular}
\caption{Effectiveness of the sampling schemes obtained with different strategies on a dataset of $1000$ signals. The table contains the average reconstruction error $J_1$ over the dataset. This dataset is the one used in the case $P=1000$, see Fig.~\ref{fig:batch_size} center column.}
\label{tab:optim_1D_cost}
\end{table*}

\section{Conclusion}

We highlighted two obstacles to the convergence of gradient based algorithms for Fourier sampling schemes optimization.
The first one is a high number of local minimizers and the second one is a vanishing gradient phenomenon for high frequencies.
As far as we know, this is the first theoretical study explaining why optimizing sampling patterns with modern automatic differentiation tools might result in algorithms being stucked at unsatisfactory locations. We also proposed three approaches to mitigate these effects.
First, the number of spurious minimizers, the width and the depth of their basins of attraction can be reduced by considering large databases of signals.
This acts as a regularization by averaging.
Second, the vanishing gradient effect can be attacked with variable metric gradient descents.
Finally, the use of a stochastic gradient instead of a deterministic gradient approach seems to allow escaping the narrow basins in a simplified 1D setting.
These remarks may help explaining why the recent approaches in the literature based on the Adam optimizer manage to slightly improve the sampling pattern efficiency. 
Our work suggests that increasing the database sizes may help further easing the numerical resolution of the sampling pattern optimization by further smoothing the energy profiles.
Many state-of-the-art reconstructors are based on a quadratic data fidelity term and we expect that some of the techniques used in this paper in the linear case can be reused even in a nonlinear setting. This is left for future research.

\section*{Acknowledgment}

The authors acknowledge a support from ANR JCJC Optimization on Measures Spaces, ANR-17-CE23-0013-01.
P.W and F.G. acknowledge the ANR-3IA Artifical and Natural Intelligence Toulouse Institute.
This work was performed using HPC resources from GENCI-IDRIS (Grant 2021-AD011012210R1).

\section{Proofs}

\subsection{Proof of Proposition~\ref{prop:bnd_q_rr}}\label{sec:proof_bnd_q_rr}

Significant progress have been made lately in the control of the extreme eigenvalues of Vandermonde matrices, which play a pivotal role in algebraic number theory \cite{bombieri1984large,moitra2015super,batenkov2020conditioning,aubel2019vandermonde}. The tightest results for well separated schemes was recently obtained in \cite{aubel2019vandermonde}. Rewriting their result in our formalism, we obtain the following inequality.

\begin{proposition}[Conditioning of Vandermonde matrices \cite{aubel2019vandermonde} \label{prop:conditioning}]
    Let $\Xi=(\xi_1,\hdots, \xi_M)$ denote a set of distinct sampling points.
    The following inequalities hold
    \begin{equation}
        \left( 1 - \frac{1}{\md(\Xi)}\right) \Id  \preccurlyeq A(\Xi)^*A(\Xi) \preccurlyeq \left( 1 + \frac{1}{\md(\Xi)}\right) \Id \label{eq:ineq_spectral_L} \\
    \end{equation}
\end{proposition}
\begin{proof}
    Relation \eqref{eq:ineq_spectral_L} is a direct consequence of \cite[eq. (31)]{aubel2019vandermonde} up to renormalizations.
\end{proof}

\begin{proof}
    For $i=1$, recall that $Q_1=\Id$ and $R_1^*R_1-\Id=A^*A$.
    Then \eqref{eq:bnd_Q} is trivial and \eqref{eq:bnd_RR} follows from Proposition~\ref{prop:conditioning}.
    
    Let $\tau_m$ denote the eigenvalues of $A(\Xi)^*A(\Xi)$.
    By Proposition~\ref{prop:conditioning}, $|\tau_m-1|\leq\epsilon<1$.

    For $i=2$, $R_2^*R_2=Q_2^*A^*AQ_2 = \left(A^*A\right)^+$, $Q_2=\left(A^*A\right)^+$. With $\epsilon<1$, $A(\Xi)^*A(\Xi)$ is invertible and we have
    \begin{equation}
        \frac{1}{1+\epsilon}  \Id  \preccurlyeq \left(A(\Xi)^*A(\Xi)\right)^{-1} \preccurlyeq \frac{1}{ 1-\epsilon} \Id.
    \end{equation}
    And we finally get
    \begin{equation}
        \frac{-\epsilon}{1+\epsilon}  \Id  \preccurlyeq \left(A(\Xi)^*A(\Xi)\right)^{-1}-\Id \preccurlyeq \frac{\epsilon}{ 1-\epsilon} \Id.
    \end{equation}

    For $i=3$, $Q_3=(1+\lambda)\left(A^*A+\lambda\Id\right)^{-1}$, so that
    \begin{equation*}
        \frac{1+\lambda}{1+\lambda+\epsilon}\Id \preccurlyeq Q_3 \preccurlyeq \frac{1+\lambda}{1+\lambda-\epsilon}\Id.
    \end{equation*}
    This gives
    \begin{equation*}
    \frac {-\epsilon}{1+\lambda+\epsilon}\Id \quad \preccurlyeq Q_3-\Id \preccurlyeq \frac \epsilon{1+\lambda-\epsilon}\Id
    \end{equation*}
    and using that $\epsilon>0$ and $\lambda\geq 0$ allows to conclude.

    In order to prove \eqref{eq:bnd_RR}, Proposition~\ref{prop:conditioning} yields $1-\epsilon\leq \tau_m \leq 1+\epsilon$. In addition, 
    \begin{equation*}
    R_3^*R_3 = (1+\lambda)^2 (A^*A+\lambda \Id)^{-1}A^*A(A^*A+\lambda \Id)^{-1}
    \end{equation*}
    can be diagonalized and its eigenvalues are therefore of the form
    $$
    (1+\lambda)^2 \frac{\tau_m}{(\tau_m+\lambda)^2}.
    $$ 
    By taking the upper-bound on the numerator and the lower-bound on the denominator, we obtain the following bound:
    \begin{equation*}
             R_3^*R_3 \preccurlyeq (1+\lambda)^2 \frac{1+\epsilon}{(1-\epsilon + \lambda)^2}\Id.
    \end{equation*}
    We can continue as follows:
    \begin{align*}
        &(1+\lambda)^2 \frac{1+\epsilon}{(1-\epsilon + \lambda)^2} \\
        &= (1+\lambda)^2 \frac{1+\epsilon}{(1+\lambda)^2 (1-\epsilon/(1+ \lambda))^2} \\
        & \leq \frac{1+\epsilon}{(1-\epsilon)^2}.
    \end{align*}
    By a similar reasoning with respect to the smallest eigenvalue of $R_3^*R_3$, we get:
    \begin{equation*}
        \frac{1-\epsilon}{(1+\epsilon)^2} \Id \preccurlyeq R_3^*R_3 \preccurlyeq \frac{1+\epsilon}{(1-\epsilon)^2} \Id. 
    \end{equation*}
    Substracting the identity on both sides and using the fact that $\epsilon^2<\epsilon$ since $\epsilon<1$ yields
    \begin{equation*}
        -\frac{4\epsilon}{(1-\epsilon)^2}\Id \preccurlyeq R_3^*R_3 - \Id \preccurlyeq \frac{4\epsilon}{(1-\epsilon)^2} \Id.
    \end{equation*}
\end{proof}

\subsection{Proof of Theorem~\ref{thm:main1}}\label{sec:proof_thm_main1_gen}

Under the hypotheses of Theorem~\ref{thm:main1}, first notice that any set $\Xi \in Z^M$ is a local maximizer of $\Xi\mapsto \|\hat u(\Xi)\|_2^2$.
Indeed any perturbation of the individual sampling locations $\xi_m$ results in a decay of the captured energy.

There are $\binom{K}{M}$ possible sampling configurations when all the points belong to $Z$.
Let $\bar \Xi=\{\bar \xi_1, \hdots, \bar \xi_M \}$ denote one of them.
The idea of the proof is to show that there is a local minimizer of $J$ in the following neighborhood $B=[\bar \xi_1-r,\bar \xi_1+r]\times \hdots \times [\bar \xi_M-r, \bar \xi_M+r]$.
A sufficient condition for the set $B$ to contain a local minimizer of $J$ is that $J(\bar \Xi) < J(\Xi)$ for all $\Xi \in \partial B$ (the boundary of $B$) since $J$ is continuous.
Throughout this proof, we use $\epsilon=\frac 1{\delta-2r}$ since we always have for all $\Xi$ under consideration, $\md(\Xi)\geq \delta-2r$.

Using the bounds of Proposition~\ref{prop:bnd_q_rr}, we obtain
\begin{equation}
    \left| J(\Xi)-\tilde J(\Xi) \right| \leq \left(\frac{b}2+a\right)\|\hat u(\Xi)\|_2^2+\frac b 2 \sigma^2M. \label{eq:second_order_bnd}
\end{equation}

For all $\Xi\in \partial B$, at least one index $m$ must verify $\bar\xi_m-\xi_m=r$ and we have by strict concavity of $|\hat u|$ around $\bar \xi_m$
\begin{align}
    \|\hat u(\bar\Xi)\|_2^2 - \|\hat u(\Xi)\|_2^2 &\geq \frac{cr^2}2 \label{eq:equation_cr}\\
    \|\hat u(\bar\Xi)\|_2^2 + \|\hat u(\Xi)\|_2^2 &\leq 2\|\hat u(\bar\Xi)\|_2^2. \label{eq:equation_somme}
\end{align}
Hence for $\Xi \in \partial B$, using \eqref{eq:equation_cr} yields
\begin{equation}\label{eq:36}
\tilde J(\Xi)  - \tilde J(\bar \Xi)\geq \frac{cr^2}{2}
\end{equation}

Combining the previous inequalities yields
\begin{align*}
&J(\Xi)-J(\bar \Xi) \\ 
& = J(\Xi)- \tilde J(\Xi) + \tilde J(\Xi) - \tilde J(\bar \Xi) + \tilde J(\bar \Xi) - J(\bar \Xi) \\
& \stackrel{\eqref{eq:36}}{\geq} \frac{cr^2}{2} + J(\Xi)- \tilde J(\Xi) + \tilde J(\bar \Xi) - J(\bar \Xi) \\
& \stackrel{\eqref{eq:second_order_bnd}}{\geq} \frac{cr^2}{2} - \left[\left(\frac{b}{2}+a\right)\left(\|\hat u(\Xi)\|_2^2+\|\hat u(\bar \Xi)\|_2^2\right)+b M\sigma^2 \right] \\
& \stackrel{\eqref{eq:equation_somme}}{\geq} \frac{cr^2}{2} - \left[\left(b+2a\right)\|\hat u(\bar \Xi)\|_2^2+b M\sigma^2 \right]
\end{align*}

Therefore, the condition
\begin{align}
  \frac{cr^2}{2} > \left(b+2a\right)\|\hat u(\bar \Xi)\|_2^2+b M\sigma^2 
\end{align}
suffices to conclude on the existence of a maximizer of $J$ in the interior of $B$.
The multiplicative factor $M\,!$ is related to the fact that for a given maximizer, all the possible permutations of indices give rise to different maximizers.

\subsection{Proof of Proposition~\ref{prop:expression_gradient}}\label{sec:proof_expression_gradient}

\begin{proof}
Let us consider a point configuration $\Xi\in \R^{M}$ and a perturbation $\epsilon\in \R^M$.
Given a vector of measurements $\hat u(\Xi)\in\C^M$, we let $\nabla \hat u(\Xi)=\begin{pmatrix}\hat u'(\xi_1)\\ \vdots\\ \hat u'(\xi_M)\end{pmatrix}$ denote the vector of derivatives at the sampling locations. 
Elementary calculus leads to the following identities for every $\epsilon$ direction of variation:
\begin{align*}
(\Jac_{A}(\Xi)\epsilon)^* &= \Jac_{A^*}(\Xi)\epsilon \\
\nabla \hat u(\Xi)\odot\epsilon &= \Jac_{A^*}(\Xi)\epsilon\cdot u.
\end{align*}

Then, we apply standard calculus of variations:
\begin{align*}
J_1(\Xi+\epsilon) =& J_1(\Xi) + \Re\langle \Jac_A(\Xi)\epsilon \cdot\hat u(\Xi), r(\Xi)\rangle \\
&\hspace{-0.5cm}+ \Re\langle A(\Xi)\Jac_{A^*}(\Xi)\epsilon\cdot u, r(\Xi)\rangle + o(\|\epsilon\|_2^2) \\
=& J_1(\Xi) + \Re\langle \hat u(\Xi), \left(\Jac_A(\Xi)\epsilon\right)^*r(\Xi)\rangle \\
&+ \Re\langle \nabla \hat u(\Xi)\odot\epsilon, \hat r(\Xi) \rangle + o(\|\epsilon\|_2^2) \\
=& J_1(\Xi) + \Re\langle \hat u(\Xi), \nabla \hat r(\Xi)\odot\epsilon\rangle \\
&+ \Re\langle \epsilon, \overline{\nabla \hat u(\Xi)}\odot\hat r(\Xi) \rangle + o(\|\epsilon\|_2^2) \\
=& J_1(\Xi) + \Re\langle \overline{\nabla \hat r(\Xi)}\odot\hat u(\Xi), \epsilon\rangle \\
&+ \Re\langle \epsilon, \overline{\nabla \hat u(\Xi)}\odot\hat r(\Xi) \rangle + o(\|\epsilon\|_2^2).
\end{align*}
Hence, by identification
\begin{align*}
\nabla J_1(\Xi) &= \Re\left(\overline{\nabla \hat r(\Xi)}\odot \hat u(\Xi) + \nabla \hat u(\Xi)\odot \overline{\hat r(\Xi)}\right) \\
&= \Re\left(\nabla\left(\hat u(\Xi)\odot\overline{\hat r(\Xi)}\right)\right).
\end{align*}
\end{proof}

\subsection{Proof of Theorem~\ref{thm:gradient_flatness}}\label{sec:proof_thm_flatness}

In order to simplify the notation, let $L(\Xi)\eqdef A(\Xi)^*A(\Xi)$.
By Proposition~\ref{prop:expression_gradient}, we have 
\begin{equation*}
\left|\frac{\partial J_1(\Xi)}{\partial\xi_m}\right| \leq |\hat u'(\xi_m)|\cdot |\hat r(\xi_m)| + |\hat u(\xi_m)|\cdot |\hat r'(\xi_m)|.
\end{equation*}
By definition, we have $\hat r(\Xi)=(L(\Xi)-\Id)\hat u (\Xi)$, hence
\begin{equation}
|\hat r(\xi_m)|\leq \|\hat r(\Xi)\|_2 \leq \frac{\|\hat u(\Xi)\|_2}{\md(\Xi)},
\end{equation}
where we used Proposition~\ref{prop:conditioning} to obtain the last inequality. Now, we also wish to control $|\hat r'(\xi_m)|$.
To this end, first notice that
\begin{align*}
\hat r'(\xi_{m}) =& \sum_{m'=1}^M \left(\frac{\partial L(\Xi)_{m,m'}}{\partial\xi_m}\hat u(\xi_{m'}) \right. \\
& \left. \vphantom{ \frac{\partial L(\Xi)_{m,m'}}{\partial\xi_m} } + L(\Xi)_{m,m'}\hat u'(\xi_{m'})\ind_{m=m'}\right) - \hat u'(\xi_m) \\
=& \sum_{m'=1}^M \frac{\partial L(\Xi)_{m,m'}}{\partial \xi_{m}} \hat u(\xi_{m'}).
\end{align*}

We start with an analytical expression of the matrix $L(\Xi)$.
\begin{proposition}[The expression of $A^*A$]
Let $L(\Xi)\eqdef A(\Xi)^*A(\Xi)$. We have
\begin{equation}
[L(\Xi)]_{m,m'} = \begin{dcases}
1 & \textrm{ if } m=m', \\
\substack{ \frac{1}{N}\exp\left(\frac{\iota\pi(\xi_m-\xi_{m'})}{N}\right) \\ \times\frac{\sin(\pi (\xi_m-\xi_{m'}))}{\sin\left(\frac{\pi (\xi_m-\xi_{m'})}{N} \right)} } & \textrm{otherwise}.
\end{dcases}
\end{equation}
\end{proposition}
\begin{proof}
We have:
\begin{align*}
\left[L(\Xi)\right]_{m,m'} =& \frac 1 N\sum_n e^{2\iota\frac\pi N\langle\xi_{m'}-\xi_{m}, n\rangle} \\
=& \frac 1 N e^{-\iota\pi(\xi_{m'}-\xi_{m})} \frac{1-e^{2\iota\pi(\xi_{m'}-\xi_{m})}}{1-e^{2\iota\frac\pi N(\xi_{m'}-\xi_{m})}} \\
=& \frac 1 N e^{-\iota\pi(\xi_{m'}-\xi_{m})} \frac{e^{\iota\pi(\xi_{m'}-\xi_{m})}}{e^{\iota\frac\pi N(\xi_{m'}-\xi_{m})}} \\
&\times \frac{e^{-\iota\pi(\xi_{m'}-\xi_{m})}-e^{\iota\pi(\xi_{m'}-\xi_{m})}}{e^{-\iota\frac\pi N(\xi_{m'}-\xi_{m})}-e^{\iota\frac\pi N(\xi_{m'}-\xi_{m})}} \\
=& \frac 1 N e^{-\iota\frac\pi N(\xi_{m'}-\xi_{m})} \frac{\sin(\pi(\xi_{m'}-\xi_{m}))}{\sin(\frac\pi N(\xi_{m'}-\xi_{m}))}.
\end{align*}
\end{proof}

Now, we will use the following lemma. 
\begin{lemma}\label{lemma:upper_bond_partial_L}
The following bound holds:
\begin{equation*}
\left|\frac{\partial L(\Xi)_{m,m'}}{\partial\xi_{m}}\right| \leq \frac\pi{N}+\frac{4}{\dist(\xi_{m'},\xi_m)}\leq \frac\pi{N}+\frac{4}{\md(\Xi)}.
\end{equation*}
\end{lemma}
\begin{proof}
Letting $\delta=\xi_m-\xi_{m'}$, we have
\begin{align*}
&\frac{\partial L(\Xi)_{m,m'}}{\partial\xi_m} = \frac\pi{N^2}\times \frac{\iota e^{\iota\frac\pi N\delta}\sin(\pi\delta)}{\sin\left(\frac\pi N\delta\right)} \\
&+ \frac{\pi}{N}\times \frac{e^{-\iota\frac\pi N\delta}}{\sin(\frac\pi N\delta)}\left( \cos(\pi\delta)-\frac{\sin(\pi\delta)} N\times\frac{\cos(\frac \pi N\delta)}{\sin(\frac\pi N\delta)} \right) .
\end{align*}

Without loss of generality we consider the case $0\leq \delta \leq N/2$. Using $\left|\frac{\sin(\pi\delta)}{N\sin(\frac\pi N\delta)}\right| \leq 1$ let us remark that
\begin{align*}
&\left|\frac{\partial L(\Xi)_{m,m'}}{\partial\xi_m}\right| \leq \frac\pi{N} \\
&\quad + \frac\pi N\left|\frac{1}{\sin(\frac\pi N\delta)}\left(\frac{\sin(\pi\delta)\cos(\frac\pi N\delta)}{N\sin(\frac\pi N\delta)}-\cos(\pi\delta)\right)\right|.
\end{align*}

Using the inequality $\left|\frac{\sin(\pi\delta)}{N\sin(\frac\pi N\delta)}\right| \leq 1$ again, we obtain
\begin{equation*}
\left|\frac{\sin(\pi\delta)\cos(\frac\pi N\delta)}{N\sin(\frac\pi N\delta)}-\cos(\pi\delta)\right| \leq \left|\cos(\frac\pi N\delta)\right|+1 \leq 2.
\end{equation*}
Finally, using the inequality $\sin(x)\geq x/2$ for $x\in (0,\pi/2)$, we get
$\left|\frac{\partial L(\Xi)_{m',m}}{\partial\xi_{m'}}\right| \leq \frac{\pi}{N}+\frac{4}{\delta}$.
\end{proof}

Lemma~\ref{lemma:upper_bond_partial_L} and a Cauchy-Schwarz inequality provides the following bound:
\begin{equation*}
|\hat r'(\xi_m)| \leq \left( \frac{\pi}{N} + \frac{4}{\md(\Xi)}\right) \|\hat u(\Xi)\|_1.
\end{equation*}
Combining everything finally yields:
\begin{align*}
\left|\frac{\partial J_1(\Xi)}{\partial\xi_m}\right| \leq& |\hat u'(\xi_m)| \cdot \frac{\|\hat u(\Xi)\|_2}{\md(\Xi)} \\
&+ |\hat u(\xi_m)| \cdot \|\hat u(\Xi)\|_1 \cdot \left(\frac{\pi}{N}+\frac{4}{\md(\Xi)}\right).
\end{align*}

Under the decay assumptions of Theorem~\ref{thm:gradient_flatness}, we obtain
\begin{equation*}
\left|\frac{\partial J_1(\Xi)}{\partial\xi_m}\right| \lesssim \frac{\|\hat u(\Xi)\|_1}{\md(\Xi) |\xi_m|^\alpha}.
\end{equation*}

{\small
\bibliographystyle{plain}
\bibliography{biblio}
}

\end{document}